\documentclass[11pt]{article}
\usepackage[a4paper, left=0.5in, right=0.5in, top=1in, bottom=1in]{geometry}

\usepackage[utf8]{inputenc}
\usepackage[T1]{fontenc}
\usepackage{lmodern} 

\usepackage{amsmath, amssymb, amsthm, mathtools, natbib}
\usepackage{bm} 
\usepackage{bbm} 
\usepackage{dsfont} 
\numberwithin{equation}{section}

\allowdisplaybreaks[4]

\newtheoremstyle{mystyle}
    {}
    {}
    {\itshape}
    {}
    {\bfseries}
    {.}
    {5pt}
    {\thmname{#1}\thmnumber{ #2}\thmnote{ #3}}
    
\theoremstyle{mystyle}    
\newtheorem{theorem}{Theorem}[section]
\newtheorem{lemma}[theorem]{Lemma}
\newtheorem{proposition}[theorem]{Proposition}

\theoremstyle{definition}
\newtheorem{definition}[theorem]{Definition} 
\newtheorem{remark}[theorem]{Remark}
\newtheorem{assump}[theorem]{Assumption}

\usepackage{enumitem}

\usepackage{graphicx}
\usepackage{caption}
\usepackage{subcaption}
\usepackage{booktabs} 
\usepackage{float} 

\usepackage{hyperref}
\hypersetup{
    colorlinks=true,
    linkcolor=blue,
    citecolor=blue,
    urlcolor=blue
}

\usepackage{cleveref}
\crefname{enumi}{Assumption}{Assumptions}

\usepackage[dvipsnames]{xcolor}

\newcommand{\R}{\mathbb{R}}
\newcommand{\E}{\mathbb{E}}

\newcommand{\1}{\mathds{1}}

\newcommand{\bs}[1]{\boldsymbol{#1}}

\usepackage{changepage}

\title{Quantitative convergence rates for extended mean field games 
with volatility control}

\author{Erhan Bayraktar%
  \thanks{Department of Mathematics, University of Michigan, United States. E-mail address: erhan@umich.edu. This author is partially supported by the National Science Foundation under grant DMS-2507940 and by the Susan M. Smith chair.}
  \and 
  Hiroaki Horikawa%
  \thanks{Department of Mathematics, University of Michigan, United States. E-mail address: hhiroaki@umich.edu.}
}

\date{}

\begin{document}
\maketitle

\begin{adjustwidth}{-0.75cm}{-0.75cm} 
\begin{abstract}
We investigate the convergence of symmetric stochastic differential games with interactions via control, where the volatility terms of both idiosyncratic and common noises are controlled.
We apply the stochastic maximum principle, following the approach of Lauri\`{e}re and Tangpi, to reduce the convergence analysis to the study of forward-backward propagation of chaos.
Under the standard monotonicity conditions, we derive quantitative convergence rates for open-loop Nash equilibria of $N$-player stochastic differential games toward the corresponding mean field equilibrium.
As a prerequisite, we also establish the well-posedness of the conditional McKean--Vlasov forward-backward stochastic differential equations by the method of continuation.
Moreover, we analyze a specific class of linear-quadratic settings to demonstrate the applicability of our main result.
\end{abstract}
\end{adjustwidth}

\tableofcontents

\section{Introduction}
We investigate the convergence problem of a class of symmetric $N$-player stochastic differential games with interactions through both states and controls. Let us first outline our setup. The precise formulation is provided in Section~\ref{sec:main_theorem}.

For each player $i \in \{1, \ldots, N\}$, the controlled state dynamics are given by:
\begin{align*}
    dX_t^{i,\bs{\alpha}} &= b_t\big(X_t^{i,\bs{\alpha}}, \alpha_t^i, L^N(\bs{X}^{\bs{\alpha}}_t, \bs{\alpha}_t)\big)dt  
    + \sigma_t\big(X_t^{i,\bs{\alpha}}, \alpha_t^i, L^N(\bs{X}^{\bs{\alpha}}_t, \bs{\alpha}_t)\big)dW_t^i 
    + \sigma^0_t\big(X_t^{i,\bs{\alpha}}, \alpha_t^i, L^N(\bs{X}^{\bs{\alpha}}_t, \bs{\alpha}_t)\big)dW_t^0.
\end{align*}
Here, $\alpha^i$ denotes the control of player $i$, and $L^N(\bs{X}^{\bs{\alpha}}_t, \bs{\alpha}_t) := \frac{1}{N} \sum^N_{i=1} \delta_{(X^{i,\bs{\alpha}}_t, \alpha^i_t) }$ represents the empirical distribution of the state-control configuration of all players at time $t$.
Given the strategies of the other players $\bs{\alpha}^{-i}$, each player $i$ aims to minimize the cost functional defined by:
\begin{align*}
J(\alpha^i; \bs{\alpha}^{-i}) :=\E  \left[
\int^T_0 f_t( X^{i,\bs{\alpha}}_t, \alpha^i_t, L^N(\bs{X}_t^{\bs{\alpha}}, \bs{\alpha}_t)) \, dt 
+ g(X^{i,\bs{\alpha}}_T,  L^N(\bs{X}^{\bs{\alpha}}_T))
\right].
\end{align*}
A strategy profile $\bs{\alpha} = (\alpha^1, \ldots, \alpha^N)$ is called an $N$-player Nash equilibrium (NE) if, for all $i \in \{1, \ldots, N\}$,
\begin{equation*}
    J(\alpha^i; \bs{\alpha}^{-i}) = \inf_{\beta \in \mathbb{A}_N} J(\beta; \bs{\alpha}^{-i}).
\end{equation*}
In a NE, each player selects the best possible strategy given the strategies of the other players.

Mean field games (MFGs), introduced by Lasry and Lions \cite{lasry2007mean}, and Huang et al.\ \cite{huang2006large}, are formulated as the limiting problem of symmetric stochastic differential games as the population size $N$ tends to infinity. Intuitively, under this approximation, the influence of any single agent diminishes as the population grows. The system thus reduces to the control problem of a representative agent whose state dynamics depend on a given flow of probability measures.
Specifically, in our setting, the state process of the representative agent is given by
\begin{align*}
    d X_t^{\alpha,\xi} = b_t\big(X_t^{\alpha,\xi}, \alpha_t, \xi_t\big)\, d t 
    + \sigma_t\big(X_t^{\alpha,\xi}, \alpha_t, \xi_t\big)\, d W_t 
    + \sigma^0_t\big(X_t^{\alpha,\xi}, \alpha_t, \xi_t\big)\, d W_t^0,
\end{align*}
for a given stochastic flow of probability measures $(\xi_t)_{t\in[0,T]}$.
The representative agent aims to find a strategy $\alpha$ that minimizes
\begin{align*}
    J^{\xi}(\alpha) \coloneqq \E \bigg[
    \int^T_0 f_t\big(X^{\alpha,\xi}_t, \alpha_t, \xi_t\big) \, d t 
    + g\big(X^{\alpha,\xi}_T, \mu_T\big)
    \bigg],
\end{align*}
where $\mu_T$ denotes the first marginal of $\xi_T$.
A pair $(\hat{\xi}, \hat{\alpha})$ is called a mean field equilibrium (MFE) if it satisfies both the optimality and the consistency conditions:
\begin{align*}
    J^{\hat{\xi}}(\hat{\alpha} ) = \inf_{\beta \in \mathbb{A}} J^{\hat{\xi}}(\beta)
    \quad \text{and} \quad
    \hat{\xi}_t = \mathcal{L}^1\big(X_t^{\hat{\alpha},\hat{\xi}},\hat{\alpha}_t\big),
\end{align*}
where $\mathcal{L}^1\big(X_t^{\hat{\alpha},\hat{\xi}},\hat{\alpha}_t\big)$ is the joint law of $(X_t^{\hat{\alpha},\hat{\xi}},\hat{\alpha}_t\big)$, conditioned on the realization of the common noise.
In contrast to \textit{classical} MFGs, which consider interactions solely through the state, our setting is referred to as \textit{extended} MFGs \cite{gomes2014existence} (or MFGs of controls).

In this paper, we address the convergence problem of the mean field approximation; specifically, the question of whether and how a sequence of NE $(\bs{\alpha}^N)_{N\in \mathbb{N}}$ for $N$-player games converges to the optimal control in the MFE. 
In our main theorem (Theorem~\ref{thm:main}), we establish the convergence rate for the squared error between the NE and the MFE:
\begin{align*}
\frac{1}{N} \sum^N_{i=1}\E \bigg[ \int^T_0 \big|\alpha_t^{N,i} -  \alpha_t^{i}\big|^2 \, dt \bigg] &\leq \frac{C}{N},
\end{align*}
where $\alpha^{N,i}$ denotes the strategy of player $i$ in the NE, and $\alpha^{i}$ represents the conditionally independent and identically distributed (i.i.d.) copies of the corresponding optimal strategy in the MFE (for the precise definition of these copies, see the end of Section~\ref{sec:main_theorem}). 

To achieve this result, we adopt the forward-backward propagation of chaos technique, pioneered by Lauri\`{e}re and Tangpi \cite{lauriere2022convergence}, and further developed in \cite{bayraktar2023propagation, jackson2024quantitative, possamai2025non}. Our strategy proceeds as follows:
\begin{enumerate}[label=(\roman*)]
    \item characterize the NE via classical forward-backward stochastic differential equations (FBSDEs) using the
    stochastic maximum principle (SMP);
    \item characterize the MFE via the conditional McKean--Vlasov FBSDEs (MKV-FBSDEs) similarly using the SMP; and 
    \item apply forward-backward propagation of chaos techniques to derive a quantitative convergence rate between the solutions of these two systems.
\end{enumerate}

Although extended MFGs with volatility control hold significant appeal for applications, particularly in finance and economics, their analysis is often considered highly intractable.
For instance, one major issue is the irregularity of the stochastic flow of conditional distributions $(\mathcal{L}^1 (X_t, \alpha_t))_{t \in [0,T]}$.

To overcome the difficulties, we leverage a type of monotonicity condition  (see Assumption~\ref{assump:displacement_monotonicity}), which was considered in \cite{bensoussan2015well, hu1995solution, peng1999fully} for the well-posedness of fully coupled (MKV-)FBSDEs and successfully applied by Jackson and Tangpi~\cite{jackson2024quantitative} to the convergence problem of classical MFGs.
Under this condition, we establish the well-posedness result of the conditional MKV-FBSDEs and consequently, the unique existence of the MFE. Furthermore, by imposing the additional smoothness condition on the coefficients of the game with respect to the measure argument (Assumption~\ref{assump:smoothness_L-derivative}), following \cite{jackson2024quantitative}, we derive the stated $O(N^{-1})$ convergence rate of the equilibria.

Moreover, to demonstrate the applicability of our main result, we examine a specific class of linear-quadratic (LQ) MFGs.
This setting appears novel compared to the existing literature, such as \cite{bensoussan2025linear, bensoussan2016linear, graber2016linear}.
To maintain our primary focus, we do not derive explicit feedback controls as in the literature of LQ-MFGs, but restrict our discussion to the well-posedness of the game and the convergence results.

\subsection{Literature review}
Introduced by Lasry and Lions \cite{lasry2007mean}, and Huang et al.\ \cite{huang2006large}, MFGs have become a powerful framework for modeling large systems of interacting agents. This framework considers a system of a large number of homogeneous agents who optimize their individual strategies while interacting weakly with the rest of the population --- that is, each player responds to the empirical distribution of the entire population. The theory analyzes the limiting behavior as the population size $N \to \infty$.  As a characterization of the MFE, a coupled system of the Hamilton--Jacobi--Bellman (HJB) equation and Fokker--Planck equation was introduced in \cite{lasry2007mean}. The former characterizes the optimality of the representative agent, while the latter describes the flow of the distribution of the player's state. 

Since the advent of the theory, the first direction of the convergence problem has been actively investigated: specifically, how the optimal strategy obtained in an MFE can form an $\varepsilon$-Nash equilibrium for an $N$-player game. However, the other direction---the naturally arising question of how and in what sense NE converge to the corresponding MFE---remained open for some time during the first decade of research in MFGs. The first major breakthroughs were achieved by Fischer \cite{fischer2017connection} and Lacker \cite{lacker2016general}, who established the convergence of the empirical measures associated with the NE using compactness arguments. Lacker \cite{lacker2020convergence} subsequently generalized these results to closed-loop controls, and Djete \cite{djete2023mfg_of_controls} extended them to accommodate extended MFGs.

Diverging from the compactness approach, the seminal work by Cardaliaguet et al. \cite{cardaliaguet2019master} succeeded in quantifying the convergence rate for closed-loop controls. By characterizing the value functions in the MFE via the \textit{master} equation---which is, formally, the limit of the $N$-Nash system---they derived the convergence rate of the error between the master equation's solution and the value functions of the $N$-player game. This approach was further applied in \cite{bayraktar2021finitestate, bayraktar2022finitestate, bayraktar2018finitestate, bayraktar2024mean, delarue2019master_central, delarue2020master_large_deviaion}. Remarkably, all of these papers also addressed the case with (uncontrolled) common noise.

However, none of the aforementioned works regarding convergence rates address MFGs with interactions through controls. Notably, such interactions naturally arise in many problems in finance and economics. See, for instance, \cite{alasseur2020extended, cardaliaguet2018mean, chan2015bertrand, graber2016linear, graber2018existence} and \cite[Sections 1.3.2 and 1.4.4]{carmona2015weak} for specific applications. For general studies of extended MFGs, we refer to \cite{bensoussan2025linear, carmona2015weak, gomes2014existence, gomes2016extended} and \cite[Section~4.6]{carmona2018probabilistic_I}.

Lauri\`{e}re and Tangpi \cite{lauriere2022convergence} were the first to derive the convergence rates for these extended games.
They characterized the equilibrium via FBSDEs and generalized the propagation of chaos technique to coupled forward-backward systems. This technique originates from Sznitman \cite{sznitman2006topics} and was adapted to backward equations in \cite{lauriere2022backward}.
Possamaï and Tangpi \cite{possamai2025non} followed a similar path for the non-Markovian setting using the weak formulation \cite{carmona2015weak}. Instead of the SMP, they employed the dynamic programming principle to characterize the MFE via \textit{generalized} MKV-BSDEs and obtained convergence rates for closed-loop equilibria.
However, both approaches have a limitation. To establish convergence for an arbitrary time horizon $T > 0$, the drift coefficient $b$ must satisfy a \textit{dissipativity} condition with a sufficiently large constant $K_B$ (see \cite[Theorem~2.2]{lauriere2022convergence} and \cite[Theorem~2.10]{possamai2025non}). Furthermore, neither work covers the case of controlled volatility.

Recently, Jackson and Tangpi \cite{jackson2024quantitative} made significant progress on these convergence results. Building on the probabilistic framework of \cite{lauriere2022convergence}, they quantified convergence for arbitrary time horizons $T > 0$, covering settings with controlled volatility and common noise. Key to their approach was the displacement monotonicity condition. This allowed them to dispense with the smallness condition on $K_b$ required in \cite{lauriere2022convergence, possamai2025non}. Using this condition, they applied the method of continuation to guarantee the well-posedness of the conditional MKV-FBSDEs and, consequently, the unique existence of the MFE. However, their analysis remains restricted to classical MFGs.

Up to now, the study of MFGs with interaction through controls, and volatility control has been limited, with the notable exception of \cite{djete2023mfg_of_controls}.
Incorporating these features simultaneously increases the difficulty of the problem significantly.
In the analytic approaches, volatility control typically complicates the Hamiltonian of the HJB equations, often resulting in the loss of semilinearity. 
Furthermore, in the coupled PDE framework \cite{lasry2007mean}, the dependence on the control distribution strengthens the coupling as in \cite{graber2018existence}. 
In the FBSDE approach, a major challenge lies in the regularity of the flow of distributions $(\mathcal{L}^1 (X_t, \alpha_t))_{t \in [0,T]}$.  Without volatility control, this flow is usually continuous. This allows for the use of Schauder's fixed point argument on the space $C([0,T]; \mathcal{P}_2(\R^n \times \R^n))$ to guarantee the existence of MKV-FBSDEs, as studied in \cite[Section~4.6]{carmona2018probabilistic_I}.
The difficulty arising from the irregularity of the flow (or more precisely, the flow of the distribution of the randomized control) was also highlighted in \cite[Remark~2.5]{djete2022extended} specifically for mean field control (MFC) problems.
Djete \cite{djete2023mfg_of_controls} addressed this issue using \textit{measure-valued} control rules, inspired by the relaxed controls introduced in \cite{lacker2016general}, and obtained convergence results for extended MFGs with volatility control via compactness arguments under general assumptions. However, the volatility of the common noise was not controlled, nor was a rate of convergence provided.

In this paper, we establish the convergence rate for extended MFGs in which the volatilities of both idiosyncratic and common noises are controlled. Crucial to achieving this goal is the monotonicity condition (Assumption~\ref{assump:displacement_monotonicity}). Various monotonicity conditions have been widely studied in the context of the method of continuation to prove the well-posedness of (MKV)-FBSDEs, as seen in \cite{bensoussan2015well, hu1995solution, peng1999fully}, \cite[Chapter~8.4]{zhang2017backward}, and \cite[Appendix~A]{jackson2024quantitative}. As previously mentioned, Jackson and Tangpi \cite{jackson2024quantitative} investigated several types of (displacement) monotonicity conditions for the convergence problem of classical MFGs.

In our context, the primary advantage of the monotonicity condition—similar to the one in \cite[(A1)]{bensoussan2015well}—is that it circumvents the need for flow regularity. This enables proving the well-posedness of the MKV-FBSDEs via Banach's fixed point theorem on the space of square-integrable processes, avoiding reliance on Schauder's fixed point theorem on the space of measure flows, and thus we can establish the unique existence of the MFE for extended MFGs in a strong formulation.

Recall that our primary goal is establishing an estimate for the error
$|\alpha^{N,i} -  \alpha^{i}\big|^2$. 
In the $N$-player game, the optimal strategy $\alpha_t^{N,i}$ of the player $i$ depends on its own state, the empirical distribution of all the players' states, \textit{plus} an additional term involving the other players' actions (denoted by $\zeta$; see \eqref{eq:zeta_def}). This term, which is specific to the extended setting, complicates the error estimation. In contrast to \cite{lauriere2022convergence}, which addressed this term by assuming a small time horizon $T$ or a dissipativity condition, our monotonicity condition also plays a crucial role in this estimate.
Consequently, under an additional smoothness condition on the coefficients with respect to the measure arguments (Assumption~\ref{assump:smoothness_L-derivative}), we obtain the quantitative convergence rate using the propagation of chaos technique.

In addition, to demonstrate the applicability of our framework, we study a specific class of LQ-MFGs. Our setting introduces novel elements compared to the existing literature. For example, Graber \cite{graber2016linear} analyzed extended LQ-MFC problems and LQ-MFGs. He derived the unique existence of an MFE by restricting the analysis to cases where the state coefficients lack distributional dependence on either the state or the control. This restriction ensures the equivalence of the MFC and MFG problems. Bensoussan et al. \cite{bensoussan2025linear} addressed extended MFGs without volatility control. To obtain an MFE, they analyzed the well-posedness of linear FBODEs, derived by taking expectations of the original MKV-FBSDEs arising from the stochastic maximum principle. Crucially, the absence of volatility control ensures that the resulting FBODEs do not depend on the martingale term $Z$ of the BSDE. In contrast, we consider a specific type of LQ-MFGs that satisfy the monotonicity condition, establishing both the unique existence of an MFE and convergence results. This setting allows for dependence on the control distribution within the state process and incorporates volatility control. However, unlike \cite{bensoussan2025linear, graber2016linear}, we do not derive the closed-loop equilibrium to align with our main focus.

\subsection{Structure of this paper}
The paper is organized as follows.
Section \ref{sec:main_theorem} formulates the $N$-player game and the corresponding mean field game, and states the main convergence result, Theorem~\ref{thm:main}.
Section \ref{sec:characterization_of_equilibria} characterizes the NE and the MFE via the associated classical and McKean--Vlasov FBSDE systems.
Section \ref{sec:proof_of_main_theorem} presents the proof of the main theorem.
Section \ref{sec:linear-quadratic} analyzes a linear-quadratic example.
Appendix \ref{sec:well_posedness_cond_MKVFBSDE} establishes the well-posedness of conditional MKV-FBSDEs using the method of continuation, guaranteeing the unique existence of the MFE.
Appendix \ref{sec:coupling_technique} constructs a conditionally i.i.d.\ sequence of solutions to the conditional MKV-FBSDEs, which serves as a prerequisite for the coupling technique used in the convergence analysis.

\section{Setting and main results}
\label{sec:main_theorem}
\subsection{Notation}
Throughout the paper, $N \in \mathbb{N}$ denotes the number of players in the finite player game. 
Take $k, m \in \mathbb{N}$. 
We denote by $|\cdot|$ and $\cdot$ the standard norm and inner product in the Euclidean space $\R^{k}$.
We use the same notation for the Frobenius norm and inner product on $\R^{k \times m}$. 
Specifically, $A \cdot B := \operatorname{tr}(A^\top B)$ and $|A| := \sqrt{A \cdot A}$ for $A, B \in  \R^{k \times m}$.

We denote by $\mathcal{P}_2(\R^{k})$ the $2$-Wasserstein space, that is, the space of all probability measures $\mu$ on the measurable space $(\R^{k}, \mathcal{B}(\R^{k}))$ with $M_2(\mu) < \infty $, where
\[
M_2(\mu) := \left(\int_{\R^k} |x|^2 d\mu(x)\right)^{\frac{1}{2}},
\]
and the space $\mathcal{P}_2(\R^{k})$ is equipped with the Wasserstein metric,
\[
\mathcal{W}_2(\mu^1, \mu^2) := \inf_{\pi} \left(\int_{\R^{k}\times \R^{k}} |x-y|^2 d\pi(x,y) \right)^{\frac{1}{2}},
\]
for $\mu^1, \mu^2 \in \mathcal{P}_2(\R^{k})$, where the infimum is taken over all probability measures $\pi$ defined on $\R^{k}\times \R^{k}$ whose first and second marginals are $\mu^1$ and $\mu^2$, respectively.

\subsection{Probabilistic  setup}
We fix a finite time horizon $T > 0$ and a measure $\mu_0 \in \mathcal{P}_2(\R^n)$ for a dimension $n \in \mathbb{N}$. 
We use $\mu_0$ as the initial distribution of the $N$ players' states.
We work on the canonical probability space and separately define the sample spaces for common noise for all players and for initial states and idiosyncratic noises for each player of the game.
Specifically, let $\mathcal{C} = C([0,T];\R^d)$, and define the sample spaces $\Omega^0 := \mathcal{C}$,   $\Omega^{1,N} := (\R^n)^N  \times \mathcal{C}^N$, and $\Omega^N := \Omega^0 \times \Omega^{1,N}$. 
We consider the probability spaces $(\Omega^0, \widetilde{\mathcal{F}}^0, \widetilde{\mathbb{P}}^0)$ and $(\Omega^{1,N}, \widetilde{\mathcal{F}}^{1,N}, \widetilde{\mathbb{P}}^{1,N})$, 
where $ \widetilde{\mathcal{F}}^0$ and $\widetilde{\mathcal{F}}^{1,N}$
are the Borel $\sigma$-algebra on $\Omega^0$ and  $\Omega^{N,1}$,
respectively, and 
\[
\widetilde{\mathbb{P}}^0 = \mu_{W}, 
\quad  \widetilde{\mathbb{P}}^{1,N} = (\mu_0)^{\otimes N} \otimes  (\mu_{W})^{\otimes N}, 
\] 
where $\mu_W$ denotes a Wiener measure on $\mathcal{C}$. 
Moreover, we denote by
\[(\Omega^0, \mathcal{F}^0, \mathbb{P}^0), \quad 
(\Omega^{1,N}, \mathcal{F}^{1,N}, \mathbb{P}^{1,N}), \quad 
\]
the completions of $(\Omega^0, \widetilde{\mathcal{F}}^0, \widetilde{\mathbb{P}}^0)$ and  $(\Omega^{1,N}, \widetilde{\mathcal{F}}^{1,N}, \widetilde{\mathbb{P}}^{1,N})$, 
respectively. Now we are ready to set
\[
(\Omega^N, \mathcal{F}^N, \mathbb{P}^N)
\]
as the completion of the product measure space $(\Omega^N, \mathcal{F}^0\otimes\mathcal{F}^{1,N}, \mathbb{P}^0\otimes\mathbb{P}^{1,N})$.
Intuitively, $\Omega^0$ accommodates the sample path of the common noise for all players, $\Omega^{1,N}$ the initial states and the idiosyncratic noises for each player, and $\Omega^{N}$ combines all sources of randomness in the game.

Now, define the canonical random variables and processes
\[
W^0_t(\omega^0, \mathbf{x},\boldsymbol{\omega}) := w^0(t), 
\quad
X_0^i(\omega^0, \mathbf{x},\boldsymbol{\omega}) := x^i, 
\quad 
W^i_t(\omega^0, \mathbf{x},\boldsymbol{\omega}) := w^i(t), \quad \text{for} \quad i \in \{1, \ldots, N \}
\]
for $(\omega^0, \mathbf{x},\boldsymbol{\omega}) = (\omega^0,x^1, \ldots,  x^N, \omega^1, \ldots, \omega^N) \in \Omega^N$.
Then by definition, the collection $(X^i_0)_{i\in \{1,\ldots N\}}$ forms an i.i.d.\ sequence of random variables with the common law $\mu_0$, $(W^i)_{i\in \{0,1,\ldots N\}}$ are independent Brownian motions on $(\Omega^N, \mathcal{F}^N, \mathbb{P}^N)$, and $(X^i_0)_{i\in \{1,\ldots N\}}$ is independent of $(W^i)_{i\in \{1,\ldots N\}}$.

Let $\mathbb{F}^0 = (\mathcal{F}^{0}_t)_{t\in [0,T]}$ and $\mathbb{F}^N= (\mathcal{F}^{N}_t)_{t\in [0,T]}$ be the completions with $\mathbb{P}^N$-null sets of the raw filtrations generated by $W^0$ and $(W^0, X^1_0, \ldots, X^N_0, W^1, \ldots W^N)$, respectively. Automatically, $\mathbb{F}^0$ and $\mathbb{F}^N$ satisfy the usual conditions.

For $N < M$, we always identify the natural extension $\widetilde{X}$ on $(\Omega^M, \mathcal{F}^M, \mathbb{P}^M)$ of a random variable $X$ defined on $(\Omega^N, \mathcal{F}^N, \mathbb{P}^N)$ with $X$ itself, by setting 
\[\widetilde{X}(\omega^0, x^1,\ldots, x^M, \omega^1,\ldots,\omega^M) := X(\omega^0, x^1,\ldots, x^N, \omega^1,\ldots,\omega^N).\]
Furthermore, we shall not distinguish a random variable X on $(\Omega^0, \mathcal{F}^0, \mathbb{P}^0)$ from its natural extension $\widetilde{X}: (\omega^0, \mathbf{x}, \boldsymbol{\omega}) \mapsto X(\omega^0)$. Similarly, for a sub-$\sigma$-algebra $\mathcal{G}^0 \subset \mathcal{F}^0$, we often simply write $\mathcal{G}^0$ for the sub-$\sigma$-algebra $\mathcal{G}^0\otimes \{\emptyset, \Omega^{1,N}\} \subset \mathcal{F}^N$.

Next, for the mean field game, viewed as a limiting problem of the $N$-player game, we only require the probability space $(\Omega, \mathcal{F}, \mathbb{F}, \mathbb{P}) := (\Omega^1, \mathcal{F}^1, \mathbb{F}^1, \mathbb{P}^1)$. Similarly, we write $(\omega^0, x, \omega)$ for a generic element of $\Omega$ and set $W := W^1$. 

We denote by $\E^0$ the expectation with respect to the probability measure $\mathbb{P}^0$, and use the same notation $\E^{1}$ for the expectation with respect to both $\mathbb{P}^1$ and $\mathbb{P}^{1,N}$, as well as $\E^{}$ for $\mathbb{P}$ and $\mathbb{P}^{N}$.

Let $(\mathcal{X}, \mathcal{G}, \mathbb{G}, \mathbb{Q})$ be a generic filtered probability space with the filtration $\mathbb{G} = (\mathcal{G}_t)_{0 \leq t \leq T}$.
We define the following spaces:
\begin{itemize}
    \item $L^2(\mathcal{X}, \mathcal{G}, \mathbb{Q}; \R^k)$ is the space of $\mathcal{G}$-measurable random variables $X: \mathcal{X} \to \R^k$ satisfying
    \[ 
        \| X \|_{L^2} := \left( \E^{\mathbb{Q}}\big[ |X|^2 \big] \right)^{1/2} < \infty. 
    \]

    \item $\mathbb{S}^2(\mathbb{G}; \R^k)$ is the space of $\R^k$-valued $\mathbb{G}$-adapted continuous processes $Y$ satisfying
    \[ 
        \| Y \|_{\mathbb{S}^2} := \left( \E^{\mathbb{Q}}\left[ \sup_{0 \leq t \leq T} |Y_t|^2 \right] \right)^{1/2} < \infty. 
    \]

    \item $\mathbb{H}^2(\mathbb{G}; \R^k)$ is the space of $\R^k$-valued $\mathbb{G}$-progressively measurable processes $Z$ satisfying
    \[ 
        \| Z \|_{\mathbb{H}^2} := \left( \E^{\mathbb{Q}}\left[ \int_0^T |Z_t|^2 \, dt \right] \right)^{1/2} < \infty. 
    \]
\end{itemize}

\begin{remark}
This canonical setup plays a crucial role in the coupling technique underlying the propagation of chaos result; see Lemma \ref{lemma:iid_copies_solves_MKVFBSDE}. 
Specifically, it enables the construction of a conditionally i.i.d.\ sequence of the solutions to the conditional MKV-FBSDEs \eqref{eq:conditional_MKVFBSDE_for_i} on $(\Omega^N, \mathcal{F}^N, \mathbb{F}^N, \mathbb{P}^N)$ from the unique solution to \eqref{eq:mean_field_MKVFBSDE} on $(\Omega, \mathcal{F}, \mathbb{F}, \mathbb{P})$. Since this is rather technical, the rigorous explanation of this technique is provided in Section~\ref{sec:coupling_technique}. For the possibility of the extension to the general (non-canonical) setup, see Remark~\ref{rem:non_canonical}. 
\end{remark}

We are now ready to define the conditional law of a random variable and a process, given the realization of $W^0$. The following result can be found in \cite[Lemma 2.4]{carmona2018probabilistic_II} and \cite[Lemma~1.1]{kurtz1988unique}.
\begin{lemma}
\label{lemma:process_conditional_dist}
Given \( X \in L^2(\Omega^N, \mathcal{F}^N, \mathbb{P}^N;\R^k) \), the mapping 
$$
\mathcal{L}^1(X) : \Omega^0 \ni \omega^0 \mapsto \mathcal{L}(X(\omega^0, \cdot, \cdot))
$$
is almost surely well defined under \( \mathbb{P}^0 \), and forms a random variable from \( (\Omega^0, \mathcal{F}^0, \mathbb{P}^0) \) into \( \mathcal{P}_2(\R^k) \) endowed with its Borel \( \sigma \)-field. The random variable \( \mathcal{L}^1(X) \) provides a conditional law of \( X \) given \( \mathcal{F}^0 \). 

Moreover, let $(X_t)_{t\in[0,T]}$ be an \( \mathbb{F}^{N} \)-progressively measurable process such that $X_t \in  L^2(\Omega^N, \mathcal{F}^N, \mathbb{P}^N;\R^k)$ for each $t\in[0,T]$. Then, the \( \mathcal{P}_2(\R^k) \)-valued process \( (\mathcal{L}(X_t|\mathcal{F}^0_t ))_{t \geq 0} \) is \( \mathbb{F}^0 \)-progressively measurable.

Finally, consider for any \( t \geq 0 \), a version of \( \mathcal{L}^1(X_t) \), which is well-defined as a random variable on \( (\Omega^0, \mathcal{F}^0, \mathbb{P}^0) \). Then,  \( (\mathcal{L}^1(X_t))_{t \geq 0} \) is a $\mathbb{P}^N$-modification of \( (\mathcal{L}(X_t|\mathcal{F}^0_t ))_{t \geq 0} \).
\end{lemma}
Based on this lemma, we will always work with a $\mathbb{F}^0$-progressively measurable version of $(\mathcal{L}^1(X_t))_{t\in[0,T]}$ for $\mathbb{F}^N$-progressively measurable process $(X_t)_{t\in[0,T]}$.

In light of the above definitions and observations, we say that a sequence of random variables $(X^i)_{i\in\{ 1, \ldots, N\}}$ defined on $(\Omega^N, \mathcal{F}^N, \mathbb{P}^N)$ is conditionally i.i.d.\ if, for almost every $\omega^0 \in \Omega^0$, the collection $(X^i(\omega^0, \cdot, \cdot))_{i\in\{ 1, \ldots, N\}}$ is i.i.d.\ on $(\Omega^{1,N}, \mathcal{F}^{1,N}, \mathbb{P}^{1,N})$. 
Similarly, we say a sequence of stochastic processes $((X_t^i)_{t\in[0,T]})_{i\in\{ 1, \ldots, N\}}$ is conditionally i.i.d.\ if, for each $t\in[0,T]$, $(X_t^i(\omega^0, \cdot, \cdot))_{i\in\{ 1, \ldots, N\}}$ is i.i.d.\ on $(\Omega^{1,N}, \mathcal{F}^{1,N}, \mathbb{P}^{1,N})$ for almost every $\omega^0 \in \Omega^0$. 

\subsection{The $N$-player game}
Let us first introduce the $N$-player game. For $\bs{\alpha} = (\alpha^1, \ldots, \alpha^N)$ where $\alpha^i$ denotes the strategy process of each agent $i=1,\ldots,N$, the dynamics of $i$-th player's state are defined by the solution to the following controlled SDE:
\begin{align*}
\left\{
\begin{aligned}
&dX_t^{i,\bs{\alpha}} = b_t\big(X_t^{i,\bs{\alpha}}, \alpha_t^i, L^N(\bs{X}^{\bs{\alpha}}_t, \bs{\alpha}_t)\big)dt 
+ \sigma_t\big(X_t^{i,\bs{\alpha}}, \alpha_t^i, L^N(\bs{X}^{\bs{\alpha}}_t, \bs{\alpha}_t)\big)dW_t^i
+ \sigma^0_t\big(X_t^{i,\bs{\alpha}}, \alpha_t^i, L^N(\bs{X}^{\bs{\alpha}}_t, \bs{\alpha}_t)\big)dW_t^0, \\
&X_0^{i,\bs{\alpha}} = 
X_0^i,
\end{aligned}
\right.
\end{align*}
where we denote
\[
\bs{X}^{\bs{\alpha}} := (X^{1,\bs{\alpha}}\, \ldots, X^{N,\bs{\alpha}} ),\, \text{and}\,\,
L^N (\bs{x}, \bs{a}) := \frac{1}{N} \sum^N_{j=1} \delta_{(x^j,a^j)}, \,\,
\text{for $(\bs{x}, \bs{a}) \in (\R^n)^N \times (\R^{\ell})^N $}.
\]
Clearly, each player’s state depends on the empirical distribution of all players’ states and controls, and thus the superscript $\bs{\alpha}$ is incorporated into $X^{i,\bs{\alpha}}$ for each $i$ to reflect this dependence.
Each player selects their strategy from the set of admissible strategies:
\[
\mathbb{A}_N := \mathbb{H}^2(\mathbb{F}^N; \R^k)
\]
and aims to minimize the following cost functional, given the running and terminal cost functions, $f$ and $g$, respectively, and  for $\bs{\alpha}\in (\mathbb{A}_N)^N$,
\begin{align*}
J(\alpha^i; \bs{\alpha}^{-i}) :=\E  \left[
\int^T_0 f(t, X^{i,\bs{\alpha}}_t, \alpha^i_t, L^N(\bs{X}_t^{\bs{\alpha}}, \bs{\alpha}_t)) \, dt 
+ g(X^{i,\bs{\alpha}}_T,  L^N(\bs{X}^{\bs{\alpha}}_T))
\right],
\end{align*}
where we use the shorthand notation 
\[\bs{\alpha}^{-i} := (\alpha^1, \ldots, \alpha^{i-1},\alpha^{i+1}, \ldots, \alpha^N ), \,\, \text{and}\,\, (\beta; \bs{\alpha}^{-i}) := (\alpha^1, \ldots, \alpha^{i-1},\beta, \alpha^{i+1},\ldots, \alpha^N ),
\]
for $ \beta\in \mathbb{A}_N$.
As usual, we define the NE of the game.
\begin{definition}
We say that $\bs{\alpha} \in (\mathbb{A}_N)^N$ is an $N$-player Nash equilibrium (NE) if for all $i = 1,2, \ldots, N$, it holds that 
\[
J(\alpha^i; \bs{\alpha}^{-i}) = \inf_{\beta \in \mathbb{A}_N} J(\beta; \bs{\alpha}^{-i})
\]
\end{definition}
Intuitively, this means that each player adopts the best possible strategy in response to the strategies of all other players.

To ensure that the game is well-defined and analytically tractable, we introduce the following assumptions on the coefficient of the state process and cost functions. 

Throughout the paper, we use the $L$-derivative to describe differentiation on the 2-Wasserstein space, which is defined as the Fr\'echet derivative of the lifted function on the space of square integrable random variables.  
See \cite[Chapter~5]{carmona2018probabilistic_I} for the precise definition and the basic properties. 
For an $\R^k$-valued function $\varphi$ defined on $\mathcal{P}_2(\R^n)$, $\mathcal{P}_2(\R^{\ell})$, 
and $\mathcal{P}_2(\R^{n\times k})$, we use the notations $\partial_{\mu}\varphi, \partial_{\nu}\varphi$ and $\partial_{\xi}\varphi$, respectively, to denote the $L$-derivatives of the functions on these spaces. 
We say that the function $\varphi$ is continuously $L$-differentiable, when the Fréchet derivative of the liftup is continuous as a function defined on the $L^2$-space of the related probability space.
\begin{assump}
\label{assump:standard_assump}
Let
\begin{equation*}
\begin{aligned}
& b : [0,T] \times \R^n \times \R^{\ell} \times \mathcal{P}_2(\R^n \times \R^{\ell}) \to \R^n, \\
& \sigma, \sigma^0 : [0,T] \times \R^n \times \R^{\ell} \times \mathcal{P}_2(\R^n \times \R^{\ell}) \to \R^{n \times d}, \\
& f : [0,T] \times \R^n \times \R^{\ell} \times \mathcal{P}_2(\R^n \times \R^{\ell}) \to \R, \,\,  \text{and} \\
& g : \R^n \times \mathcal{P}_2(\R^n) \to \R
\end{aligned}
\end{equation*}
be Borel measurable and satisfy:
\begin{enumerate}[label=(\roman*)]
    \item \label{assump:differentiability}
    For \( \varphi = b, \sigma, \sigma^0, f \), the function \(\varphi_t( x, a, \xi)\) is continuously differentiable in its last three arguments, and we denote its derivatives by \(\partial_x \varphi\), \(\partial_{\alpha} \varphi\), and \(\partial_{\xi} \varphi\), respectively. In addition, \( g(x, \mu) \) is continuously differentiable in all its arguments, and we write \(\partial_x g\) and \(\partial_{\mu} g\) for its derivatives.

    \item \label{assump:Lipschitz}
    The functions \(\varphi = b, \sigma, \sigma^0, \partial_x f\) are Lipschitz continuous in their last three arguments, uniformly in \(t\); that is, there exists a constant \(L_f > 0\) such that, for any \((t, x, a, \xi), (t, x', a', \xi') \in [0, T] \times \R^n \times \R^{\ell} \times \mathcal{P}_2(\R^n \times \R^{\ell})\),
    \[
    |\varphi_t( x, a, \xi) - \varphi_t( x', a', \xi')| \leq L_f \big( |x - x'| + |a - a'| + \mathcal{W}_2(\xi, \xi') \big).
    \]
    
    In addition, for any \((t, x, a, \xi ,u ), (t, x', a', \xi', u') \in [0, T] \times \R^n \times \R^{\ell} \times \mathcal{P}_2(\R^n \times \R^{\ell}) \times \R^n\), the following Lipschitz conditions hold:
    \begin{align*}
    |\partial_{\mu} f_t( x, a, \xi)(u) - \partial_{\mu} f_t( x', a', \xi')(u')| &\leq L_f \big( |x - x'| + |a - a'| + |u - u'| + \mathcal{W}_2(\xi, \xi') \big), \\
    |\partial_x g(x, \mu) - \partial_x g(x', \mu')| &\leq L_f \big( |x - x'| + \mathcal{W}_2(\mu, \mu') \big), \\
    |\partial_{\mu} g(x, \mu)(u) - \partial_{\mu} g(x', \mu')(u')| &\leq L_f \big( |x - x'| + |u - u'| + \mathcal{W}_2(\mu, \mu') \big),
    \end{align*}
    where $\mu$ is the first marginal of $\xi$ on $\mathcal{P}_2(\R^n)$.

    \item \label{assump:linear_growth}
    The functions \(\varphi = b, \sigma, \sigma^0, \partial_x f\) are of linear growth; that is, there exists a constant \( C > 0 \) such that, for any 
    $(t, x, a, \xi) \in [0, T] \times \R^n \times \R^{\ell} \times \mathcal{P}_2(\R^n \times \R^{\ell})$, we have
    \[
    |\varphi_t(x, a, \xi)| \leq C \big(1 + |x| + |a| + M_2(\xi)\big).
    \]
    In addition, the derivatives  \(\partial_x g\), \(\partial_{\xi} f\), and \(\partial_{\mu} g\) also satisfy the following linear growth condition: for any 
    $
    (t, x, a, \xi, u, v) \in [0, T] \times \R^n \times \R^{\ell} \times \mathcal{P}_2(\R^n \times \R^{\ell}) \times \R^n \times \R^{\ell},
    $  we have
    \begin{align*}
    |\partial_{\xi} f_t(x, a, \xi)(u, v)| &\leq C \big(1 + |x| + |a| + |u| + |v| + M_2(\xi)\big), \\
    |\partial_x g(x, \mu)| &\leq C \big(1 + |x| + M_2(\mu)\big), \\
    |\partial_{\mu} g(x, \mu)(u)| &\leq C \big(1 + |x| + |u| + M_2(\mu)\big).
    \end{align*}

    \item \label{assump:quadratic_growth}
    For any \((t, x, a, \xi )  \in [0, T] \times \R^n \times \R^{\ell} \times \mathcal{P}_2(\R^n \times \R^{\ell})\times  \R^n\), $f$ and $g$ satisfy the following quadratic growth condition, 
    \[
    |f_t(x,a,\xi)| \leq C_f( 1 + |x|^2 + |a|^2 + M^2_2(\xi)),
    \]
    \[
    |g(x,\mu)| \leq C_g( 1 + |x|^2 + M^2_2(\mu)),
    \]
    for some constant $C_f, C_g >0$.

    \item   \label{assump:bounded_L_derivative}
    The $L$-derivative $\partial_{\xi}b, \partial_{\xi}\sigma$, and $\partial_{\xi}\sigma^0$ are all bounded by some positive constant $C$.
    
    \item \label{assump:decomposition}
    The functions \(\varphi = b, \sigma, \sigma^0, f\) can be decomposed as
    \[
    \varphi_t( x, a, \xi) = \varphi^{(1)}_{t} (x, a, \mu) + \varphi^{(2)}_t(t, x, \xi),
    \]
    where \(\varphi^{(1)}\) and \(\varphi^{(2)}\) are Borel measurable functions, and \(\mu\) denotes the first marginal of \(\xi\) on \(\R^n\).
\end{enumerate}
\end{assump}

\begin{remark}
Let us elaborate on the assumptions. 
\ref{assump:differentiability} is a prerequisite for applying the SMP. Specifically, the derivatives assumed here drive the adjoint equations in the FBSDEs that characterize the equilibria.
\ref{assump:Lipschitz}, \ref{assump:linear_growth}, and \ref{assump:bounded_L_derivative} are invoked repeatedly throughout the proof of the main theorem to establish moment bounds and error estimates for the FBSDE systems.
\ref{assump:quadratic_growth} ensures that the cost functional remains finite for any choice of admissible strategies.
\ref{assump:decomposition} is usually referred to as the separability condition. Although this is restrictive, it is fairly standard in the literature on extended MFGs; see \cite{cardaliaguet2018mean, carmona2015weak, djete2022extended, lauriere2022convergence}. In our setting, this condition decouples the optimizer $\Lambda$ of the Hamiltonian (see Lemma \ref{lemma:optimizer_of_Hamiltonian}) from the distribution of the control. In order to dispense with this condition, one might rely on an analysis of the fixed-point maps as detailed in \cite[Section~3]{jackson2025mfg_control}. However, to avoid further  technical complexity, we impose the condition.
\end{remark}

Let us also impose some assumptions on the Hamiltonian $H$, taking values in $\R$:
\begin{equation}
\label{eq:def_Hamiltonian}
H_t(x, y, z, z^0, a,  \xi) := f_t(x,a,\xi) + b_t(x,a,\xi)\cdot y +\sigma_t(x,a,\xi) \cdot z + \sigma^0_t (x,a,\xi) \cdot z^0,
\end{equation}
for  $(t, x, y, z, z^0, a, \xi) \in [0,T] \times \R^n \times \R^n \times \R^{n\times d} \times \R^{n\times d} \times \R^{\ell}  \times \mathcal{P}_2(\R^n\times\R^{\ell})$.
\begin{assump}
\label{assump:Hamiltonian}
Let $H$ be as defined in \eqref{eq:def_Hamiltonian}.
\begin{enumerate}[label=(\roman*)]
    \item \label{assump:Lipschitz_Hamiltonian}
    We assume that for all $(t, x, y, z, z^0, a, \xi, u)$ and $(t, x^{\prime},  y^{\prime}, z^{\prime}, z^{0\prime}, a^{\prime}, \xi^{\prime}, u^{\prime})$ in 
    \[
    [0, T] \times \R^n \times \R^n \times \R^{n \times d} \times \R^{n \times d} \times \R^{\ell} \times \mathcal{P}_2(\R^n \times \R^{\ell}) \times \R^n,
    \]
    the functions $\partial_x H$ and $\partial_{\mu} H$ satisfy a Lipschitz continuity condition uniformly in $t$, that is,
    \begin{align*}
    &|\partial_x H_t( x, y, z, z^0, a, \xi) - \partial_x H_t( x^{\prime}, y^{\prime}, z^{\prime}, z^{0\prime}, a^{\prime}, \xi^{\prime})| \\
    &\hspace{4em} \leq L_f \big( |x - x^{\prime}|  + |y - y^{\prime}| + |z - z^{\prime}| + |z^0 - z^{0\prime}| + |a - a^{\prime}| + \mathcal{W}_2(\xi, \xi^{\prime}) \big),\\
    &|\partial_{\mu} H_t (x, y, z, z^0, a, \xi)(u) - \partial_{\mu} H_t (x^{\prime}, y^{\prime}, z^{\prime}, z^{0\prime}, a^{\prime}, \xi^{\prime})(u^{\prime})| \\
    &\hspace{4em} \leq L_f \big( |x - x^{\prime}|  + |y - y^{\prime}| + |z - z^{\prime}| + |z^0 - z^{0\prime}| 
    + |a - a^{\prime}| + |u - u^{\prime}| + \mathcal{W}_2(\xi, \xi^{\prime}) \big),
    \end{align*}
    where $L_f$ is a positive constant.

    \item \label{assump:convex_Hamiltonian}
    We assume that $H$ is strongly convex with respect to $a$, that is,  
    for any fixed $(t, x, y, z, z^0, \xi) \in [0, T] \times \R^n \times \R^n \times \R^{n \times d} \times \R^{n \times d} \times \mathcal{P}_2(\R^n \times \R^{\ell})$ and any $a, a^{\prime} \in \R^{\ell}$, we have
    \begin{align*}
    H_t(x,  y, z, z^0, a^{\prime}, \xi) - H_t(x, y, z, z^0, a, \xi) 
    \geq \partial_{\alpha} H_t(x, y, z, z^0, a,  \xi) \cdot (a^{\prime} - a) + \gamma |a^{\prime} - a|^2,
    \end{align*}
    for some $\gamma > 0$.

    \item \label{assump:convex_g_H}
    The functions $x\mapsto g(x,\mu)$ and $x\mapsto H_t( x,  y, z, z^0, a, \xi)$ are convex for any fixed $(t, x, y, z, z^0, a, \xi) \in [0, T] \times \R^n \times \R^n \times \R^{n \times d} \times \R^{n \times d} \times \R^\ell \times \mathcal{P}_2(\R^n \times \R^{\ell})$
\end{enumerate}
\end{assump}
\begin{remark}
It is clear that if Assumptions \ref{assump:standard_assump}, \ref{assump:Lipschitz}, \ref{assump:linear_growth}, and \ref{assump:bounded_L_derivative} are satisfied, then $\partial_{x} H$ and $\partial_{\mu} H$ have at most linear growth. The strong convexity assumption ensures the existence of a unique minimizer of the Hamiltonian and guarantees that this minimizer is Lipschitz continuous; see the subsequent Lemma \ref{lemma:optimizer_of_Hamiltonian}. The convexity conditions on $g$ and $H$ in x of \ref{assump:convex_g_H} are also required to establish the sufficiency part of the maximum principle in Lemma~\ref{lemma:SMP_MFG}.
\end{remark}

The next result is standard in the literature. For example, see \cite[Lemma~5.3]{lauriere2022convergence} and \cite[Lemma 3.3]{carmona2018probabilistic_I}.
\begin{lemma}
\label{lemma:optimizer_of_Hamiltonian}
Let Assumption \ref{assump:Hamiltonian} hold. Then, there exists a Borel measurable map
\[
\Lambda: [0, T] \times \R^n \times \R^n \times \R^{n \times d} \times \R^{n \times d} \times \mathcal{P}_2(\R^n) \times \R^{\ell} \to \R^{\ell}
\]
such that  for any $\zeta \in \R^{\ell}$ and any $(t, x, y, z, z^0, \xi) \in [0, T] \times \R^n \times \R^n \times \R^{n \times d} \times \R^{n \times d} \times \mathcal{P}_2(\R^n \times \R^{\ell})$,  
$\Lambda_t(x, y, z, z^0, \mu, \zeta)$ is the unique minimizer of the function $a \mapsto H_t( x, y, z, z^0, a, \xi) + a \cdot \zeta$.
Equivalently, it satisfies the following first-order condition:
\[
\partial_{\alpha} H_t(x,  y, z, z^0, \Lambda_t(x, y, z, z^0, \mu, \zeta), \xi) + \zeta = 0,
\]
or more specifically,
\begin{align*}
&\partial_{\alpha} f^{(1)}_t\big(x, \Lambda_t(x, y, z, z^0, \mu, \zeta), \mu\big)
+ \partial_{\alpha} b^{(1)}_t\big(x, \Lambda_t(x, y, z, z^0, \mu, \zeta), \mu\big) \\
&\quad + \partial_{\alpha} \sigma^{(1)}_t\big(x, \Lambda_t(x, y, z, z^0, \mu, \zeta), \mu\big)
+ \partial_{\alpha} \sigma^{0,(1)}_t\big(x, \Lambda_t(x, y, z, z^0, \mu, \zeta), \mu\big)
+ \zeta = 0,
\end{align*}
where $\mu$ denotes the first marginal of $\xi$.
Moreover, $\Lambda$ is Lipschitz continuous uniformly in $t$ and has at most linear growth.
\end{lemma}

\begin{remark}
Instead of using $\R^{\ell}$ as the space of the players' control values, we can consider a convex closed subset of $\R^{\ell}$. In that case, the optimizer $\Lambda$ of $H + a\cdot \zeta$ still exists, but it satisfies a variational inequality instead of the first-order condition stated above.

The variable $\zeta$ accounts for interactions between players via their controls. Specifically, in the context of Lemma \ref{lemma:smp_N_player}, $\zeta$ corresponds to the $L$-derivative terms appearing in the partial derivative $\partial_{\alpha_i} H^{N,i}$ of the N-player Hamiltonian \eqref{eq:N-player_Hamiltonian_def} (see \eqref{eq:zeta_def}). Such terms arise specifically in extended games where the coefficients $b, \sigma, \sigma^0, f$ depend on the distribution of controls.
\end{remark}

\subsection{The mean field game}
\label{sec:mean-field_game}
We now formulate the mean field game of the $N$-player game introduced in the previous section. Throughout this section, we work on the filtered probability space $(\Omega, \mathcal{F}, \mathbb{F}, \mathbb{P})$. The representative player selects a strategy $\alpha$ from the set of admissible controls:
\[
\mathbb{A} := \mathbb{H}^2(\mathbb{F}; \R^\ell).
\]
Given some stochastic flow of probability measures $\xi:[0,T]\times \Omega^0 \to \mathcal{P}_2(\R^n \times \R^{\ell})$, which is $\mathbb{F}^0$-progressively measurable, the player's state process is defined as the solution of the following SDE:
\begin{align*}
\left\{
\begin{aligned}
&dX_t^{\alpha,\xi} = b_t\big(X_t^{\alpha,\xi}, \alpha_t, \xi_t\big)\, dt 
+ \sigma_t\big(X_t^{\alpha,\xi}, \alpha_t, \xi_t\big)\,dW_t 
+ \sigma^0_t\big(X_t^{\alpha,\xi}, \alpha_t, \xi_t\big)\, dW_t^0,\\ 
&X_0^{\alpha,\xi} = X^1_0.
\end{aligned}
\right.
\end{align*}
Analogously to the $N$-player game, we now define the cost functional for the representative player 
\[J^{\xi}(\alpha):=\E  \left[ 
\int^T_0 f_t( X^{\alpha,\xi}_t,  \alpha_t,  \xi_t) \, dt 
+ g(X^{\alpha,\xi}_T,\mu_T)
\right], 
\]
where $\mu_T \in \mathcal{P}_2(\R^n)$ is the first marginal of $\xi_T$.

Now we are ready to define the mean field equilibrium:
\begin{definition}
Given $\hat{\alpha} \in \mathbb{A}$ and an $\mathbb{F}^0$-progressively measurable process $\hat{\xi}: [0,T] \times \Omega \to \mathcal{P}_2(\R^n \times \R^\ell)$ with $\E \int^T_0 M_2^2(\hat{\xi}_t) \, dt < \infty$, we say that the pair $(\hat{\alpha}, \hat{\xi})$ is a mean field equilibrium (MFE) if it satisfies the optimality:
    \[J^{\hat{\xi}}(\hat{\alpha} ) = \inf_{\beta \in \mathbb{A}} J^{\hat{\xi}}(\beta) , \]
and the fixed point condition:
    \[
    \hat{\xi}_t = \mathcal{L}^1(X_t^{\hat{\alpha},\hat{\xi}},\hat{\alpha}_t),  \,\, \mathbb{P}^0 \text{-} a.s. \,\, \text{for all $t\in[0,T]$.}
    \]  

We say that the MFG admits a unique MFE if it possesses at least one MFE, and for any two MFE $(\alpha, \xi)$ and $(\alpha', \xi')$, we have $\alpha = \alpha'$ in $\mathbb{H}^2(\mathbb{F}; \R^{\ell})$ and $\xi$ and $\xi'$ are modifications of each other.
\end{definition}

We conclude this section by establishing the existence of conditionally i.i.d. copies of the state and strategy pair in equilibrium. The construction of such copies is conducted in Appendix~\ref{sec:coupling_technique}.
\begin{lemma}
\label{lemma:conditionally_i.i.d.copies}
Suppose Assumptions~\ref{assump:standard_assump} and \ref{assump:Hamiltonian} hold. Let $(\alpha, \xi)$ be an MFE.
Then, on the space $(\Omega^N, \mathcal{F}^N, \mathbb{P}^N)$, there exist conditionally i.i.d.\ copies $\{(X^i, \alpha^i)\}_{i=1}^N$ of the process $(X^{\alpha, \xi}, \alpha)$.
These processes share the same joint distribution on $\mathcal{C} \times L^2([0,T])$ as the pair $(X^{\alpha, \xi}, \alpha)$.
\end{lemma}

\subsection{Main convergence result}
Here, we present our main monotonicity assumptions and the convergence result.
First, using $\Lambda$ from Lemma \ref{lemma:optimizer_of_Hamiltonian}, for $\theta = (x,y,z,z^0)$ and $\xi \in \mathcal{P}_2(
\R^n\times \R^n\times \R^{n\times d}\times \R^{n\times d } 
) 
\to \mathcal{P}_2 (\R^n \times \R^{\ell} ) $, let us denote
\begin{equation}
\label{eq:from_b_to_B}
\begin{aligned}
B_t(\theta,\xi)
 &:= b_t\big(x, \Lambda_t(\theta, \mu, 0), \varphi_t(\xi)\big), \\
\Sigma_t(\theta,\xi)
 &:= \sigma_t \big(x, \Lambda_t(\theta, \mu, 0), \varphi_t(\xi)\big), \\
\Sigma^0_t(\theta,\xi)
 &:= \sigma^0_t\big( x, \Lambda_t(\theta, \mu, 0), \varphi_t(\xi)\big), \\
F_t(\theta,\xi)
 &:= \partial_x H_t\big(x, \Lambda_t(\theta, \mu, 0), \varphi_t(\xi)\big), \\
G(x,\mu)
 &:= \partial_x g(x,\mu).
\end{aligned}
\end{equation}
for a function $\varphi : [0,T] \times \mathcal{P}_2(
\R^n\times \R^n\times \R^{n\times d}\times \R^{n\times d } 
) 
\to \mathcal{P}_2 (\R^n \times \R^{\ell} ) $:
\begin{equation}
\label{eq:change_space_of_measure}
\varphi_t(\xi) := \xi \circ \left( \mathrm{id}_{x}, \Lambda_t( \cdot, \cdot, \cdot, \cdot, \mu, 0) \right)^{-1},
\end{equation}
where $\mu$ is the first marginal of $\xi$, and $\mathrm{id}_{x}$ denotes the projection $\theta = (x,y,z,z^0)\mapsto x$.

The following is our standard assumption on the monotonicity of the coefficients. The explanation is provided in Remark \ref{rem:main_on_assumption}.
\begin{assump} 
\label{assump:displacement_monotonicity}
There exists $C_H, C_G >0$ such that for any square integrable random variables $\Theta= (X, Y, Z,  Z^0), \Theta' = ( X',  Y',  Z', Z^{0 \prime} )$  of appropriate dimensions,
\begin{equation}
\label{eq:standard_monotonicity}
\begin{aligned}
\E_{ (\Theta, \Theta')}\Big[ \, 
-\ &\Delta X \cdot \big( F_t(\Theta, \mathcal{L}(\Theta)) - F_t(\Theta', \mathcal{L}(\Theta')) \big) \\
+\ &\Delta Y \cdot \big( B_t(\Theta, \mathcal{L}(\Theta)) - B_t(\Theta', \mathcal{L}(\Theta')) \big) \\
+\ &\Delta Z \cdot \big( \Sigma_t(\Theta, \mathcal{L}(\Theta)) - \Sigma_t(\Theta', \mathcal{L}(\Theta')) \big) \\
+\ &\Delta Z^0 \cdot \big( \Sigma^0_t(\Theta, \mathcal{L}(\Theta)) - \Sigma^0_t(\Theta', \mathcal{L}(\Theta')) \big)
\, \Big] 
\leq -C_H \E_{(X, X')}[|X - X'|^2]
\end{aligned}
\end{equation}
and 
\begin{equation}
\label{eq:displacement_monotonicity}
\E_{(X, X')} \left[ \Delta X \cdot \big( G(X, \mathcal{L}(X)) - G(X', \mathcal{L}(X')) \big) \right] 
\geq C_G \E_{(X, X')}[|X - X'|^2]
\end{equation}
where $\E_{(\Theta, \Theta')}$ and $\E_{(X, X')}$ denote the expectations with respect to the joint distributions of $(\Theta, \Theta')$ and $(X, X')$, respectively, and where we have defined
\[
\Delta X = X - X', \quad \Delta Y = Y - Y', \quad \Delta Z = Z - Z', \quad \Delta Z^0 = Z^0 - Z^{0\prime}.
\]
\end{assump}

Next, we introduce the smoothness assumption on the $L$-derivatives of the coefficients. Again, the discussion is deferred to Remark \ref{rem:main_on_assumption}.
\begin{assump}
\label{assump:smoothness_L-derivative}
For $\Phi = B, \Sigma, \Sigma^0, F$ and each fixed $(t,x,y,z,z^0)$, the map $\xi \mapsto \Phi_t(x,y,z,z^0,\xi)$ is twice $L$-differentiable. The first and second order derivatives, denoted by
\begin{align*}\partial_\xi \Phi_t(x,y,z,z^0, \xi, p)  &= \partial_\xi [\Phi_t(x,y,z,z^0, \cdot)] (\xi, p), \\
\partial_{\xi\xi} \Phi_t(x,y,z,z^0, \xi, p,q) &= \partial_{\xi\xi} [\Phi_t(x,y,z,z^0, \cdot)] (\xi, p, q),
\end{align*}are assumed to be bounded and Lipschitz continuous in $\xi$, uniformly with respect to $(t,x,y,z,z^0)$.
Similarly, for each fixed $x$, the map $\mu \mapsto G(x,\mu)$ is twice $L$-differentiable. Its derivatives,
\begin{align*}\partial_\mu G(x,\mu, p)  &= \partial_\mu [G(x,\cdot) ] (\mu, p), \\ 
\partial_{\mu\mu} G(x,\mu, p,q) &= \partial_{\mu\mu} [G(x,\cdot)] (\mu, p, q),
\end{align*}
    are bounded and Lipschitz continuous in  $\mu$ uniformly in $x$.
\end{assump}

We are now ready to state our main convergence result.
\begin{theorem}
\label{thm:main}
Suppose Assumptions~\ref{assump:standard_assump}, \ref{assump:Hamiltonian}, \ref{assump:displacement_monotonicity} and \ref{assump:smoothness_L-derivative} hold. Then the MFG admits a unique MFE, with the representative player's control \(\alpha \in \mathbb{A}\) and state process \(X \in \mathbb{S}^2(\mathbb{F}; \R^n)\).

Furthermore, for any $N \in \mathbb{N}$, let \(\bs{\alpha}^N \in (\mathbb{A}_N)^N\) be the NE of the \(N\)-player game, and denote the corresponding state processes by \((X^{N,1}, \ldots, X^{N,N}) \in (\mathbb{S}^2(\mathbb{F}^N; \R^n))^N\). Then, there exists a constant $C>0$, independent of $N$, such that the following estimates hold:
\[
\frac{1}{N} \sum^N_{i=1}
\E \left[
\sup_{0 \leq t \leq T} |X^i_t - X_t^{N,i}|^2
\right]
\leq \frac{C}{N},
\]
and
\[
\frac{1}{N} \sum^N_{i=1}
\E \left[
    \int_0^T
    |\alpha^i_t - \alpha_t^{N,i}|^2 \, dt
\right]
\leq \frac{C}{N},
\]
where $\{(X^i, \alpha^i)\}_{i=1}^N$ denote the conditionally i.i.d.\ copies of the pair $(X, \alpha)$ as defined in Lemma~\ref{lemma:conditionally_i.i.d.copies}.
\end{theorem}

\begin{remark}
\label{rem:main_on_assumption}
Let us elaborate on Assumption \ref{assump:displacement_monotonicity} and \ref{assump:smoothness_L-derivative}.

Theorem~\ref{thm:unique_existence_MFE} guarantees the unique existence of the MFE by establishing the well-posedness of the corresponding conditional MKV-FBSDE system. Notably, this result does not require the additional smoothness condition imposed by Assumption~\ref{assump:smoothness_L-derivative}.
The monotonicity condition in Assumption \ref{assump:displacement_monotonicity} is standard in the literature of the well-posedness of (MKV-)FBSDE, based on the method of continuation argument; see \cite{bensoussan2015well, hu1995solution}, \cite[Appendix~A]{jackson2024quantitative}, and \cite[Chapter~8.4]{zhang2017backward}.

Assumption \ref{assump:smoothness_L-derivative} is inspired by \cite[Assumption~6]{jackson2024quantitative} and imposed solely to invoke \cite[Lemma 9]{jackson2024quantitative}, which relies on the result by Chassagneux et al.\ \cite[Theorem 2.14]{chassagneux2022weak}; see \eqref{eq:E^2_estimate_second_term} at the end of the proof of Theorem \ref{thm:main}.
If the solution $\Theta=(X,Y,Z,Z^0)$ to the FBSDE \eqref{eq:mean_field_system} possesses a finite $q$-th moment $\E[ |\Theta_t|^q ] < \infty$ for some $q>2$, we could effectively dispense with the smoothness assumption and instead employ the result by Fournier and Guillin \cite[Theorem~1]{fournier2015rate}.
However, obtaining pointwise a priori estimates such as $\E[|Z_t|^q] < \infty$ and $\E[|Z^0_t|^q]<\infty$ for the backward components is typically difficult.
Even assuming $\E[\sup_{t\in [0,T]} |X_t|^2] < \infty$ (which is itself subtle in our setting; see \cite[Remark~5]{jackson2024quantitative}), standard a priori estimates for FBSDEs typically yield bounds only on the integrated terms, such as $\E[(\int^T_0|Z_t|^2\, dt )^q]$ and $\E[(\int^T_0|Z^0_t|^2\, dt )^q]$.
For this reason, we adopt this smoothness assumption.
\end{remark}

\begin{remark}
If we restrict our attention to the symmetric NE $\bs{\alpha}^N$---symmetric in the sense that the tuple $\left( (X^1_0, \alpha^{N,1}, W^0, W^1) , \ldots, (X^N_0, \alpha^{N,N}, W^0, W^N) \right)$ is exchangeable---then the equality $\E [ \int^T_0 |\alpha^i_t - \alpha^{i,N}_t|^2 \, dt ] = \E [ \int^T_0 |\alpha^j_t - \alpha^{j,N}_t|^2 \, dt ]$ holds for any $i,j \in \{1, \ldots, N\}$. Consequently, the estimate above can be strengthened to
\[
\E \left[
    \int_0^T
    |\alpha^i_t - \alpha_t^{N,i}|^2 \, dt
\right]
\leq \frac{C}{N},
\]
for all $i \in \{1, \ldots, N\}$.
This holds true, in particular, when the uniqueness of the $N$-player NE is guaranteed. Such uniqueness has been established for games without control interactions; see \cite[Lemma 7]{jackson2024quantitative}. In our setting, however, establishing uniqueness is subtle. For this reason, we rely on the average over $i\in\{1,\ldots,N\}$.
\end{remark}

\section{Characterization of equilibria via the stochastic maximum principle}
\label{sec:characterization_of_equilibria}
In this section, we characterize the equilibria for both the $N$-player game and the MFG by applying Pontryagin's maximum principle.

\subsection{The maximum principle for the $N$-player game}
First, let us define the Hamiltonian for the \( N \)-player game:
\begin{equation}
\label{eq:N-player_Hamiltonian_def}
\begin{split}
    H^{N,i}_t(\bs{x}, \bs{y}, \bs{z}, \bs{z}^{0}, \bs{\alpha}) 
    &\quad := f_t(x^i, \alpha^i, L^N(\bs{x}, \bs{\alpha})) 
        + \sum_{j=1}^N b_t(x^j, \alpha^j, L^N(\bs{x}, \bs{\alpha})) \cdot y^{j} \\
    & \qquad + \sum_{j=1}^N \sigma_t(x^j, \alpha^j, L^N(\bs{x}, \bs{\alpha})) \cdot z^{j} 
        + \sum_{j=1}^N \sigma^0_t(x^j, \alpha^j, L^N(\bs{x}, \bs{\alpha})) \cdot z^{0,j}, 
\end{split}
\end{equation}
and introduce the notation
\[
g^{N,i}(\bs{x}) := g(x^i, L^N(\bs{x})).
\]
For the \(N\)-player game, we establish the necessary condition for optimality.

\begin{lemma}
\label{lemma:smp_N_player}
Let Assumptions~\ref{assump:standard_assump} and~\ref{assump:Hamiltonian} hold, and let \(\bs{\hat{\alpha}} \in (\mathbb{A}_N)^N\) be a NE. Then there exists a solution \((Y^{i,j}, Z^{i,j,k}, Z^{0,i,j})_{(i,j,k) \in \{1, \ldots, N\}^3}\), where \(Y^{i,j} \in \mathbb{S}^2(\mathbb{F}^N; \R^n)\), and \(Z^{i,j,k}, Z^{0,i,j} \in \mathbb{H}^2(\mathbb{F}^N; \R^{n \times d})\), satisfying the following adjoint system:
\begin{equation}
    \left\{
        \begin{aligned}
            &dY_t^{i,j} = -\partial_{x_j} H^{N,i}_t(\bs{X}_t^{\hat{\alpha}}, \bs{Y_t}^{i,:}, \bs{Z_t}^{i,:}, \bs{Z_t}^{0,i,:},          
                \bs{\hat{\alpha}}_t)\, dt 
                + \sum_{k=1}^N Z_t^{i,j,k}\, dW_t^k + Z_t^{0,i,j}\, dW_t^0, \\
            &Y_T^{i,j} = \partial_{x_j} g^{N,i}(\bs{X}_T^{\hat{\alpha}}),
            \label{eq:N-player_backward}
        \end{aligned}
    \right.
\end{equation}
where we denote 
\[
\bs{Y}_t^{i,:} := (Y_t^{i,j})_{j =1,\ldots,N}, \quad 
\bs{Z}_t^{i,:} := (Z_t^{i,j,j})_{j=1,\ldots,N}, \quad 
\bs{Z}_t^{0,i,:} := (Z_t^{0,i,j})_{j=1,\ldots,N}.
\]

Moreover, it holds that
\[
\partial_{\alpha^i} H^{N,i}_t(\bs{X}_t^{\hat{\alpha}}, \bs{Y}_t^{i,:}, \bs{Z}_t^{i,:}, \bs{Z}_t^{0,i,:}, \bs{\hat{\alpha}}_t) \cdot (\beta_t^i - \hat{\alpha}_t^i) \geq 0, \quad dt \otimes \mathbb{P}\text{-a.e.},
\]
for any \(\beta^i \in \mathbb{A}_N\) and all \(i = 1,2,\ldots, N\).
\end{lemma}
\begin{proof}
Since \(\bs{\hat{\alpha}}\) is a NE, the definition of NE implies that \(\hat{\alpha}^i\) satisfies the following optimality condition:
\[
J(\hat{\alpha}^i; \bs{\hat{\alpha}}^{-i}) = \inf_{\beta \in \mathbb{A}_N} J(\beta; \bs{\hat{\alpha}}^{-i}).
\]
Therefore, we may apply the standard stochastic maximum principle argument; see, for example, \cite[Theorem~2.15]{carmona2018probabilistic_I}.
\end{proof}

For later use, we compute the driver of the above adjoint equation:
\begin{align*}
    \partial_{x_j} H^{N,i}_t( \bs{x},  \bs{y}, \bs{z}, \bs{z}^{0}, \bs{\alpha}) 
    &=
        \delta_{i,j} \partial_x f_t(x^i, \alpha^i, L^N(\bs{x}, \bs{\alpha})) 
            + \frac{1}{N} \partial_\mu f_t(x^i, \alpha^i, L^N(\bs{x}, \bs{\alpha}))(x^j) \\
        &\quad + \partial_x b_t(x^j, \alpha^j, L^N(\bs{x}, \bs{\alpha})) \cdot y^{j}
            + \frac{1}{N} \sum_{k=1}^N \partial_\mu b_t(x^k, \alpha^k, L^N(\bs{x}, \bs{\alpha}))(x^j) \cdot y^{k} \\
        &\quad + \partial_x \sigma_t(x^j, \alpha^j, L^N(\bs{x}, \bs{\alpha})) \cdot z^{j}
            + \frac{1}{N} \sum_{k=1}^N \partial_\mu \sigma_t(x^k, \alpha^k, L^N(\bs{x}, \bs{\alpha}))(x^j) \cdot z^{k} \\
        &\quad + \partial_x \sigma^0_t(x^j, \alpha^j, L^N(\bs{x}, \bs{\alpha})) \cdot z^{0,j}
        + \frac{1}{N} \sum_{k=1}^N \partial_\mu \sigma^0_t(x^k, \alpha^k, L^N(\bs{x}, \bs{\alpha}))(x^j) \cdot z^{0,k}.
\end{align*}
We also compute the terminal condition:
\begin{align*}
\partial_{x_j} g^{N,i}(\bs{x}) = \delta_{i,j} \partial_x g(x^i, L^N(\bs{x})) + \frac{1}{N} \partial_\mu g(x^i, L^N(\bs{x}))(x^j).
\end{align*}
For the computation of the derivative of functions of empirical measures, see~\cite[Proposition~5.35]{carmona2018probabilistic_I}.

\subsection{The maximum principle for the mean field game}
Next, we provide the necessary and sufficient conditions for the existence of an MFE. The following result is standard; see, for example, \cite[Theorem~2.15 and~2.16]{carmona2018probabilistic_I}.
\begin{lemma}
\label{lemma:SMP_MFG}
Let Assumptions~\ref{assump:standard_assump} and~\ref{assump:Hamiltonian} hold, and let \((\hat{\alpha}, \hat{\xi})\) be an MFE such that $\E\int_0^T M^2_2(\hat{\xi}_t) \, dt  < \infty$. 
Then there exists a unique solution \((X^{\hat{\alpha}, \hat{\xi}}, Y, Z, Z^0)\), where \(X^{\hat{\alpha}, \hat{\xi}} \in \mathbb{S}^2(\mathbb{F}; \R^n)\), \(Y \in \mathbb{S}^2(\mathbb{F}; \R^n)\), \(Z \in \mathbb{H}^2(\mathbb{F}; \R^{n \times d})\), and \(Z^0 \in \mathbb{H}^2(\mathbb{F}; \R^{n \times d})\), to the FBSDE system:
\begin{equation}
\left\{
    \begin{aligned}
        & d X_t^{\hat{\alpha}, \hat{\xi}} = b_t(X_t^{\hat{\alpha}, \hat{\xi}}, \hat{\alpha}_t, \hat{\xi}_t)\, dt 
            + \sigma_t(X_t^{\hat{\alpha}, \hat{\xi}}, \hat{\alpha}_t, \hat{\xi}_t)\, dW_t 
            + \sigma^0_t(X_t^{\hat{\alpha}, \hat{\xi}}, \hat{\alpha}_t, \hat{\xi}_t)\, dW_t^0, 
            \quad  X_0^{\hat{\alpha}, \hat{\xi}} = X^1_0, 
        \\
        &dY_t = -\partial_x H_t(X_t^{\hat{\alpha}, \hat{\xi}},  Y_t, Z_t, Z_t^0, \hat{\alpha}_t, \hat{\xi}_t)\, dt  +  Z_t\, dW_t + Z_t^0\, dW_t^0, \quad Y_T = \partial_x g(X_T^{\hat{\alpha}, \hat{\xi}}, \hat{\mu}_T),
    \end{aligned}
\right.
\label{eq:mean_field_system}
\end{equation}
and it holds that for any \(\beta \in \R^{\ell}\),
\begin{align}
\label{eq:hamiltonian_optimality}
H_t( X_t^{\hat{\alpha}, \hat{\xi}},  Y_t, Z_t, Z_t^0, \beta,\hat{\xi}_t) 
\geq H_t( X_t^{\hat{\alpha}, \hat{\xi}}, Y_t, Z_t, Z_t^0,  \hat{\alpha}_t, \hat{\xi}_t), 
\quad dt \otimes \mathbb{P}\text{-a.e.}
\end{align}

Conversely, suppose there exists a solution to the system~\eqref{eq:mean_field_system} for \(\hat{\alpha} \in \mathbb{A}\) and for an \(\mathcal{P}_2(\R^n \times \R^{\ell})\)-valued, \(\mathbb{F}^0\)-progressively process \(\xi^{\hat{\alpha}}\) with $\E \int_0^T M^2_2(\xi^{\hat{\alpha}}_t) \, dt  < \infty$, such that \(\hat{\alpha}\) satisfies~\eqref{eq:hamiltonian_optimality} and \(\xi_t^{\hat{\alpha}} = \mathcal{L}^1(X_t^{\hat{\alpha}}, \hat{\alpha}_t)\), \(\mathbb{P}^0\)-a.s.\ for all \(t \in [0, T]\). Then \((\hat{\alpha}, \xi^{\hat{\alpha}})\) is an MFE.
\end{lemma}

Finally, we state the well-posedness of the MFG.
\begin{theorem}
\label{thm:unique_existence_MFE}
Let Assumptions \ref{assump:standard_assump}, \ref{assump:Hamiltonian}, \ref{assump:displacement_monotonicity} hold. Then, the MFG admits a unique MFE.
\end{theorem}
Proof of this theorem is provided at the end of Section~\ref{sec:well_posedness_cond_MKVFBSDE} using the well-posedness result for the following conditional MKV-FBSDE:
\begin{equation}
\label{eq:mean_field_MKVFBSDE}
\left\{
\begin{aligned}
&d X_t = B_t\left(\Theta_t, \mathcal{L}^1 (\Theta_t)\right)\, dt 
+ \Sigma_t\left(\Theta_t, \mathcal{L}^1 (\Theta_t)\right) \, dW_t 
+ \Sigma^0_t\left(\Theta_t, \mathcal{L}^1 (\Theta_t)\right) \, dW^0_t 
\\
&dY_t =  -F_t\left(\Theta_t, \mathcal{L}^1 (\Theta_t)\right) \, dt 
+ Z_t \, dW_t + Z^{0}_t \, dW^0_t 
\\
& X_0  = X^1_0, \hspace{1cm} Y_T = G\left(X_T, \mathcal{L}^1 (X_T)\right),
\end{aligned}
\right.
\end{equation}
where $\Theta = (X, Y, Z, Z^0)$.

\section{Proof of Theorem \ref{thm:main}}
\label{sec:proof_of_main_theorem}
Throughout this section, we adhere to the same assumptions and notation as in Theorem~\ref{thm:main}.  
In particular, we assume that Assumptions~\ref{assump:standard_assump}, \ref{assump:Hamiltonian}, \ref{assump:displacement_monotonicity}, and \ref{assump:smoothness_L-derivative} are satisfied.
Furthermore, let \(\bs{\alpha}^N \in \mathbb{A}_N^N\) be an \(N\)-player NE, which is assumed to exist.  
Denote the corresponding state process by \((X^{N,1}, \ldots, X^{N,N}) \in (\mathbb{S}^2(\mathbb{F}^N; \R^n))^N\).  
The associated adjoint system, as given in Lemma~\ref{lemma:smp_N_player}, is denoted by  
\begin{align*}
&(Y^{N,i,j})_{(i,j)\in \{ 1, \ldots N \}^2} \in (\mathbb{S}^2(\mathbb{F}^N;\R^n))^{N\times N}, \\
&(Z^{N,i,j,k})_{(i,j,k)\in \{ 1, \ldots N \}^3} \in (\mathbb{H}^2(\mathbb{F}^N;\R^{n \times d}))^{N^3}, \\
&(Z^{0,N,i,j})_{(i,j)\in \{ 1, \ldots N \}^2} \in (\mathbb{H}^2(\mathbb{F}^N;\R^{n\times d}))^{N^2},    
\end{align*}
and we denote 
\[
\Theta^{N,i} = (X^{N,i}, Y^{N,i,i}, Z^{N,i,i,i}, Z^{0,N,i,i}),
\quad 
\bs{\Theta}^{N} = (\Theta^{N,1}, \ldots, \Theta^{N,N}).
\]
Moreover, we define the process \((\zeta_t^i)_{t \in [0,T]}\) by
\begin{equation}
    \begin{split}
        \zeta^i_t &:= \frac{1}{N} \, \partial_{\nu} f^{(2)}_t \big( X^{N,i}_t, L^N(\bs{X}^N_t, \bs{\alpha}^N_t)\big)(\alpha^{N,i}_t) \\
            &\quad+ \frac{1}{N} \sum_{j=1}^N \partial_{\nu} b^{(2)}_t \big(X^{N,j}_t, L^N(\bs{X}^N_t, \bs{\alpha}^N_t)\big)(\alpha^{N,i}_t) \cdot Y^{N,i,j}_t \\ 
            &\quad + \frac{1}{N} \sum_{j=1}^N \partial_{\nu} \sigma^{(2)}_t \big( X^{N,j}_t, L^N(\bs{X}^N_t, \bs{\alpha}^N_t)\big)(\alpha^{N,i}_t) \cdot Z^{N,i,j,j}_t \\
            &\quad + \frac{1}{N} \sum_{j=1}^N \partial_{\nu} \sigma^{0,(2)}_t \big(X^{N,j}_t, L^N(\bs{X}^N_t, \bs{\alpha}^N_t)\big)(\alpha^{N,i}_t) \cdot Z^{0,N,i,j}_t.
    \end{split}
\label{eq:zeta_def}
\end{equation}
Then, using the optimizer \(\Lambda\) from Lemma~\ref{lemma:optimizer_of_Hamiltonian} and Proposition~\ref{lemma:smp_N_player}, we obtain
\[
\alpha^{N,i}_t = \Lambda_t\big( \Theta^{N,i}_t, L^N(\bs{X}_t^N), \zeta^i_t\big), \quad dt \otimes \mathbb{P}^N\text{-a.e.} 
\]
This follows from the strict convexity of the Hamiltonian \(H\) (more precisely, of \(H + a \cdot \zeta\)) with respect to $a$.

In light of Lemma~\ref{lemma:smp_N_player}, $\bs{\Theta}^N$ satisfies the following system:
\begin{equation}
    \label{eq:N-player_systems_FBSDEs}
    \left\{
    \begin{aligned}
        &d X^{N,i}_t = B^i_t\left(\bs{\Theta}^N_t, \bs{\zeta}_t\right)\, dt 
        + \Sigma^i_t\left(\bs{\Theta}^N_t, \bs{\zeta}_t\right) \, dW^i_t 
        + \Sigma^{0,i}_t\left(\bs{\Theta}^N_t, \bs{\zeta}_t\right) \, dW^0_t ,
        \\
        &X^{N,i}_0  = X^{i}_0, \\
        &dY^{N,i,i}_t =  -\Bigg[F^i_t\left(\bs{\Theta}^N_t, \bs{\zeta}_t\right)
        + \frac{1}{N} 
                \partial_{\mu} H^{N,i}_t \bigl( 
                    \bs{X}^N_t, \bs{Y}^{N,i,:}_t, \bs{Z}^{N,i,:}_t, \bs{Z}^{0,N,i,:}_t, \bs{\alpha}^N_t
                \bigr)( X_t^{N,i} )\,
                \Bigg]dt 
        \\
        &\qquad \qquad \qquad + \sum^N_{k=1} Z^{N,i,i,k}_t \, dW^i_t + Z^{0,i,i}_t \, dW^0_t ,
        \\
        &  Y^{N,i}_T = G\left(X^{N,i}_T, L^N (\bs{X}^{N}_T)\right) 
        + \frac{1}{N} \partial_{\mu} g\left(
            X^{N,i}_T, L^N (\bs{X}^{N}_T) 
            \right)( X^{N,i}_T ),
    \end{aligned}
\right.
\end{equation}
where we denote
\begin{equation}
\label{eq:tilde_coefficients}
\begin{aligned}
{B}^i_t(\boldsymbol{\theta},\boldsymbol{\zeta})
 &:= b_t\big(x^i, \Lambda_t(\theta^i, L^N(\bs{x}), \zeta^i), L^N(\bs{x}, \boldsymbol{\Lambda}_t)\big), \\
{\Sigma}^i_t(\boldsymbol{\theta},\boldsymbol{\zeta})
 &:= \sigma_t\big(x^i, \Lambda_t(\theta^i, L^N(\bs{x}), \zeta^i), L^N(\bs{x}, \boldsymbol{\Lambda}_t)\big), \\
{\Sigma}^{0,i}_t(\boldsymbol{\theta},\boldsymbol{\zeta})
 &:= \sigma^0_t\big( x^i, \Lambda_t(\theta^i, L^N(\bs{x}), \zeta^i), L^N(\bs{x}, \boldsymbol{\Lambda}_t)\big), \\
{F}^i_t(\boldsymbol{\theta},\boldsymbol{\zeta})
 &:= \partial_x H_t\big( x^i, \Lambda_t(\theta^i, L^N(\bs{x}), \zeta^i), L^N(\bs{x}, \boldsymbol{\Lambda}_t)\big).
\end{aligned}
\end{equation}
for 
\[
\theta^i = (x^i, y^i, z^i, z^{0,i}), \quad  \bs{\theta} = (\theta^1, \ldots, \theta^N), \quad \bs{x} = (x^1, \ldots, x^N), \quad \boldsymbol{\zeta} = (\zeta^1, \ldots, \zeta^N),
\]
and 
\[
\bs{\Lambda}_t = (\Lambda_t(\theta^1, L^N(\bs{x}), \zeta^1), \ldots, \Lambda_t(\theta^N, L^N(\bs{x}), \zeta^N)).
\]

Next, we denote by \(\Theta := (X, Y, Z, Z^0)\) the solution of the conditional MKV-FBSDE \eqref{eq:mean_field_MKVFBSDE}, defined on \((\Omega, \mathcal{F}, \mathbb{P})\).  
By (the proof of) Theorem \ref{thm:unique_existence_MFE} and \ref{thm:wellposedness_conditional_MKVFBSDE}, the unique existence of such a solution is guaranteed, and the pair 
\[
\alpha_t = \Lambda_t( \Theta_t, \mathcal{L}^1(X_t), 0), \,\, \text{and $\mathbb{F}^0$-progressively measurable  process}\,\,
\left(\mathcal{L}^1(X_t, \alpha_t ) \right)_{t\in[0,T]}, 
\]
is the unique MFE. 
Therefore, as in Lemma~\ref{lemma:iid_copies_solves_MKVFBSDE} and the preceding discussion, we can construct a sequence of stochastic processes \((\Theta^i = (X^i, Y^i, Z^i, Z^{0,i}))_{i \in \{1, \ldots, N\}}\) such that, for each \(i\),  $\Theta^i \in \mathbb{S}^2(\mathbb{F}^N;\R^n) \times \mathbb{S}^2(\mathbb{F}^N;\R^n) \times \mathbb{H}^2(\mathbb{F}^N;\R^n) \times \mathbb{H}^2(\mathbb{F}^N;\R^n)$ is the unique solution to~\eqref{eq:mean_field_MKVFBSDE} on \((\Omega^N, \mathcal{F}^N, \mathbb{P}^N)\) where $(X_0, W^0, W)$ is replaced by $(X^i_0, W^0, W^i)$.
Moreover, for each fixed \(t \in [0, T]\), the collection \((\Theta_t^i(\omega^0, \cdot, \cdot))_{i \in \{1, \ldots, N\}}\) forms an i.i.d.\ sequence of random variables on \((\Omega^{1,N}, \mathcal{F}^{1,N}, \mathbb{P}^{1,N})\) for \(\mathbb{P}^0\)-almost all $\omega^0$, whose common law coincides with \(\mathcal{L}^1(\Theta)\). We also define 
\[
\alpha_t^i := \Lambda_t ( \Theta^i_t, \mathcal{L}^1(X_t^i), 0),
\]
and denote 
\[
\bs{\Theta} = (\Theta^1, \ldots, \Theta^N).
\]

We begin with the following lemma as introduced in \cite[Lemma~6]{jackson2024quantitative}.
\begin{lemma}
Let Assumption \ref{assump:displacement_monotonicity} hold. 
Then we have
\begin{align*}
&\sum_{i=1}^N \Big(
 -\big( F_t(\theta^i, L^N(\bs{\theta}) )
      - F_t(\tilde{\theta}^i, L^N(\bs{\tilde{\theta}}) ) \big)
      \cdot (x^i - \tilde{x}^i)  \\[-8pt]
&\qquad
 + \big( B_t(\theta^i, L^N(\bs{\theta}) )
      - B_t(\tilde{\theta}^i, L^N(\bs{\tilde{\theta}}) ) \big)
      \cdot (y^i - \tilde{y}^i) \\
&\qquad
 + \big( \Sigma_t(\theta^i, L^N(\bs{\theta}) )
      - \Sigma_t(\tilde{\theta}^i, L^N(\bs{\tilde{\theta}}) ) \big)
      \cdot (z^i - \tilde{z}^i) \\
&\qquad
 + \big( \Sigma^0_t(\theta^i, L^N(\bs{\theta}) )
      - \Sigma^0_t(\tilde{\theta}^i, L^N(\bs{\tilde{\theta}}) ) \big)
      \cdot (z^{0,i} - \tilde{z}^{0,i})
\Big)  \leq  -C_H \sum_{i=1}^N |x^i - \tilde{x}^i|^2 .
\end{align*}
where
\[
\theta^i = (x^i, y^i, z^i, z^{0,i}), \qquad
\bs{\theta} = (\theta^1,\ldots,\theta^N),
\]
and
\[
\tilde{\theta}^i = (\tilde{x}^i, \tilde{y}^i, \tilde{z}^i, \tilde{z}^{0,i}), \qquad
\bs{\tilde{\theta}} = (\tilde{\theta}^1,\ldots,\tilde{\theta}^N).
\]

Moreover
\begin{align*}
\sum^N_{i=1} (G(x^i, L^N(\bs{x})) - G(\tilde{x}^i, L^N(\bs{\tilde{x}})) )  \cdot (x^i - \tilde{x}^i)
\geq C_G \sum^N_{i=1} |x^i - \tilde{x}^i|^2 .
\end{align*}
\end{lemma}
\begin{proof}
Take random variables $\Theta = (X,Y,Z,Z^0)$ and $\tilde{\Theta} = (\tilde{X},\tilde{Y},\tilde{Z},\tilde{Z}^0)$ with 
\[
\mathbb{P} (\Theta = \theta^i,  \tilde{\Theta} = \tilde{\theta}^i ) = \mathbb{P} (X = x^i,  \tilde{X} = \tilde{x}^i )  =  \frac{1}{N},
\]
for all $i = 1,\ldots, N $. 
Apply Assumption~\ref{assump:displacement_monotonicity} to conclude.
\end{proof}

In the subsequent lemmas, we estimate the solution of the FBSDE systems for the $N$-player game.
\begin{lemma}
\label{lemma:moment_estimate_Y_Z_Z0_ii}
Let Assumption \ref{assump:standard_assump}, \ref{assump:Hamiltonian}, and \ref{assump:displacement_monotonicity} hold.
Then, for sufficiently large $N$, there exists a constant \(C > 0\), independent of \(i\), \(j\), and \(N\), such that
\begin{equation}
\begin{split}    \label{eq:moment_estimate_Y_Z_Z0_ii}
    &\sum^N_{i=1}\E\Bigg[ 
         |X^{N,i}_T| + 
        \int^T_0  |X^{N,i}_t|^2 + |Y^{N, i,i}_t|^2 +  \sum^N_{j=1}|Z_t^{N,i,i,j}|^2 + |Z_t^{0,N,i,i}|^2 \,dt 
    \Bigg] \\
    &\quad \leq 
     CN
    +  C  \sum^N_{i=1} \E\Bigg[
        \int^T_0|\zeta^i_t|^2 \, dt 
    \Bigg]  +\frac{C}{N^{1/2}}  \sum^N_{i=1} \E\Bigg[ 
        \int^T_0\sum^N_{j \neq i} (| Y^{N,i,j} |^2 +  | Z^{N,i,j,j } |^2 + | Z^{0,N,i,j } |^2) \, dt
    \Bigg]. 
\end{split}
\end{equation}
\end{lemma}

\begin{proof}
First, applying It\^o's product rule yields
\begin{align*}
    &d(X^{N,i} \cdot Y^{N,i,i})_t  \\
    &\quad =  X^{N,i}_t\, dY^{N,i,i}_t + Y^{N,i,i}_t\, dX^{N,i}_t +  d[X^{N,i}, Y^{N,i,i}]_t\\[4pt]
    &\quad = \bigg[-  F^i_t\left( \bs{\Theta}^N_t, \bs{\zeta}_t  \right) \cdot X^{N,i}_t 
        +  B^i_t\left( \bs{\Theta}^N_t, \bs{\zeta}_t  \right) \cdot Y^{N,i,i}_t 
        + \Sigma^i_t\left( \bs{\Theta}^N_t, \bs{\zeta}_t  \right) \cdot Z^{N,i,i,i}_t 
        + \Sigma^{0,i}_t\left( \bs{\Theta}^N_t, \bs{\zeta}_t  \right) \cdot Z^{0,N,i,i}_t \\
    &\qquad\qquad - \frac{1}{N} 
            \partial_{\mu} H^{N,i}_t \bigl( 
                \bs{X}^N_t, \bs{Y}^{N,i,:}_t, \bs{Z}^{N,i,:}_t, \bs{Z}^{0,N,i,:}_t, \bs{\alpha}^N_t
            \bigr)( X_t^{N,i} )  \cdot X^{N,i}_t  \bigg] \, dt 
            + dM^i_t,
\end{align*}
where $M^i$ is a $\mathbb{P}^N$-martingale.
Thus we get 
\footnote{
    Within the proofs, we use the shorthand notation $\sum_i = \sum^N_{i=1}$ and $\sum_{j \neq i} = \sum^N_{j \neq i}$. 
    The constant $C > 0$ varies line by line, unless it depends on $i$ or $N$.
}
\begin{align*}
&\E \Bigg[ \sum_i X^{N,i}_T \cdot Y^{N,i,i}_T \Bigg]  \\[2pt]
&\quad \leq \E \Bigg[ \sum_i X^{N,i}_0 \cdot Y^{N,i,i}_0 \Bigg] \\[2pt]
&\quad + \E \Bigg[ 
    \int^T_0 \sum_i \bigg(
        - \Big( F_t(\Theta^{N,i}_t, L^N(\bs{\Theta}^N_t)) - F_t(\boldsymbol{0}, \delta_0) \Big) \cdot X^{N,i}_t  
        + \Big( B_t(\Theta^{N,i}_t, L^N(\bs{\Theta}^N_t)) - B_t(\boldsymbol{0}, \delta_0) \Big) \cdot Y^{N,i,i}_t \\
&\hspace{3cm}      
        + \Big( \Sigma_t(\Theta^{N,i}_t, L^N(\bs{\Theta}^N_t)) - \Sigma_t(\boldsymbol{0}, \delta_0) \Big) \cdot Z^{N,i,i,i}_t 
        + \Big( \Sigma^0_t(\Theta^{N,i}_t, L^N(\bs{\Theta}^N_t)) - \Sigma^0_t(\boldsymbol{0}, \delta_0) \Big) \cdot Z^{0,N,i}_t  
    \bigg) dt 
\Bigg] \\[2pt]
&\quad + \E \Bigg[ 
    \int^T_0 \sum_i \bigg(
        - \Big( F^i_t(\bs{\Theta}^N_t, \bs{\zeta}_t) - F_t(\Theta^{N,i}_t, L^N(\bs{\Theta}^N_t)) \Big) \cdot X^{N,i}_t   
        + \Big( B^i_t(\bs{\Theta}^N_t, \bs{\zeta}_t) - B_t(\Theta^{N,i}_t, L^N(\bs{\Theta}^N_t)) \Big) \cdot Y^{N,i,i}_t \\[2pt]
&\hspace{3cm}      
        + \Big( \Sigma^i_t(\bs{\Theta}^N_t, \bs{\zeta}_t) - \Sigma_t(\Theta^{N,i}_t, L^N(\bs{\Theta}^N_t)) \Big) \cdot Z^{N,i,i,i}_t 
+ \Big( \Sigma^{0,i}_t(\bs{\Theta}^N_t, \bs{\zeta}_t) - \Sigma^0_t(\Theta^{N,i}_t, L^N(\bs{\Theta}^N_t)) \Big) \cdot Z^{0,N,i}_t  
    \bigg) dt 
\Bigg] \\[2pt]
&\quad + \E \Bigg[ 
    \int^T_0 \sum_i \bigg(
        - F_t(\boldsymbol{0}, \delta_0) \cdot X^{N,i}_t  
        + B_t(\boldsymbol{0}, \delta_0) \cdot Y^{N,i,i}_t    
        + \Sigma_t(\boldsymbol{0}, \delta_0) \cdot Z^{N,i,i,i}_t  
        + \Sigma^0_t(\boldsymbol{0}, \delta_0) \cdot Z^{0,N,i}_t  
    \bigg) dt 
\Bigg] \\[2pt]
&\quad - \E \Bigg[ 
    \int^T_0 \sum_i \frac{1}{N} 
        \partial_{\mu} H^{N,i}_t \bigl( 
            \bs{X}^N_t, \bs{Y}^{N,i,:}_t, \bs{Z}^{N,i,:}_t, \bs{Z}^{0,N,i,:}_t, \bs{\alpha}^N_t
        \bigr)( X_t^{N,i} ) \cdot X^{N,i}_t \, dt 
\Bigg].
\end{align*}
Then, by applying the monotonicity condition in Assumption~\ref{assump:displacement_monotonicity},  
the boundedness of \(\partial_{\mu} b\), \(\partial_{\mu} \sigma\), and \(\partial_{\mu} \sigma^0\),  
the Lipschitz continuity and linear growth conditions of $b,\sigma, \sigma^0,\partial_x H$,  
together with the \(\varepsilon\)-Young inequality for an arbitrary \(\varepsilon > 0\), we obtain
\begin{equation}
\begin{split}
    &\E \Bigg[ \sum_i X^{N,i}_T \cdot Y^{N,i,i}_T \Bigg]\\
    &\leq C \, \E \Bigg[ 
        \frac{1}{\varepsilon} \sum_i |X^{N,i}_0|^2 +  \varepsilon \sum_i |Y^{N,i,i}_0|^2 - \sum_i \int_0^T |X_t^{N,i}|^2 \, dt 
    \Bigg]  \\
    &\quad + \varepsilon C \, \E \Bigg[
        \int_0^T \sum_i \big( |X^{N,i}_t|^2 + |Y^{N,i,i}_t|^2 + |Z^{N,i,i,i}_t|^2 + |Z^{0,N,i,i}_t|^2 \big) \, dt 
    \Bigg]
    + \frac{C}{\varepsilon} \Bigg( N + \E \Bigg[ \sum_i  \int_0^T |\zeta^i_t|^2 \, dt \Bigg] \Bigg)  \\
    &\quad  + C \, \E \Bigg[ \int_0^T ( \frac{1}{N^{1/2}} + \frac{1}{N} ) \sum_i |X^{N,i}_t|^2 
        + \frac{1}{N^{1/2}} \sum_i \sum_{j} \left( |Y_t^{N,i,j}|^2 + |Z_t^{N,i,j,j}|^2 + |Z_t^{0,N,i,j}|^2 \right) dt 
    \Bigg].
    \label{eq:running_X^Ni_Y^Nii_cross_term_estimate}
\end{split}
\end{equation}
Similarly, we can estimate $X^{N,i}_T \cdot Y^{N,i,i}_T$ by using the monotonicity of $G$, the linear growth condition for $\partial_{\mu}g$, and $\varepsilon$-Young's inequality,
\begin{align}
    \E\left[ \sum_i X^{N,i}_T \cdot Y^{N,i,i}_T  \right] 
    &= \E\left[ 
           \sum_i X^{N,i}_T \cdot G (X^{N,i}_T, L^N(\bs{X}_T^N))
           + \frac{1}{N} \sum_i X^{N,i}_T \cdot \partial_{\mu}g(X^{N,i}_T, L^N(\bs{X}_T^N))(X^{N,i}_T)  
       \right] \notag\\
    &= \E\Bigg[ 
           \sum_i X^{N,i}_T \cdot \left(G (X^{N,i}_T, L^N(\bs{X}_T^N)) -G(\mathbf{0},\delta_{\mathbf{0}})\right) 
           + \sum_i X^{N,i}_T \cdot G (\mathbf{0},\delta_{\mathbf{0}}) \notag \\
           &\qquad + \frac{1}{N}\sum_i X^{N,i}_T \cdot \partial_{\mu}g(X^{N,i}_T, L^N(\bs{X}_T^N))(X^{N,i}_T)
       \Bigg] \notag\\
    & \geq 
       \E\bigg[ 
           C_g \sum_i|X^{N,i}_T|^2 \bigg] 
           - C \E\bigg[  
               \varepsilon \sum_i|X^{N,i}_T|^2  -  \frac{1}{N} \sum_i|X^{N,i}_T|^2  \bigg] 
           - C ( 1 + \frac{N}{\varepsilon} ) . 
    \label{eq:terminal_X^Ni_Y^Nii_cross_term_estimate}
\end{align}
Thus, combining \eqref{eq:running_X^Ni_Y^Nii_cross_term_estimate} and \eqref{eq:terminal_X^Ni_Y^Nii_cross_term_estimate}, we get
\begin{equation}
\begin{split}
    \E\Bigg[ \sum_{i}|X^{N,i}_T|^2  + \int^T_0 \sum_{i} |X^{N,i}_t|^2 dt \Bigg]  
    &\leq \varepsilon C\E\Bigg[ \sum_{i} |Y^{N,i,i}_0|^2 \Bigg] 
          + (\varepsilon+\frac{1}{N}) C \E\Bigg[ \sum_{i}|X^{N,i}_T|^2 \Bigg] \\
    &\quad + (\varepsilon + \frac{1}{N^{1/2}}  +\frac{1}{N}) C \E\Bigg[
             \int^T_0  \sum_{i} (|X^{N,i}_t|^2 + |Y^{N,i,i}_t| + | Z^{N,i,i,i}_t| ^2 + |Z^{0,N,i,i}_t| ^2 ) \, dt
             \Bigg] \\
    &\quad + \frac{C}{N^{1/2}}  \E\Bigg[
             \int^T_0  \sum_{i}\sum_{j\neq i} ( |Y^{N,i,j}_t| + | Z^{N,i,j,j}_t| ^2 + |Z^{0,N,i,j}_t| ^2 ) \, dt
             \Bigg]  \\
    &\quad + \frac{C}{\varepsilon} \Bigg(
             N + \E \Bigg[ \int^T_0  \sum_{i} |\zeta^i_t|^2 \, dt\Bigg]
             \Bigg),
    \label{eq:estimate_for_sum_X^Ni}
\end{split}
\end{equation}
where we used the fact that $(X_0^{N,i})_{i=1,\ldots,N}$ is an  i.i.d.\ sample from $\mu_0\in\mathcal{P}_2(\R^n)$.
Next, applying Itô's formula to \(\sum_i|Y^{N,i,i}|^2\) and taking expectation, we obtain
\begin{align}
    &\E \Bigg[ \sum_i |Y^{N, i,i}_t|^2 + \int^T_t \sum_i\sum_{k}|Z_s^{N,i,i,k}|^2 + \sum_i|Z_s^{0,N,i,i}|^2 \,ds \Bigg] \notag \\[2pt]
    &\leq \E \Bigg[ \sum_i|Y_T^{N,i}|^2 \Bigg] 
          + C \E \Bigg[ \int^T_t \sum_i|Y^{N,i,i}_s| \cdot |F^i_s(\bs{\Theta}^N_s, \bs{\zeta}_s)| \notag \\
    &\qquad + \frac{1}{N} \sum_i |Y_s^{N,i,i} | \cdot |\partial_{\mu} H^{N,i}_s \bigl( \bs{X}^N_s, \bs{Y}^{N,i,:}_s, \bs{Z}^{N,i,:}_s, \bs{Z}^{0,N,i,:}_s, \bs{\alpha}^N_s \bigr)( X_s^{N,i} )| \, ds \Bigg] \notag \\[4pt]
    &\leq\mathbb{E} \Bigg[ \sum_i|Y_T^{N,i}|^2 \Bigg]
          + C(1+ \eta + \frac{1}{\eta} ) \mathbb{E} \Bigg[ \int^T_t \sum_i |Y^{N,i,i}_s|^2 \,ds \Bigg] 
          \label{eq:estimate_for_sum_Y^Ni_Z^Ni_Z^0Ni}\\
        &\quad + (\eta + \frac{1}{N})C\mathbb{E} \Bigg[ \int^T_t N + \sum_i |X^{N,i}_s|^2 + \sum_i|\zeta^i_s|^2\,ds \Bigg] 
        \notag \\
        &\quad + \frac{C}{N}\mathbb{E} \Bigg[ \int^T_t \sum_i\sum_{j \neq i} (| Y^{N,i,j}_s |^2 + | Z^{N,i,j,j}_s |^2 + | Z^{0,N,i,j }_s |^2) \,ds \Bigg] 
        \notag\\
        &\quad + (\eta + \frac{1}{N})C \mathbb{E} \Bigg[ \int^T_t \sum_i (|Z^{N,i,i,i }_s|^2 + |Z^{0:N,i,i}_s|^2 ) \,ds \Bigg],
        \notag
\end{align}
where we used the boundedness of $\partial_{\mu}b, \partial_{\mu}\sigma$, and $\partial_{\mu}\sigma^0$, the linear growth of $\partial_x H,  \partial_{\mu}f$ and $\Lambda$, and $\eta$-Young's inequality for an arbitrary $\eta>0$. Thus, taking $\eta$ sufficiently small and $N$ large, we get 
\begin{align*}
    \E\Bigg[ \sum_i |Y^{N, i,i}_t|^2\Bigg] 
    &\leq \E\Bigg[ \sum_i|Y_T^{N,i}|^2 \Bigg]
          + C(1+ \eta + \frac{1}{\eta} ) \E\Bigg[ \int^T_t \sum_i |Y^{N,i,i}_s|^2 \,ds \Bigg] \\
    &\quad + (\eta + \frac{1}{N})C\E\Bigg[ \int^T_0 N + \sum_i |X^{N,j}_s|^2 + \sum_i|\zeta^i_s|^2\,ds \Bigg] \\[2pt]
    &\quad + \frac{C}{N}\E\Bigg[ 
             \int^T_0 \sum_i\sum_{j \neq i} (| Y^{N,i,j}_s |^2 + | Z^{N,i,j,j}_s |^2 + | Z^{0,N,i,j }_s |^2) \,ds 
             \Bigg], \quad \text{for any $t \in [0,T]$}.
\end{align*}
Apply Gronwall's inequality to get
\begin{align*}
    \mathbb{E}\Bigg[ \sum_i |Y^{N, i,i}_t|^2\Bigg]
    &\leq CN + C \E\Bigg[ \sum_i|X_T^{N,i}|^2 \Bigg]\\
    &\quad + (\eta + \frac{1}{N})C\mathbb{E}\Bigg[ \int^T_0 N + \sum_i |X^{N,j}_s|^2 + \sum_i|\zeta^i_s|^2\,ds \Bigg] \\[2pt]
    &\quad + \frac{C}{N}\mathbb{E}\Bigg[ 
             \int^T_0 \sum_i\sum_{j \neq i} (| Y^{N,i,j}_s |^2 + | Z^{N,i,j,j}_s |^2 + | Z^{0,N,i,j }_s |^2) \,ds 
             \Bigg].
\end{align*}
Plug this into \eqref{eq:estimate_for_sum_Y^Ni_Z^Ni_Z^0Ni}, and then \eqref{eq:estimate_for_sum_X^Ni}, and take $\varepsilon$ sufficiently small and $N$ large, to obtain 
\begin{equation}
\begin{split}
    &\mathbb{E}\Bigg[ \int^T_0 \sum_{i} |X^{N,i}_t|^2 \,dt + \sum_i|X_T^{N,i}|^2 \Bigg] \\
    &\leq NC + C \mathbb{E}\Bigg[ \int^T_0 \sum_{i} |\zeta^i_t|^2 \, dt \Bigg]  + C\mathbb{E}\Bigg[
             \int^T_0 \frac{1}{N^{1/2}} \sum_{i} \sum_{j\neq i} \left( | Y^{N,i,j} |^2 + | Z^{N,i,j,j } |^2 + | Z^{0,N,i,j } |^2 \right) \, dt
             \Bigg],
\end{split}
\end{equation}
where $C$ does not depend on $i, j$, and $N$.
From this estimate, the result follows immediately.
\end{proof}

\begin{lemma}
\label{lemma:moment_estimate_Y_Z_Z0_iJ}
Let Assumption \ref{assump:standard_assump}, \ref{assump:Hamiltonian}, and \ref{assump:displacement_monotonicity} hold.
Then, for sufficiently large $N$, we have
\begin{equation}
\begin{split}
    &\mathbb{E}\left[
        \int^T_0 \sum^N_{i=1} \sum^N_{i\neq j} \left(| Y_t^{N,i,j} |^2 +  \sum^N_{k=1} |Z_t^{N,i,j,k} |^2 + | Z_t^{0,N,i,j } |^2\right) \, dt
    \right] 
    \leq C +  \frac{C}{N} \mathbb{E}\left[\int^T_0 
            \sum^N_{i=1} |\zeta^i_t|^2 \, dt 
        \right],
\end{split}
\end{equation}
for some constant $C>0$, which is independent of $i,j,$ and $N$.
\end{lemma}
\begin{proof}
Fix $i=1$, for example, and define 
\[
    Y_t^{-1} := (Y_t^{N,1,2}, \ldots, Y_t^{N,1,N}), \quad 
    Z_t^{-1} := (Z_t^{N,1,2,2}, \ldots, Z_t^{N,1,N,N}), \quad 
    Z_t^{0,-1} := (Z_t^{0,N,1,2}, \ldots, Z_t^{0,N,1,N}), 
\]
\[
    P_t := \left( \partial_\mu f_t \left(  X_t^{N,1}, \alpha_t^{N,1}, L^N(\bs{X}_t^N, \bs{\alpha}_t^N) \right)(X_t^{N,j}) \right)_{j=2,\ldots,N}, 
\]
and for $\varphi = b, \sigma, \sigma^0$,
\begin{align*}
    Q^{\varphi}_t &:= \left( \partial_\mu \varphi_t\left(  X_t^{N,1}, \alpha_t^{N,1}, L^N(\bs{X}_t^N, \bs{\alpha}_t^N) \right)(X_t^{N,j}) \right)_{j=2,\ldots,N}, \\
    R^{\varphi}_t &:= \left( \partial_\mu \varphi_t\left( X_t^{N,m}, \alpha_t^{N,m}, L^N(\bs{X}_t^N, \bs{\alpha}_t^N) \right)(X_t^{N,j}) \right)_{m,j=2,\ldots,N}, \\
    S^{\varphi}_t &:= \operatorname{diag} \left( \partial_x \varphi_t\left(X_t^{N,j}, \alpha_t^{N,j}, L^N(\bs{X}_t^N, \bs{\alpha}_t^N) \right) \right)_{j=2,\ldots,N}.
\end{align*}
Recall for $i\neq j$, the process $Y^{i,j}_t$ satisfies
\begin{equation*}
    \begin{split}
        -dY_t^{i,j} 
        &= \Bigg[
            \frac{1}{N} \partial_\mu f_t\left( X_t^{N,i}, \alpha_t^{N,i}, L^N(\bs{X}_t^N, \bs{\alpha}_t^N) \right)(X_t^{N,j}) \\
            &\qquad + \partial_x b_t\left( X_t^{N,j}, \alpha_t^{N,j}, L^N(\bs{X}_t^N, \bs{\alpha}_t^N) \right) \cdot Y_t^{N,i,j}\\
            &\qquad + \partial_x \sigma_t\left( X_t^{N,j}, \alpha_t^{N,j}, L^N(\bs{X}_t^N, \bs{\alpha}_t^N) \right) \cdot Z_t^{N,i,j,j} \\
            &\qquad + \partial_x \sigma^0_t\left( X_t^{N,j}, \alpha_t^{N,j}, L^N(\bs{X}_t^N, \bs{\alpha}_t^N) \right) \cdot Z_t^{0,N,i,j}
            \Bigg] \, dt \\
            &\quad + \frac{1}{N} \sum_{k} \Bigg[ 
                \partial_\mu b_t\left( X_t^{N,k}, \alpha_t^{N,k}, L^N(\bs{X}_t^N, \bs{\alpha}_t^N) \right)(X_t^{N,j}) \cdot Y_t^{N,i,k} \\
                &\qquad + \partial_\mu \sigma_t\left( X_t^{N,k}, \alpha_t^{N,k}, L^N(\bs{X}_t^N, \bs{\alpha}_t^N) \right)(X_t^{N,j}) \cdot Z_t^{N,i,k,k} \\
                &\qquad + \partial_\mu \sigma^0_t\left( X_t^{N,k}, \alpha_t^{N,k}, L^N(\bs{X}_t^N, \bs{\alpha}_t^N) \right)(X_t^{N,j}) \cdot Z_t^{0,N,i,k} 
            \Bigg]\, dt\\
            &\quad  - \sum_{k} Z_t^{N,i,j,k} dW_t^k  - Z_t^{0, N, i,j}\, dW_t^0,
    \end{split}
\end{equation*}
with
\[
Y_T^{i,j}  = \frac{1}{N}\partial_{\mu} g(X^{N,i}_T, L^N(\bs{X}^N_T))(X^{N,j}_T),
\]
and thus we can write the dynamics of $Y_t^{-1}$ as the following:
\begin{align*}
dY_t^{-1} 
&= - \Bigg(
    \frac{1}{N} P_t + \frac{1}{N}Q^{b}Y^{N, 1,1}_t + \frac{1}{N}R_t^{b}Y^{-1}_t  + S^{b}_t Y^{-1}_t \\
    &\quad  \quad + \frac{1}{N}Q_t^{\sigma}Z^{N,1,1,1}_t + \frac{1}{N}R_t^{\sigma}Z^{-1}_t + S^{\sigma}_t  Z^{-1}_t\\ 
    &\quad  \quad  + \frac{1}{N}Q_t^{\sigma^0}Z^{{0,N, 1,1}}_t + \frac{1}{N}R_t^{\sigma^0}Z^{0,-1}_t + S^{\sigma^0}_t Z_t^{0,-1}
    \Bigg) \, dt
    + \sum_{k} Z_t^{-1,k} dW_t^k   + Z_t^{0,-1}\, dW_t^0.
\end{align*}
Using It\^o's formula, we obtain
\begin{align*}
    &\mathbb{E} \Bigg[ |Y_t^{-1}|^2 
        + \int_t^T \sum_{k=1}^N |Z_s^{-1,k}|^2 + |Z_s^{0,-1}|^2 \, ds \,\Big|\, \mathcal{F}_t^N \Bigg] \\
    &= \mathbb{E} \Bigg[ |Y_T^{-1}|^2 + \int_t^T 2 Y_s^{-1} \cdot \Bigg(
        \frac{1}{N} P_s + \frac{1}{N} Q^{b} Y_s^{N,1,1} + \frac{1}{N} R_s^{b} Y_s^{-1} + S_s^{b} Y_s^{-1}\\
    &\hspace{12em} + \frac{1}{N} Q_s^{\sigma} Z_s^{N,1,1,1} + \frac{1}{N} R_s^{\sigma} Z_s^{-1} + S_s^{\sigma} Z_s^{-1} \\
    &\hspace{12em}  + \frac{1}{N} Q_s^{\sigma^0} Z_s^{0,N,1,1} + \frac{1}{N} R_s^{\sigma^0} Z_s^{0,-1} + S_s^{\sigma^0} Z_s^{0,-1}
        \Bigg) \, ds \,\Big|\, \mathcal{F}_t^N \Bigg] \\
    &\leq \mathbb{E} \Bigg[ |Y_T^{-1}|^2 + C \left(1 + \frac{1}{\varepsilon}\right)
        \int_t^T |Y_s^{-1}|^2 
        +\frac{C}{N^2} |P_s|^2 \, ds \,\Big|\, \mathcal{F}_t^N \Bigg] \\
    &\quad + C \mathbb{E} \Bigg[ \int_t^T \frac{1}{N^2} |Q^{b} Y_s^{N,1,1}|^2 
        + \frac{1}{N^2} |R_s^{b} Y_s^{-1}|^2 + |S_s^{b}   Y_s^{-1}|^2 \\
    &\hspace{5em}  + \frac{1}{N^2} |Q_s^{\sigma} Z_s^{N,1,1,1}|^2 
        + \frac{1}{N^2} |R_s^{\sigma} Z_s^{-1}|^2 + \varepsilon |S_s^{\sigma} Z_s^{-1}|^2 \\
    &\hspace{5em} + \frac{1}{N^2} |Q_s^{\sigma^0} Z_s^{0,N,1,1}|^2 
        + \frac{1}{N^2} |R_s^{\sigma^0} Z_s^{0,-1}|^2 + \varepsilon |S_s^{\sigma^0} Z_s^{0,-1}|^2 \, ds \,\Big|\, \mathcal{F}_t^N
    \Bigg] \\
    &\leq \mathbb{E} \Bigg[ |Y_T^{-1}|^2 + C \left(1 + \frac{1}{\varepsilon}\right)
        \int_t^T |Y_s^{-1}|^2 
        +\frac{C}{N^2} |P_s|^2 \, ds \,\Big|\, \mathcal{F}_t^N 
    \Bigg] \\
    &\quad + C \mathbb{E} \Bigg[ \int_t^T \frac{1}{N} |Y_s^{N,1,1}|^2 + \frac{1}{N} |Y_s^{-1}|^2 + |Y_s^{-1}|^2 \\
    &\hspace{5em} + \frac{1}{N} |Z_s^{N,1,1,1}|^2 + \frac{1}{N} |Z_s^{-1}|^2 + \varepsilon |Z_s^{-1}|^2 \\
    &\hspace{5em}  + \frac{1}{N} |Z_s^{0,N,1,1}|^2 + \frac{1}{N} |Z_s^{0,-1}|^2 
        + \varepsilon |Z_s^{0,-1}|^2 \, ds \,\Big|\, \mathcal{F}_t^N \Bigg],
\end{align*}
where we used the boundedness of $\partial_{x}b,\partial_{x}\sigma, \partial_{x}\sigma^0, \partial_{\mu}b,\partial_{\mu}\sigma,$ and $\partial_{\mu}\sigma^0$, as well as the $\varepsilon$-Young's inequality for an arbitrary $\varepsilon>0$.
Then take $\varepsilon$ sufficiently small and $N$ large, and take the expectation of the both sides to obtain
\begin{equation}
    \begin{split}
        &\mathbb{E} \left[|Y_t^{-1}|^2 
        + \int^T_t \sum_{k} |Z_s^{-1,k}|^2 
        + |Z_s^{0,-1}|^2 \, ds \right]  \\
        &\leq  \mathbb{E} \left[|Y_T^{-1}|^2 \right] + C (1 + \frac{1}{\varepsilon} + \frac{1}{N}) \mathbb{E} \Big[
         \int^T_t |Y_s^{-1}|^2 \, ds \Big] \\
        &\quad +  C\mathbb{E} \left[
        \int^T_t \frac{1}{N^2}|P_s|^2 + 
        \frac{1}{N} (|Y^{N, 1,1}_s|^2 +  |Z^{N,1,1,1}_s|^2 + |Z^{{0,N, 1,1}}_s|^2 ) \, ds
        \right].
        \label{eq:Y-1_Z-1_Z0-1_estimate}
    \end{split}
\end{equation}
Thus, by Gronwall's inequality, we obtain for any $t\in[0,T]$,
\begin{align*}
    \mathbb{E} \left[|Y_t^{-1}|^2  \right] \notag 
    &\leq
        \mathbb{E} \left[|Y_T^{-1}|^2  + \int^T_0
        \frac{1}{N^2}|P_s|^2 + \frac{1}{N} \left(|Y^{N, 1,1}_s|^2 + |Z^{N,1,1,1}_s|^2 + |Z^{{0,N, 1,1}}_s|^2\right) \, ds
        \right]
        \notag\\
    &\leq 
        \mathbb{E} \left[|Y_T^{-1}|^2 \right]  
        + 
        \frac{C}{N} \mathbb{E} \left[\int^T_0
        1 + |X^{N,1}_s|^2 + |Y^{N,1,1}_s|^2 +  |Z^{N,1,1,1}_s|^2 +  |Z^{0,N,1,1}_s|^2 +  |\zeta^{1}_s|^2
        \,ds
        \right] \\
    &\quad + 
        \frac{C}{N^2} \mathbb{E} \left[ \int^T_0 \sum_k ( 
        1 + |X^{N,k}_s|^2 + |Y^{N,k,k}_s|^2 +  |Z^{N,k,k,k}_s|^2 +  |Z^{0,N,k,k}_s|^2 +  |\zeta^{k}_s|^2
        \,ds
        \right], 
\end{align*}
where we used the definition of $P$, as well as the linear growth of $\partial_{\mu}f$.
\begin{align*}
    \mathbb{E}\Bigg[
        |Y_T^{-1}|^2
    \Bigg]
    &\leq
    C\mathbb{E}\Bigg[
        \frac{1}{N^2}\sum^N_{j=2}\left(1 + |X^1_T|^2 + |X^j_T|^2 + \frac{1}{N}\sum^N_{k=1}|X^k_T|^2\right)
    \Bigg] \\
    &\leq
    C\mathbb{E}\Bigg[
        \frac{1}{N}\left(1 + |X^1_T|^2  + \frac{1}{N}\sum^N_{i=1}|X^i_T|^2\right)
    \Bigg].
\end{align*}
Plugging these into \eqref{eq:Y-1_Z-1_Z0-1_estimate}, we finally obtain
\begin{align*}
    &\mathbb{E} \Bigg[ |Y_t^{-1}|^2 
    + \int^T_0 \sum_{k} |Z_t^{-1,k}|^2 
    + |Z_t^{0,-1}|^2 \, dt \Bigg]  \\
    &\quad \leq
    C\mathbb{E} \Bigg[
    \frac{1}{N}\left(1 + |X^{N,1}_T|^2  + \frac{1}{N}\sum_{k}|X^{N,k}_T|^2\right)
    \Bigg] \\
    &\qquad 
    + \frac{C}{N}\mathbb{E} \Bigg[ \int^T_0
     \left( 1 + |X^{N,1}_t|^2 + |Y^{N,1,1}_t|^2 +  |Z^{N,1,1,1}_t|^2 +  |Z^{0,N,1,1}_t|^2 + |\zeta_t^{1}|^2  \right) \, dt
    \Bigg]
    \\
    &\qquad + 
    \frac{C}{N^2} \mathbb{E} \left[ \int^T_0 \sum_k ( 
        1 + |X^{N,k}_s|^2 + |Y^{N,k,k}_s|^2 +  |Z^{N,k,k,k}_s|^2 +  |Z^{0,N,k,k}_s|^2 +  |\zeta^{k}_s|^2
    )
    \,ds
    \right] .
\end{align*}
Analogously, we can derive the same estimates for all $i=1,\ldots, N$. Summing these estimates over $i$, and applying Lemma~\ref{lemma:moment_estimate_Y_Z_Z0_ii}, we obtain
\begin{align*}
    &\mathbb{E}  \Bigg[ \sum_{i}\sum_{j \neq i} \Big( |Y_t^{N,i,j}|^2 
        + \int^T_0\sum_{k} |Z_t^{i,j,k}|^2 
        + |Z_t^{0,N,i,j}|^2 \, dt \Big) 
    \Bigg]  \\
    &\quad \leq
        \frac{C}{N}\mathbb{E}\Bigg[N + \sum_{i}|X^{N,i}_T|^2
    \Bigg]
    + \frac{C}{N}\mathbb{E}\Bigg[ \int^T_0
      N + \sum_{i}(|X^{N,i}_t|^2 + |Y^{N,i,i}_t|^2 +  |Z^{N,i,i,i}_t|^2 +  |Z^{0,N,i,i}_t|^2 + |\zeta^i_t|^2 )  \, dt
    \Bigg] \\
    &\quad \leq 
    C + \frac{C}{N} \mathbb{E}\Bigg[
        \int^T_0 \sum_{i} |\zeta^i_t|^2 \, dt 
    \Bigg]
    +\frac{C}{N^{3/2}}\mathbb{E}\Bigg[ 
        \int^T_0  \sum_{i} \sum_{j\neq i} (| Y_t^{N,i,j} |^2 +  | Z_t^{N,i,j,j } |^2 + | Z_t^{0,N,i,j } |^2) \, dt 
    \Bigg].
\end{align*}
Again, choosing $N$ sufficiently large and integrating from $0$ to $T$, we conclude
\begin{align*}
    \mathbb{E}\Bigg[ 
        \int^T_0 \sum_{i} \sum_{j \neq i} \left(| Y_t^{N,i,j} |^2 +  \sum_{k} |Z_t^{N,i,j,k} |^2 + | Z_t^{0,N,i,j } |^2\right) \, dt
    \Bigg]  
    \leq C +  \frac{C}{N} \mathbb{E}\Bigg[
        \int^T_0 \sum_{i} |\zeta^i_t|^2 \, dt 
    \Bigg].
\end{align*}
\end{proof}

We are now ready to derive the convergence rate for $\zeta^i$, which is needed for the proof of Theorem~\ref{thm:main}.
\begin{proposition} 
\label{prop:zeta_convergence}
Let Assumption \ref{assump:standard_assump}, \ref{assump:Hamiltonian}, and \ref{assump:displacement_monotonicity} hold. Then, for sufficiently large $N$,
it holds that 
\begin{equation}
\mathbb{E}\left[ \int^T_0 
\sum^N_{i=1 }|\zeta^i_t|^2 
\, dt
\right] \leq C,
\label{eq:zeta_convergence}
\end{equation}
for some constant $C>0$ independent of $N$.
\end{proposition}

\begin{proof}
Recall the definition \eqref{eq:zeta_def} of $\zeta^i$. Using the linear growth condition of $\partial_{\nu}f $, the boundedness of $\partial_{\nu} b$, $\partial_{\nu} \sigma$, $\partial_{\nu} \sigma^0$ and Lemma \ref{lemma:moment_estimate_Y_Z_Z0_ii} and  \ref{lemma:moment_estimate_Y_Z_Z0_iJ}, 
\begin{align*}
    \mathbb{E}\Bigg[ \int^T_0 
        \sum_{i}|\zeta^i_t|^2 
        \, dt
    \Bigg]
    &\leq 
    \frac{C}{N^2} 
    \mathbb{E} \Bigg[
         \int^T_0  \sum_{i} \left( 1 + |X^{N,i}_t|^2 
        +  |Y^{N,i,i}_t|^2 
        + |Z^{N,i,i,i}_t|^2
        + |Z^{0,N,i,i}_t|^2 + |\zeta^{N,i}_t|^2 \right)
        \, dt
    \Bigg]\\
    &\quad + 
    \frac{C}{N} 
    \mathbb{E} \Bigg[
         \int^T_0  \sum_{i} \left(
         |Y^{N,i,i}_t|^2 
        + |Z^{N,i,i,i}_t|^2
        + |Z^{0,N,i,i}_t|^2\right)
        \, dt
    \Bigg]\\
    &\quad + 
    \frac{C}{N} 
    \mathbb{E} \Bigg[
         \int^T_0  \sum_{i}\sum_{j \neq i}  \left(
         |Y^{N,i,j}_t|^2 
        + |Z^{N,i,j,j}_t|^2
        + |Z^{0,N,i,j}_t|^2\right)
        \, dt
    \Bigg]\\
    &\leq 
    C + 
    \frac{C}{N} 
    \mathbb{E} \Bigg[
         \int^T_0  \sum_{i}|\zeta^{i}_t|^2 
        \, dt
    \Bigg].
\end{align*}
Take $N$ sufficiently large to conclude.
\end{proof}

The following proposition plays a central role in establishing the convergence of the equilibria.
\begin{proposition}
\label{prop:diff_estimate_for_all_i}
Let Assumptions \ref{assump:standard_assump}, \ref{assump:Hamiltonian}, and \ref{assump:displacement_monotonicity} hold, and let 
\[
    \Delta X^i = X^{N,i} - X^i, \quad
    \Delta Y_t^i = Y^{N,i,i} - Y^i, \quad
    \Delta Z_t^{i,j} = Z^{N,i,i,j} - \delta_{ij}Z^{i}, \quad
    \Delta Z_t^{0,i} = Z^{0,N,i} - Z^{0,i}, 
\]
Then, there exists a sufficiently large $N$ such that the inequality
\begin{equation}
    \begin{split}
        &\sum^N_{i=1 } \mathbb{E}\left[ \int^T_0  |\Delta X^i_t|^2 
            + \int^T_0|\Delta Y^i_t|^2
            +  \sum^{N}_{j=1 }|\Delta Z^{i,j}_t|^2
            + |\Delta Z^{0,i}_t|^2 \, dt 
        \right]\\
        &\quad \leq C + 
        C \sum^N_{i=1 } \mathbb{E}\left[ 
             |E^{G,i}|^2 
            + \int^T_0  
                |E_t^{F,i}|^2
                + |E_t^{B,i} |^2
                + |E_t^{\Sigma,i}|^2
                + |E_t^{\Sigma^0,i}|^2
             \, dt 
        \right],
    \end{split}
\end{equation}
holds for some constant $C>0$ independent of $N$, and for
\begin{equation}
\label{eq:definition_Error_terms}
    \begin{aligned}
        &E_t^{B,i}  = B_t \bigl( \Theta_t^i, L^N(\bs{\Theta}_t) \bigr) 
             - B_t\bigl(\Theta_t^i, \mathcal{L}^1(\Theta_t)\bigr), \quad 
         \\[3pt]
        &E_t^{\Sigma,i} = \Sigma_t \bigl( \Theta_t^i, L^N(\bs{\Theta}_t) \bigr) 
             - \Sigma_t\bigl(\Theta_t^i, \mathcal{L}^1(\Theta_t)\bigr), \quad 
         \\[3pt]
        &E_t^{\Sigma^0,i} 
        = \Sigma_t^0 \bigl( \Theta_t^i, L^N(\bs{\Theta}_t) \bigr) 
            - \Sigma_t^0\bigl(\Theta_t^i, \mathcal{L}^1(\Theta_t)\bigr), \quad \\
        &E_t^{F,i} 
        = F_t \bigl( \Theta_t^i, L^N(\bs{\Theta}_t) \bigr) 
            - F_t\bigl(\Theta_t^i, \mathcal{L}^1(\Theta_t)\bigr)
            + \frac{1}{N}\partial_{\mu}H_t (\Theta_t^i, \alpha_t^i, L^N(\bs{X}_t, \bs{\alpha}_t)) (X_t^i)  
            , \\[1pt]
        &E^{G,i} 
        =  G\bigl(X_T^i, L^N(\bs{X}_T)\bigr) 
           - G\bigl(X_T^i, \mathcal{L}^1(X_T)\bigr) + \frac{1}{N} \partial_{\mu} g\bigl(X_T^i, L^N(\bs{X}_T)\bigr)( X_T^i ).
    \end{aligned}
\end{equation}
\end{proposition}
\begin{proof}
First, use It\^o's product rule to compute
\begin{align*}
&d \Delta X_t^i \cdot \Delta Y_t^i \\
&= \Big[
    - \Delta X_t^i \cdot \left( 
        F^i_t(\bs{\Theta}^N_t, \bs{\zeta}_t) - F_t(\Theta_t^i , \mathcal{L}^1(\Theta_t)) 
    \right) 
    + \Delta Y_t^i \cdot \left( 
        B^i_t(\bs{\Theta}^N_t, \bs{\zeta}_t) - B_t(\Theta_t^i , \mathcal{L}^1(\Theta_t)) 
    \right) 
    \\
    &\qquad + \Delta Z_t^{i,i} \cdot \left( 
        \Sigma^i_t(\bs{\Theta}^N_t, \bs{\zeta}_t) - \Sigma_t(\Theta_t^i , \mathcal{L}^1(\Theta_t)) 
    \right) 
    + \Delta Z_t^{0,i} \cdot \left( 
        \Sigma^{0,i}_t(\bs{\Theta}^N_t, \bs{\zeta}_t) - \Sigma^0_t(\Theta_t^i , \mathcal{L}^1(\Theta_t)) 
    \right)
    \\[-0.5em]
    &\qquad - \Delta X_t^i \cdot \frac{1}{N} \partial_{\mu} H^{N,i}_t \bigl( 
        \bs{X}^N_t,  \bs{Y}^{N,i,:}_t, \bs{Z}^{N,i,:}_t, \bs{Z}^{0,N,i,:}, \bs{\alpha}^N_t
    \bigr)( X_t^{N,i} ) 
\Bigg] dt + dM^i_t 
\\
&= \Bigg[
    - \Delta X_t^i \cdot \left( 
        F_t(\Theta_t^{N,i}, L^N(\bs{\Theta}_t^{N})) - F_t(\Theta_t^i , L^N(\bs{\Theta}_t)) 
    \right) 
    + \Delta Y_t^i \cdot \left( 
        B_t(\Theta_t^{N,i}, L^N(\bs{\Theta}_t^{N})) - B_t(\Theta_t^i ,  L^N(\bs{\Theta}_t)) 
    \right) 
    \\
    &\qquad + \Delta Z_t^{i,i} \cdot \left( 
        \Sigma_t(\Theta_t^{N,i}, L^N(\bs{\Theta}_t^{N})) - \Sigma_t(\Theta_t^i ,  L^N(\bs{\Theta}_t)) 
    \right)
    + \Delta Z_t^{0,i} \cdot \left( 
        \Sigma^0_t(\Theta_t^{N,i}, L^N(\bs{\Theta}_t^{N})) - \Sigma^0_t(\Theta_t^i ,  L^N(\bs{\Theta}_t)) 
    \right)
    \\[0.5em]
    &\qquad 
    - \Delta X_t^i \cdot \left( 
        F^i_t(\bs{\Theta}^N_t, \bs{\zeta}_t) - F_t(\Theta_t^{N,i}, L^N(\bs{\Theta}_t^{N})) 
    \right) 
    + \Delta Y_t^i \cdot \left( 
        B^i_t(\bs{\Theta}^N_t, \bs{\zeta}_t) - B_t(\Theta_t^{N,i}, L^N(\bs{\Theta}_t^{N})) 
    \right) 
    \\
    &\qquad + \Delta Z_t^{i,i} \cdot \left( 
        \Sigma^i_t(\bs{\Theta}^N_t, \bs{\zeta}_t) - \Sigma_t(\Theta_t^{N,i}, L^N(\bs{\Theta}_t^{N})) 
    \right)
    + \Delta Z_t^{0,i} \cdot \left( 
        \Sigma^{0,i}_t(\bs{\Theta}^N_t, \bs{\zeta}_t) - \Sigma^0_t(\Theta_t^{N,i}, L^N(\bs{\Theta}_t^{N})) 
    \right)
    \\[0.5em]
    &\qquad 
    - \Delta X_t^i \cdot \frac{1}{N} \bigg( 
        \partial_{\mu} H^{N,i}_t \bigl( 
            \bs{X}^N_t, \bs{Y}^{N,i,:}_t, \bs{Z}^{N,i,:}_t, \bs{Z}^{0,N,i,:}, \bs{\alpha}^N_t
        \bigr)( X_t^{N,i} ) 
        - \partial_{\mu}H_t \bigl( 
            \Theta_t^i, \alpha_t^i, L^N(\bs{X}_t, \bs{\alpha}_t) 
        \bigr) (X_t^i) 
    \bigg)
    \\
    &\qquad 
    + \left( 
        -\Delta X_t^i \cdot E_t^{F,i} 
        + \Delta Y_t^i \cdot E_t^{B,i}  
        + \Delta Z_t^{i,i} \cdot E_t^{\Sigma,i} 
        + \Delta Z_t^{i,0} \cdot E_t^{\Sigma^0,i} 
    \right) 
\Bigg] dt  + dM_t^i,
\end{align*}
where $M_t^i$ is a $\mathbb{P}^N$-martingale. 
Similarly, we can compute the terminal condition
\begin{equation*}
\begin{split}
\Delta X_T^i \cdot \Delta Y_T^i
&= \Delta X_T^i \cdot 
  \Big(
    G(X^{N,i}, L^N(\bs{X}^N_T))
    - G(X^{i}, L^N(\bs{X}_T))
  \Big) \\
&\quad + \Delta X_T^i \cdot \frac{1}{N} 
  \Big(
    \partial_{\mu} g(X^{N,i}, L^N(\bs{X}^N_T))(X^{N,i}_T)
    - \partial_{\mu} g(X^{i}, L^N(\bs{X}_T))(X^{i}_T)
  \Big) \\[0.5em]
&\quad + \Delta X_T^i \cdot E^{G,i}.
\end{split}
\end{equation*}
Thus, using the monotonicity condition in Assumption~\ref{assump:displacement_monotonicity} and the Lipschitz continuity of $\partial_x H, b, \sigma, \sigma^0, \Lambda$, we obtain
\begin{align*}
&C\sum_i\mathbb{E}\Bigg[ 
     |\Delta X^i_T| ^2 
    +   \Delta X^i_T 
    \cdot \frac{1}{N}\left(
        \partial_{\mu} g(X^{N,i}, L^N(\bs{X}^N_T))(X^{N,i}_T)
        - \partial_{\mu}g(X^{i}, L^N(\bs{X}_T))(X^{i}_T)
    \right)  +  \Delta X^i_T \cdot E^{G,i}
\Bigg] 
\\
&\quad \leq C \sum_i \mathbb{E}\Bigg[ 
    \int^T_0 -  |\Delta X^i_t| ^2 
    +  \left(
        - \Delta X_t^i \cdot E_t^{F,i} 
        + \Delta Y_t^i \cdot E_t^{B,i} 
        + \Delta Z_t^{i,i} \cdot E_t^{\Sigma,i}
        + \Delta Z_t^{i,0} \cdot E_t^{\Sigma^0,i} 
    \right) \, dt 
\Bigg]
\\
&\quad \quad + C \sum_i\mathbb{E}\Bigg[ \int^T_0 
    \big(
        |\Delta X_t^i| 
        + |\Delta Y_t^i|
        + |\Delta Z_t^{i,i}|
        + |\Delta Z_t^{i,0}| 
    \big) \, |\zeta^i_t|  \, dt 
\Bigg] \notag \\
&\quad \quad - \sum_i \mathbb{E}\Bigg[
    \int^T_0  \Delta X_t^i \cdot \frac{1}{N} \bigg(
        \partial_{\mu} H^{N,i}_t\bigl(
             \bs{X}^N_t, \bs{Y}^{N,i,:}_t, \bs{Z}^{N,i,:}_t, \bs{Z}^{0,N,i,:}, \bs{\alpha}^N_t
        \bigr)( X_t^{N,i} ) 
        - \partial_{\mu}H_t \bigl(
             \Theta_t^i, \alpha_t^i,  L^N(\bs{X}_t, \bs{\alpha}_t)
        \bigr) (X_t^i)
    \bigg) \, dt 
\Bigg].
\end{align*}
Therefore, using the Lipschitz continuity of $\Lambda$, $\partial_{\mu}g$, and $\partial_{\mu}H$ and the boundedness of $\partial_{\mu}b, \partial_{\mu}\sigma$, and $\partial_{\mu}\sigma^0$, and applying the $\varepsilon$-Young's inequality for an arbitrary $\varepsilon > 0$, we deduce
\begin{align}
&C\sum_i \mathbb{E}\Bigg[
    |\Delta X^i_T| ^2 + \int^T_0 |\Delta X^i_t| ^2 \, dt
\Bigg] 
\notag \\
&\quad \leq 
\sum_i \mathbb{E}\Bigg[ 
    \frac{1}{N}  |\Delta X^i_T| \,
    |\partial_{\mu} g(X^{N,i}, L^N(\bs{X}^N_T))(X^{N,i}_T)
    - \partial_{\mu} g(X^{i}, L^N(\bs{X}_T))(X^{i}_T)| 
    +  |\Delta X^i_T|\,  |E^{G,i}|   
\Bigg]
\notag \\
&\quad \quad + C \sum_i  \mathbb{E}\Bigg[ 
    \int^T_0 \big(
        |\Delta X_t^i| \, | E_t^{f,i}| 
        + |\Delta Y_t^i|\,|E_t^{B,i} | 
        + |\Delta Z_t^{i,i}| \, |E_t^{\Sigma,i}|
        + |\Delta Z_t^{i,0}| \, |E_t^{\Sigma^0,i}| \, 
    \big) \, dt 
\Bigg]
\notag \\
&\quad \quad + C \sum_i \mathbb{E}\Bigg[ 
    \int^T_0 \big(
        |\Delta X_t^i| 
        + |\Delta Y_t^i|
        + |\Delta Z_t^{i,i}|
        + |\Delta Z_t^{i,0}| 
        \big)  \, |\zeta^i_t|  \, dt 
\Bigg]
\notag \\
&\quad \quad + \frac{1}{N} \sum_i \mathbb{E}\Bigg[
    \int^T_0  |\Delta X_t^i| \,  
    \Big|
        \partial_{\mu}H_t \bigl(
             \Theta_t^{N,i}, \alpha_t^{N,i}, L^N(\bs{X}^N_t, \bs{\alpha}^N_t)
        \bigr) (X_t^{N,i}) 
        - \partial_{\mu}H_t \bigl(
             \Theta_t^i, \alpha_t^i, L^N(\bs{X}_t, \bs{\alpha}_t)
        \bigr) (X_t^i)
    \Big| \, dt 
\Bigg]
\notag \\
&\quad \quad + \frac{C}{N} \sum_{i} \mathbb{E}\Bigg[
    \int^T_0  |\Delta X_t^i| \, \sum_{j }  \,  \big(
        |Y_t^{N,i,j}| + |Z_t^{N,i,j,j}| + |Z_t^{0,N,i,j}|
    \big)
\Bigg] 
\notag \\[5pt]
&\quad \leq \sum_{i} \mathbb{E}\Bigg[
    \left(\frac{1}{N}+ \varepsilon \right) C  |\Delta X^i_T|^2 
        + \left(\frac{1}{N}+ \varepsilon \right)  C \int^T_0 
     \big(
    |\Delta X^i_t|^2 + |\Delta Y^i_t|^2 + |\Delta Z^{i,i}_t|^2+ |\Delta Z^{0,i}_t|^2
    \big) \, dt
\Bigg]
\notag \\
&\quad \quad + \frac{C}{\varepsilon} \mathbb{E}\sum_{i} \Bigg[ 
     |E^{G,i}|^2  
    + \int^T_0   \big( 
        |E^{f,i}_t|^2 +  E^{b,i}_t|^2 + | E^{\sigma,i}_t|^2+ | E^{\sigma^0,i}_t|^2 
    \big) 
\Bigg] 
+ C + (\frac{1}{\varepsilon} + \frac{1}{N})C \sum_{i}\mathbb{E}\Bigg[ 
    \int^T_0 |\zeta^i_t|^2 \, dt 
\Bigg]
\notag \\[5pt] 
&\quad \leq \sum_{i} \mathbb{E}\Bigg[
    \left(\frac{1}{N}+ \varepsilon \right) C |\Delta X^i_T|^2  + \left(\frac{1}{N}+ \varepsilon \right)  C \int^T_0 
     \big(
    |\Delta X^i_t|^2 + |\Delta Y^i_t|^2 + |\Delta Z^{i,i}_t|^2+ |\Delta Z^{0,i}_t|^2
    \big) \, dt
\Bigg]
\label{eq:DeltaX_estimate_for_i_with_epsilon} \\
&\quad \quad + \frac{C}{\varepsilon} +  \frac{C}{\varepsilon}  \sum_{i}\mathbb{E}\Bigg[ 
     |E^{G,i}|^2  
    + \int^T_0  \big( 
        |E^{f,i}_t|^2 +  |E^{b,i}_t|^2 + | E^{\sigma,i}_t|^2+ | E^{\sigma^0,i}_t|^2 
    \big) 
\Bigg],
\notag
\end{align}
where we applied Lemma~\ref{lemma:moment_estimate_Y_Z_Z0_iJ} in the second inequality to estimate the terms $\sum_{j \neq i}(|Y_t^{N,i,j}|^2 + |Z_t^{N,i,j,j}|^2 + |Z_t^{0,N,i,j}|^2)$, and Lemma~\ref{prop:zeta_convergence} in the third one.

Next, we are estimating the difference of the backward systems. Apply It\^o's formula to $|\Delta Y^i|^2$, and $\eta$-Young's inequality for an arbitrary $\eta>0$ to derive
\begin{align}
&\mathbb{E} \Bigg[ 
    |\Delta Y^i_t|^2
    + \int^T_t \sum_j |\Delta Z^{i,j}_s|^2 + |\Delta Z^{0,i}_s|^2 \, ds \,\Big|\, \mathcal{F}_t^N 
\Bigg] 
\notag \\[2pt]
& \leq
    C\, \mathbb{E} \Bigg[ 
        \left| G(X^{N,i}, L^N(\bs{X}^N_T))
        - G(X^{i}, L^N(\bs{X}_T)) \right|^2
        \notag \\[2pt]
        &\qquad \qquad  + \frac{C}{N^2}
        \left|
            \partial_{\mu} g(X^{N,i}, L^N(\bs{X}^N_T))(X^{N,i}_T)
            - \partial_{\mu}g(X^{i}, L^N(\bs{X}_T))(X^{i}_T) 
        \right|^2  
        +  |E^{G,i}|^2 \,\Big|\, \mathcal{F}_t^N
    \Bigg]  
    + \frac{C}{\eta} \mathbb{E} \Bigg[
        \int^T_t |\Delta Y^i_s|^2 \, ds \,\Big|\, \mathcal{F}_t^N 
    \Bigg] 
    \notag \\[2pt]
    &\quad + C\eta\, \mathbb{E} \Bigg[ 
        \int^T_t \left|
            F^i_s (\bs{\Theta}_s^{N}, \bs{\zeta}_t) 
            - F_s ( \Theta_s^{i}, L^N(\bs{\Theta}_s)) 
        \right|^2 
        \notag \\[2pt]
        &\qquad \qquad \quad  + \frac{1}{N^2}  \big|
                \partial_{\mu} H^{N,i}_s\bigl(
                    \bs{X}^N_s,  \bs{Y}^{N,i,:}_s, \bs{Z}^{N,i,:}_s, \bs{Z}^{0,N,i,:}_s, \bs{\alpha}^N_s
                \bigr)( X_s^{N,i} ) 
                - \partial_{\mu}H_t \bigl(
                     \Theta_s^i, \alpha_s^i,  L^N(\bs{X}_s, \bs{\alpha}_s) 
                \bigr)
        \big|^2
        + |E^{f,i}_s|^2 \, ds \,\Big|\, \mathcal{F}_t^N 
    \Bigg]
    \notag \\
&\leq
    C\, \mathbb{E} \Bigg[
        \left(1 + \frac{1}{N^2}\right) |\Delta X^i_T|^2 + \left( \frac{1}{N} + \frac{1}{N^3} \right) \sum_j |\Delta X^j_T|^2 \,\Big|\, \mathcal{F}_t^N 
    \Bigg]
    \label{eq:DeltaYZZ0_estimate_for_i_with_eta} \\[2pt]
    &\quad + 
    C\left( \eta + \frac{1}{\eta} + \frac{1}{N^2}\right) \mathbb{E} \Bigg[
        \int^T_t |\Delta Y^i_s|^2 \, ds \,\Big|\, \mathcal{F}_t^N 
    \Bigg] \notag \\
    &\quad + C \eta \left( 1 + \frac{1}{N^2} \right) \mathbb{E} \Bigg[ 
        \int^T_t |\Delta X^i_s|^2 + |\Delta Z^{i,i}_s|^2 + |\Delta Z^{0,i}_s|^2 \, ds \,\Big|\, \mathcal{F}_t^N 
    \Bigg]
    \notag \\[2pt]
    &\quad + C \eta \left( 1 + \frac{1}{N^2} \right)\mathbb{E} \Bigg[ \frac{1}{N}
        \int^T_t \sum_j (|\Delta X^j_s|^2 + |\Delta Y^j_s|^2 + |\Delta Z^{j,j}_s|^2 + |\Delta Z^{0,j}_s|^2) 
        \, ds \,\Big|\, \mathcal{F}_t^N
    \Bigg]
    \notag \\[2pt]
    &\quad + \frac{C}{N} \E \Bigg[
        \int^T_0 \sum_{j \neq i} (|Y_s^{N,i,j}|^2 + |Z_s^{N,i,j,j}|^2 + |Z_s^{0,N,i,j}|^2 )
        \, ds \,\Big|\, \mathcal{F}_t^N
    \Bigg] 
    \notag \\ 
    &\quad + 
    C\, \mathbb{E} \Bigg[ 
    |E^{g,i}|^2 + \eta \int^T_0 |E_s^{f,i}|^2 
    + \eta (1 + \frac{1}{N^2}) C  |\zeta^i_s|^2 +  \frac{1}{N}\sum_j |\zeta^j_s|^2 
    \, ds \,\Big|\, \mathcal{F}_t^N 
    \Bigg],\notag
\end{align}
We sum this estimate over all $i=1,\ldots, N$ and use Lemma~\ref{lemma:moment_estimate_Y_Z_Z0_iJ} and Proposition~\ref{prop:zeta_convergence} to derive
\begin{align*}
&\sum_{i}\mathbb{E} \Bigg[
    |\Delta Y^i_t|^2
    + \left(1 - \eta C - \frac{C}{N}\right) \int^T_t  \sum_{j} |\Delta Z^{i,j}_s|^2
    +  |\Delta Z^{0,i}_s|^2 
    \, ds \,\Big|\, \mathcal{F}_t^N 
\Bigg] \\
&\leq
    C \left(1 + \eta + \frac{1}{\eta}\right) \sum_{i} \mathbb{E} \Bigg[ 
        \int^T_t |\Delta Y^i_s|^2 \, ds \,\Big|\, \mathcal{F}_t^N 
    \Bigg]  \\
    &\quad + C \sum_{i} \mathbb{E} \Bigg[
        |\Delta X^i_T|^2 
        + \left(\eta + \frac{1}{N}\right) \int^T_0 |\Delta X^i_s|^2  
        \, ds \,\Big|\, \mathcal{F}_t^N 
    \Bigg] \\
    &\quad + \eta C  + C\, \mathbb{E} \sum_{i} \Bigg[ 
        |E^{g,i}|^2  + \eta \int^T_0 |E^{f,i}_s|^2 
        \, ds \,\Big|\, \mathcal{F}_t^N 
    \Bigg].
\end{align*}
Taking $\eta$ sufficiently small and $N$ sufficiently large, and then taking the expectation of both sides, we can apply Gronwall's inequality to obtain, for any $t \in [0,T]$,
\begin{equation}
\label{eq:Sum_DEltaYZZ_estimate}
    \begin{split}
        &\sum_{i} \mathbb{E} \Bigg[
            |\Delta Y^i_t|^2
            + \int^T_0 \sum_{j} |\Delta Z^{i,j}_t|^2
            + |\Delta Z^{0,i}_t|^2 
            \, dt 
        \Bigg] \\
        &\leq
            C + 
            C \, \sum_{i} \mathbb{E} \Bigg[
                |\Delta X^i_T|^2 
                + \int^T_0 |\Delta X^i_t|^2  \, dt \Bigg] 
                + C \, \mathbb{E} \sum_{i} \Bigg[  |E^{g,i}|^2  + \int^T_0  |E^{f,i}_t|^2  
                \, dt 
            \Bigg],
    \end{split}
\end{equation}
where $C>0$ depends on $\eta$ as well. 
Plug this into \eqref{eq:DeltaX_estimate_for_i_with_epsilon}, and choose $\varepsilon$ sufficiently small and $N$ sufficiently large. Then we obtain
\begin{equation}
    \begin{split}
        &\sum_{i}\mathbb{E}\Bigg[ 
            |\Delta X^i_T| ^2 + \int^T_0 |\Delta X^i_t| ^2 \, dt 
        \Bigg]
        \\
        &\quad \leq 
            C + C \mathbb{E}\sum_{i}\Bigg[
                | E^{g,i}|^2 
                + \int^T_0
                 (| E^{f,i}_t|^2 + |  E^{b,i}_t|^2 + |  E^{\sigma,i}_t|^2+ |  E^{\sigma^0,i}_t|^2 ) 
                \, dt 
            \Bigg].
            \label{eq:Sum_DeltaX_estimate}
    \end{split}
\end{equation}
Plugging \eqref{eq:Sum_DeltaX_estimate} back into \eqref{eq:Sum_DEltaYZZ_estimate}, the statement follows immediately.
\end{proof}

\begin{proof}[Proof of Theorem \ref{thm:main}]
Recall \eqref{eq:definition_Error_terms}. For example, for $E^{F,i}$,
\begin{equation*}
    | E_t^{F,i} |^2 
    \leq 2 | F_t\bigl( \Theta_t^i, L^N(\bs{\Theta}_t ) \bigr) 
    - F_t\bigl(\Theta_t^i, \mathcal{L}^1(\Theta_t)\bigr) |^2
    + \frac{2}{N^2} |\partial_{\mu}H_t  (\Theta_t^i, \alpha_t^i,  L^N(\bs{X}_t, \bs{\alpha}_t)) (X_t^i)|^2  .
\end{equation*}
Our aim is to control $\E \left[\int^T_0 \sum_i |E_t^{F,i} |^2 \, dt\right]$ by some constant $C>0$, which is independent of $i,N$. 
For the second term, we have
\begin{equation}
\label{eq:E^2_estimate_second_term}
\begin{split}
    &\frac{1}{N^2} \E \Big[
    \big| \partial_{\mu}H (t, \Theta_t^i, \alpha_t^i, L^N(\bs{X}_t, \bs{\alpha}_t)) (X_t^i)  \big|^2
    \Big]  \\
    &\leq   
        \frac{C}{N^2}  \mathbb{E} \Big[
            |X^i_t|^2 + |Y^{i}_t|^2 + |Z^{i}_t|^2 + |Z^{0,i}_t|^2 
            + \frac{1}{N}\sum_{j} (1 + |X^{j}_t|^2  + |Y^{j}_t|^2 + |Z^{j}_t|^2 + |Z^{0,j}_t|^2 ) 
        \Big].
\end{split}
\end{equation}
Next, for the first term of $E^{F,i}_t$, we follow the same argument as in the proof of \cite[Theorem~2]{jackson2024quantitative}. Set $L^{N, -i}(\bs{\Theta}_t) = \frac{1}{N-1} \sum_{j \neq i} \delta_{\Theta_t^j}$. Then, we have
\begin{equation*}
    \E \big[
    \mathcal{W}_2^2(L^N(\bs{\Theta}_t ), L^{N, -i}(\bs{\Theta}_t))\big]
    \leq \frac{C}{N} \E \big[
        |X^1_t|^2 +  |Y^1_t|^2 + |Z^1_t|^2 + |Z^{0,1}_t|^2
    \big].
\end{equation*}
Thus, it follows
\begin{equation}
\label{eq:E^2_estimate_first_term}
    \begin{split}
        &\E \big[ 
            \big|
                F_t\bigl( \Theta_t^i, L^N(\bs{\Theta}_t ) \bigr) 
                - F_t\bigl( \Theta_t^i, \mathcal{L}^1(\Theta_t)\bigr)
            \big|^2
        \big] \\
        &\leq 
        \frac{C}{N} \E \big[
            |X^1_t|^2 +  |Y^1_t|^2 + |Z^1_t|^2 + |Z^{0,1}_t|^2  
        \big] 
        + 2\E \big[ 
            \big|
                F_t\bigl( \Theta_t^i, L^{N, -i}(\bs{\Theta}_t) \bigr) 
                - F_t\bigl( \Theta_t^i, \mathcal{L}^1(\Theta_t)\bigr)
            \big|^2
        \big].
    \end{split}
\end{equation}
Thanks to assumption \ref{assump:smoothness_L-derivative}, we can use \cite[Lemma~9]{jackson2024quantitative} to compute
\begin{align}
    &\E \big[ 
        \big|
            F_t\bigl( \Theta_t^i, L^{N, -i}(\bs{\Theta}_t) \bigr) 
            - F_t\bigl(\Theta_t^i, \mathcal{L}^1(\Theta_t)\bigr)
        \big|^2
    \big] \notag \\
    &\leq \E^0 \bigg[ 
        \E^1 \big[
            \big|
                F_t\bigl( \Theta_t^i, L^{N, -i}(\bs{\Theta}_t) \bigr) 
                - F_t\bigl( \Theta_t^i, \mathcal{L}^1(\Theta_t)\bigr)
            \big|^2
        \big]
    \bigg] \notag \\
    &= \E^0 \bigg[ 
        \int_{\R^n\times \R^n \times \R^{n\times d} \times \R^{n\times d}}
            \big|
                F_t\bigl(\theta, L^{N, -i}(\bs{\Theta}_t) \bigr) 
                - F_t\bigl(\theta, \mathcal{L}^1(\Theta_t)\bigr)
            \big|^2
        \,\, \mathcal{L}^1(\Theta_t)(d\theta)
    \bigg] \notag  \\
    &\leq \E^0 \bigg[ 
        \int_{\R^n\times \R^n \times \R^{n\times d} \times \R^{n\times d}}
            \frac{C}{N}(1 + |\theta|^2 + M^2_2(\mathcal{L}^1(\Theta_t)) )
        \,\, \mathcal{L}^1(\Theta_t)(d\theta)
    \bigg] \notag \\
    &\leq \frac{C}{N} \E \bigg[ 
    (1 +  |X^1_t|^2 +  |Y^1_t|^2 + |Z^1_t|^2 + |Z^{0,1}_t|^2  ) 
    \bigg]. \notag 
\end{align}
Therefore, from \eqref{eq:E^2_estimate_second_term}, \eqref{eq:E^2_estimate_first_term}, we get
\begin{equation}
\label{eq:Control_SUM_E^F_by_constant}
    \E \left[\int^T_0 \sum_i |E_t^{F,i} |^2 \, dt \right] \leq C.
\end{equation}
Likewise, it is easy to check that similar bounds also hold for $E^{B,i}$, $E^{\Sigma,i}$, $E^{\Sigma^0,i}$, and $E^{G,i}$ (possibly with a different constant $C$).
Thus, by Proposition \ref{prop:diff_estimate_for_all_i}, we have
\begin{align}
\label{eq:sum_Delta_XYZZ0_by_constant}
&\mathbb{E}\left[
    \int^T_0 \sum_{i} (
        |\Delta X^i_t|^2 
        + |\Delta Y^i_t|^2
        + |\Delta Z^{i,i}_t|^2
        + |\Delta Z^{0,i}_t|^2 
    ) \, dt 
\right] 
\leq C .
\end{align}

Now, compute
\begin{equation*}
\begin{split}
    |\Delta X^i_t|^2 
    &\leq \int^t_0 \left| B^i_s(\bs{\Theta}^N_s, \bs{\zeta}_s) - B_s(\Theta^i_s, \mathcal{L}^1(\Theta_s)) \right|^2 \, ds   
    + \left| \int^t_0 \left( \Sigma^i_s(\bs{\Theta}^N_s, \bs{\zeta}_s) - \Sigma_s(\Theta^i_s, \mathcal{L}^1(\Theta_s)) \right) \, dW^i_s \right|^2 \\
    &\quad 
    + \left| \int^t_0 \left( \Sigma^{0,i}_s(\bs{\Theta}^N_s, \bs{\zeta}_s) - \Sigma^0_s(\Theta^i_s, \mathcal{L}^1(\Theta_s)) \right) \, dW^0_s \right|^2 .
\end{split}
\end{equation*}
We apply Burkholder-Davis-Gundy inequality to get 
\begin{align*}
\mathbb{E} \left[ \sup_{t \in [0,T]} |\Delta X^i_t|^2 \right] 
&\leq C \mathbb{E} \Bigg[ 
    \int^T_0 \left| B^i_t(\bs{\Theta}^N_t, \bs{\zeta}_t) - B_t(\Theta^i_t, \mathcal{L}^1(\Theta_t))  \right|^2 
    + \left| \Sigma^i_t(\bs{\Theta}^N_t, \bs{\zeta}_t) - \Sigma_t(\Theta^i_t, \mathcal{L}^1(\Theta_t)) \right|^2 \\
    &\qquad \qquad \quad 
    +\left| \Sigma^{0,i}_t(\bs{\Theta}^N_t, \bs{\zeta}_t) - \Sigma^0_t(\Theta^i_t, \mathcal{L}^1(\Theta_t)) \right|^2 dt 
\Bigg] \\[2pt]
&\leq C \mathbb{E} \Bigg[ 
    \int^T_0  \left| B^i_t(\bs{\Theta}^N_t, \bs{\zeta}_t) - B_t(\Theta^i_t, L^N(\bs{\Theta}_t)) \right|^2 
    +  \left| \Sigma^i_t(\bs{\Theta}^N_t, \bs{\zeta}_t) - \Sigma_t(\Theta^i_t, L^N(\bs{\Theta}_t)) \right|^2 \\
    &\qquad \qquad \quad   
    +  \left| \Sigma^{0,i}_t(\bs{\Theta}^N_t, \bs{\zeta}_t) - \Sigma^0_t(\Theta^i_t, L^N(\bs{\Theta}_t)) \right|^2
    + | E^{B,i}_t |^2 + | E^{\Sigma,i}_t |^2  + | E^{\Sigma^0,i}_t |^2 \, dt 
\Bigg] \\[2pt]
&\leq C \mathbb{E} \Bigg[ \int^T_0 
    | \Delta X^i_t |^2 + | \Delta Y^i_t |^2  + | \Delta Z^{i,i}_t |^2 + | \Delta Z^{0,i}_t |^2 \\
    &\qquad \qquad \quad  
    + \frac{1}{N} \sum_{j} \left( | \Delta X^j_t |^2 + | \Delta Y^j_t |^2  + | \Delta Z^{j,j}_t |^2 + | \Delta Z^{0,j}_t |^2 \right) \\
    &\qquad \qquad \quad 
    + |\zeta^i_t|^2 + \frac{1}{N} \sum_{j} |\zeta^j_t|^2 + | E^{B,i}_t |^2 + | E^{\Sigma,i}_t |^2 + | E^{\Sigma^0,i}_t |^2 \, dt
\Bigg] ,
\end{align*}
and sum this estimate over all $i$;
\begin{equation*}
\begin{split}
    &\sum_{i}  \mathbb{E} \left[ \sup_{t \in [0,T]} |\Delta X^i_t|^2 \right] 
    \leq C \mathbb{E} \Bigg[ \int^T_0 \sum_{i} \big(
        | \Delta X^i_t |^2 + | \Delta Y^i_t |^2  + | \Delta Z^{i,i}_t |^2 + | \Delta Z^{0,i}_t |^2 \\
        &\hspace{15em}+ |\zeta^i_t|^2 + | E^{B,i}_t |^2 + | E^{\Sigma,i}_t |^2 + | E^{\Sigma^0,i}_t |^2 
    \big) dt \Bigg].
\end{split}
\end{equation*}
We apply Lemma~\ref{prop:zeta_convergence}, \eqref{eq:Control_SUM_E^F_by_constant}, and \eqref{eq:sum_Delta_XYZZ0_by_constant} to conclude
\begin{equation*}
    \sum_{i} \mathbb{E} \left[ \sup_{t \in [0,T]} |\Delta X^i_t|^2 \right] \leq C.
\end{equation*}

Finally, using the Lipschitz continuity of $\Lambda$, we conclude that
\begin{align*}
\sum_i\mathbb{E}\Bigg[
\int^T_0  |\alpha_t^{N,i} - \alpha_t^{i}|^2   \, dt 
\Bigg]
&\leq 
C \mathbb{E}\Bigg[ \int^T_0 \sum_i ( 
    |\Delta X^i_t|^2 + |\Delta Y^i_t|^2 + |\Delta Z^i_t|^2 +|\Delta Z^i_t|^2 
    + |\zeta^i_t|^2
) \, \,  dt
\Bigg]\\
&\leq  C,
\end{align*}
where we used Lemma \ref{prop:zeta_convergence} and \eqref{eq:sum_Delta_XYZZ0_by_constant} in the second line.
\end{proof}

\section{Linear-quadratic example}
\label{sec:linear-quadratic}

In this section, we investigate a class of linear-quadratic mean field games (LQ-MFGs) with a specific structure. 
We demonstrate that the unique MFE exists for this class and that the convergence result established in Theorem~\ref{thm:main} holds.
Given $\alpha \in \mathbb{A}$, $\mu \in \mathbb{H}^2(\mathbb{F}^0; \R^n)$, and $\nu \in \mathbb{H}^2(\mathbb{F}^0; \R^\ell)$, consider the following state process of the representative player:
\begin{equation}
\label{eq:LQ-state_process}
\begin{split}
    dX_t &= \Big( A_t X_t + B_t (c_1\alpha_t - c_2\nu_t)  \Big)  dt 
    + \Big( C_t X_t + D_t (c_1\alpha_t - c_2\nu_t)  \Big) dW_t + \Big( C^0_t X_t + D^0_t (c_1\alpha_t - c_2\nu_t) \Big) dW^0_t,
\end{split}
\end{equation}
where $c_1, c_2 \in \R$. 
Although the process depends on $\alpha$ and $\nu$, we suppress this explicit dependence for notational simplicity. 
In equilibrium, the consistency condition $\nu_t = \mathbb{E}[ \alpha_t \mid \mathcal{F}_t^0 ]$ must hold.

The player aims to minimize the following quadratic cost functional:
\begin{equation*}
    \begin{split}
        J(\alpha, \mu, \nu) &= \mathbb{E} 
        \Bigg[
        \frac{1}{2} \int^T_0 \Big( X^{\top}_t Q_t X_t 
        + (X_t - S_t \mu_t)^\top \bar{Q}_t (X_t - S_t \mu_t) 
        +  \alpha_t^\top P_t \alpha_t 
        +  \nu_t ^\top \bar{P}_t  \nu_t  \Big)
        \, dt 
        \Bigg] \\
        &\quad + \mathbb{E} 
        \Bigg[
        \frac{1}{2} X^\top_T Q_T X_T + \frac{1}{2} (X_T - S_T \mu_T)^\top \bar{Q}_T (X_T - S_T \mu_T)
        \Bigg],
    \end{split}
\end{equation*}
where $M^\top$ denotes the transpose of a matrix $M$.

A linear-quadratic mean field equilibrium (LQ-MFE) is defined as a triplet $(\hat{\alpha}, \hat{\mu}, \hat{\nu}) \in \mathbb{A} \times \mathbb{S}^2(\mathbb{F}^0; \R^n) \times \mathbb{H}^2(\mathbb{F}^0; \R^\ell)$ satisfying the optimality condition:
\begin{equation}
    J(\hat{\alpha}, \hat{\mu}, \hat{\nu}) = \inf_{\alpha \in \mathbb{A}} J(\alpha, \hat{\mu}, \hat{\nu}),
\end{equation}
and the consistency conditions:
\begin{equation}
    \hat{\mu}_t = \mathbb{E}^1 \big[ \hat{X}_t  \big] \quad \text{and} \quad
    \hat{\nu}_t = \mathbb{E}^1 \big[ \hat{\alpha}_t \big], \quad \mathbb{P}^0\text{-a.s. for all } t \in [0,T],
\end{equation}
where $\hat{X}$ denotes the state process corresponding to the control $\hat{\alpha}$.

The following is the main assumption for the LQ-MFG:
\begin{assump}
\label{assump:linear-quadratic}
The coefficients of the LQ-MFG satisfy the following conditions:
\begin{enumerate}[label=(\roman*)]
    \item The functions 
    \begin{align*}
        &A, C, C^0,Q, \bar{Q} , S \colon [0,T] \to \R^{n \times n}, \\
        &B, \bar{B} \colon [0,T] \to \R^{n \times \ell}\\
        &D, \bar{D}, D^0, \bar{D}^0 \colon [0,T] \to \R^{n \times d \times \ell } , \,\text{and} \\
        &P, \bar{P}\colon [0,T] \to \R^{\ell \times \ell}
    \end{align*}
    are measurable and uniformly bounded in $t$.

    \item For all $t \in [0,T]$, the matrices satisfy 
    \[P_t, Q_t > 0, \quad \bar{Q}_tS_t \leq 0, \quad Q_t + \bar{Q}_t \geq \lambda I_\ell\] 
    for some constant $\lambda > 0$. 
    Here, $I_{\ell}$ denotes the $\ell \times \ell$ identity matrix, and the symbols $>$ and $\geq$ denote the standard partial ordering for square matrices.

    \item $c_1 = 0$, or ($c_1 \neq 0$ and $c_2 / c_1 \leq 1$).
\end{enumerate}
\end{assump}
\begin{remark}
The structural condition on the state process and Assumption \ref{assump:linear-quadratic} are sufficient to ensure that the game satisfies the monotonicity condition in Assumption \ref{assump:displacement_monotonicity}.

Unlike \cite{graber2016linear}, this setting allows the state process to depend on the distribution of the control. The constants $c_1$ and $c_2$ characterize the nature of this interaction. For instance, when $c_1 = c_2$, the player's control acts relative to the random environment $\nu$; effectively, the player controls the difference $\alpha_t - \nu_t$.When $c_2 = - c_1$, the control and the random environment affect the state dynamics in exactly the same manner.The case $c_2 = 0$ corresponds to a standard single-player optimization problem.

Note that in the cost functional $J$, the terms $\alpha$ and $\nu$ are separated to align with the separability condition in Assumption \ref{assump:standard_assump} \ref{assump:decomposition}. 
The separability condition is required primarily for the convergence result in this setting.
Thus, if we were concerned solely with the well-posedness of this LQ-MFE, we could consider a running cost term of the form $(\alpha - R_t\nu_t)^\top \bar{P}_t (\alpha - R_t\nu_t)$.
\end{remark}

To align with the notation in the previous sections, we introduce the following definitions before proving Proposition~\ref{prop:LQ-wellposedness}. 

For $(t, x, a, \mu, \nu) \in [0,T] \times \R^n \times \R^\ell \times \R^n \times \R^\ell$, the coefficients of the state process and the cost functions are defined as follows:
\begin{align*}
    b^{LQ}_t( x , a, \nu) &= A_t x + B_t (c_1 a - c_2 \nu), \\
    \sigma^{LQ}_t( x , a , \nu) &= C_t x + D_t (c_1 a - c_2 \nu), \\
    \sigma^{LQ,0}_t( x , a, \nu) &= C^0_t x + D^0_t (c_1 a - c_2 \nu), \\
    f^{LQ}_t(x, a, \mu, \nu) &= \frac{1}{2} \Big( x^\top Q_t x + a^\top P_t a 
    + (x - S_t \mu)^\top \bar{Q}_t (x - S_t \mu) + \nu^\top \bar{P}_t \nu \Big), \\
    g^{LQ}(x, \mu) &= \frac{1}{2} \Big( x^\top Q_T x + (x - S_T \mu)^\top \bar{Q}_T (x - S_T \mu) \Big).
\end{align*}
The Hamiltonian is defined as in \eqref{eq:def_Hamiltonian}:
\begin{equation*}
    \begin{split}
        H^{LQ}_t(x,y,z,z^0,a,\mu,\nu) &= f^{LQ}_t(x, a, \mu, \nu) + b^{LQ}_t(x , a, \nu) \cdot y 
        + \sigma^{LQ}_t( x , a , \nu) \cdot z + \sigma^{LQ,0}_t( x , a, \nu) \cdot z^0.
    \end{split}
\end{equation*}
The first-order derivatives with respect to $x$ and $a$ are computed as:
\begin{align*}
    \partial_x H^{LQ} &= Q_t x + \bar{Q}_t (x - S_t \mu) + A_t^\top y + C_t^\top z + (C^0_t)^\top z^0, \\
    \partial_a H^{LQ} &= P_t a + c_1 \left( B_t^\top y + D_t^\top z + (D^0_t)^\top z^0 \right).
\end{align*}
By the first-order condition $\partial_a H^{LQ} = 0$, the optimizer $\Lambda^{LQ}$ of the Hamiltonian $H^{LQ}$ (see Lemma~\ref{lemma:optimizer_of_Hamiltonian}) is given by
\begin{equation*}
    \Lambda^{LQ}_t(y,z,z^0) = - c_1 P^{-1}_t \left( B_t^\top y + D_t^\top z + (D^0_t)^\top z^0 \right).
\end{equation*}
Let us now define the coefficients of the Hamiltonian system. For $\theta = (x,y,z,z^0) \in \R^n \times \R^n \times \R^{n\times d} \times \R^{n\times d}$ and $\xi \in \mathcal{P}_2 (\R^n \times \R^n \times \R^{n\times d} \times \R^{n\times d})$, we define:
\begin{align*}
    B^{LQ}_t(\theta, \xi) 
    &:= A_t x - c_1^2 B_t P_t^{-1} \left( B_t^\top y + D_t^\top z + (D^0_t)^\top z^0 \right) 
        + c_1 c_2 B_t P_t^{-1} \int \left( B_t^\top v + D_t^\top w + (D^0_t)^\top w^0 \right) \, d\xi(u,v,w,w^0), \\[1ex]
    \Sigma^{LQ}_t(\theta, \xi) 
    &:= C_t x - c_1^2 D_t P_t^{-1} \left( B_t^\top y + D_t^\top z + (D^0_t)^\top z^0 \right)
        + c_1 c_2 D_t P_t^{-1} \int \left( B_t^\top v + D_t^\top w + (D^0_t)^\top w^0 \right)\, d\xi(u,v,w,w^0), \\[1ex]
    \Sigma^{0,LQ}_t(\theta, \xi) 
    &:= C^0_t x - c_1^2 D^0_t P_t^{-1} \left( B_t^\top y + D_t^\top z + (D^0_t)^\top z^0 \right) 
        + c_1 c_2 D^0_t P_t^{-1} \int \left( B_t^\top v + D_t^\top w + (D^0_t)^\top w^0 \right) \, d\xi(u,v,w,w^0), \\[1ex]
    F^{LQ}_t(\theta, \xi) 
    &:= Q_t x + \bar{Q}_t \left(x - S_t \int u \, d\xi(u,v,w,w^0)\right) 
        + A_t^\top y + C_t^\top z + (C^0_t)^\top z^0, \\[1ex]
    G^{LQ}(x, \xi) 
    &:= Q_T x + \bar{Q}_T \left(x - S_T \int u \, d\xi(u,v,w,w^0)\right).
\end{align*}

We are now ready to state the main proposition of this section.
\begin{proposition}
\label{prop:LQ-wellposedness}
Let Assumption~\ref{assump:linear-quadratic} be in force. Then:
\begin{enumerate}[label=(\roman*)]
    \item The coefficients $(b^{LQ}, \sigma^{LQ}, \sigma^{LQ, 0}, f^{LQ}, g^{LQ})$ and the Hamiltonian $H^{LQ}$ satisfy Assumptions~\ref{assump:standard_assump} and \ref{assump:Hamiltonian}.
    \item The coefficients $(B^{LQ}, \Sigma^{LQ}, \Sigma^{LQ, 0}, F^{LQ}, G^{LQ})$ satisfy Assumptions~\ref{assump:displacement_monotonicity} and \ref{assump:smoothness_L-derivative}.
\end{enumerate}
\end{proposition}

\begin{remark}
Thanks to Proposition~\ref{prop:LQ-wellposedness}, Theorem~\ref{thm:unique_existence_MFE} and Theorem~\ref{thm:wellposedness_conditional_MKVFBSDE} guarantee the unique existence of the LQ-MFE, and Theorem~\ref{thm:main} establishes the $O(N^{-1})$ convergence result.

Typically, as seen in \cite[Theorem~3.2]{bensoussan2016linear} and \cite[(2.23)]{bensoussan2025linear}, the unique existence of an LQ-MFE reduces to the well-posedness of corresponding FBODEs, which is derived by taking the expectation of the FBSDE obtained via the SMP.
Here, in contrast, by imposing specific structural conditions, we directly prove the well-posedness of the FBSDE.

To obtain a closed-loop equilibrium, one might employ an affine ansatz, such as $Y_t = K_t X_t + \phi_t$, for the solution of the FBSDE system. This would lead to a Riccati equation, reducing the problem to the solvability of that equation.
However, since such an endeavor lies beyond the scope of this paper, we focus here on establishing the unique existence of the open-loop control.
\end{remark}

\begin{proof}[Proof of Proposition~\ref{prop:LQ-wellposedness}]
We verify that the coefficients $(B^{LQ}, \Sigma^{LQ}, \Sigma^{LQ, 0}, F^{LQ}, G^{LQ})$ satisfy the monotonicity condition in Assumption~\ref{assump:displacement_monotonicity}.
Let $\Theta = (X, Y, Z, Z^0)$ and $\Theta' = (X', Y', Z', Z^{0\prime})$ be square-integrable random variables of appropriate dimensions. We denote the differences as:
\begin{equation*}
    \Delta X = X - X', \quad \Delta Y = Y - Y', \quad \Delta Z = Z - Z', \quad \Delta Z^0 = Z^0 - Z^{0\prime}.
\end{equation*}
First, we compute the term involving the terminal cost $G^{LQ}$:
\footnote{Here, in this proof, the notation $\langle x, y \rangle$ (resp. $\langle M, N \rangle$) denotes the standard inner product on $\R^k$  (resp. the Frobenius inner product defined by $\operatorname{trace}(M^\top N)$), for $x,y \in \R^k$ (resp. $M,N \in \R^{k \times k}$).}
\begin{align*}
    \mathbb{E}\left[ \left\langle \Delta X, G^{LQ}(X, \mathcal{L}(X)) - G^{LQ}(X', \mathcal{L}(X')) \right\rangle \right]  
    &= \mathbb{E}\left[ \langle \Delta X, (Q_T + \bar{Q}_T) \Delta X \rangle \right] 
    - \mathbb{E}\left[ \langle \Delta X, \bar{Q}_T S_T \mathbb{E}[\Delta X] \rangle \right] \\
    &\quad \geq \lambda \mathbb{E}\left[ |\Delta X|^2 \right],
\end{align*}
where we used the condition $Q_T + \bar{Q}_T \ge \lambda I$ and $\bar{Q}_T S_T \le 0$. 

Next, compute
\begin{align*}
    & \mathbb{E}\Big[ 
        - \langle \Delta X, F^{LQ}_t(\Theta, \mathcal{L}(\Theta)) - F^{LQ}_t(\Theta', \mathcal{L}(\Theta')) \rangle 
        + \langle \Delta Y, B^{LQ}_t(\Theta, \mathcal{L}(\Theta)) - B^{LQ}_t(\Theta', \mathcal{L}(\Theta')) \rangle \\
        &\qquad  + \langle \Delta Z, \Sigma^{LQ}_t(\Theta, \mathcal{L}(\Theta)) - \Sigma^{LQ}_t(\Theta', \mathcal{L}(\Theta')) \rangle 
        + \langle \Delta Z^0, \Sigma^{0,LQ}_t(\Theta, \mathcal{L}(\Theta)) - \Sigma^{0,LQ}_t(\Theta', \mathcal{L}(\Theta')) \rangle
    \Big] \\[1ex]
    %
    \leq & \mathbb{E} \Bigg[ 
        - \langle \Delta X, (Q_t + \bar{Q}_t) \Delta X \rangle 
        + \langle \mathbb{E}[\Delta X], \bar{Q}_t S_t \mathbb{E}[\Delta X] \rangle \\
        %
        &\quad 
            - \Big\langle \Delta X, A_t^\top \Delta Y + C_t^\top \Delta Z + (C^0_t)^\top \Delta Z^0 \Big\rangle
            + \Big\langle A_t \Delta X, \Delta Y \Big\rangle 
            + \Big\langle C_t \Delta X, \Delta Z \Big\rangle
            + \Big\langle C^0_t \Delta X, \Delta Z^0 \Big\rangle
         \\
        %
        &\quad - c^2_1 \Big| P_t^{-\frac{1}{2}} \Psi_t \Big|^2 
        + c_1 c_2 \Big\langle P_t^{-\frac{1}{2}} \mathbb{E}[\Psi_t], P_t^{-\frac{1}{2}} \Psi_t \Big\rangle
    \Bigg] \\
    & \quad \left( \text{where } \Psi_t := B_t^\top \Delta Y + D_t^\top \Delta Z + (D^0_t)^\top \Delta Z^0 \right) \\[1ex]
    %
    \leq & 
    -\mathbb{E} \left[ \langle \Delta X, (Q_t + \bar{Q}_t) \Delta X \rangle \right] 
    + \langle \mathbb{E}[\Delta X], \bar{Q}_t S_t \mathbb{E}[\Delta X] \rangle
    - \mathbb{E} \Bigg[ c^2_1 \Big| P_t^{-\frac{1}{2}} \Psi_t \Big|^2 - c_1 c_2 \Big| P_t^{-\frac{1}{2}} \mathbb{E} [\Psi_t] \Big|^2 \Bigg] \\[1ex]
    \leq & - \lambda \mathbb{E} \left[ |\Delta X|^2 \right],
\end{align*}
where the last inequality follows from the conditions $Q_t + \bar{Q}_t \ge \lambda I$ and $\bar{Q}_t S_t \leq 0$, the inequality $\mathbb{E} [|\boldsymbol{U}|^2 ] \geq |\mathbb{E} [\boldsymbol{U}]|^2$ for any random vector $\boldsymbol{U}$, and the assumptions on $c_1$ and $c_2$.
\end{proof}

\begin{appendix}
\section{Well-posedness of conditional McKean--Vlasov FBSDEs}
\label{sec:well_posedness_cond_MKVFBSDE}
In this section, we prove the well-posedness result for the following conditional McKean--Vlasov Equation:
\begin{align}
\left\{
\begin{aligned}
&d X_t = B_t( X_t, Y_t , Z_t, \mathcal{L}^1 (\Theta_t))\ dt 
+ \Sigma_t ( X_t, Y_t , Z_t,  \mathcal{L}^1 (\Theta_t)) \ d\overline{W}_t
\\
&dY_t =  -F_t ( X_t, Y_t , Z_t,  \mathcal{L}^1 (\Theta_t))  dt 
+ Z_t \ d\overline{W}_t,  \\
& X_0  = X^1_0, \hspace{1cm}Y_T = G(X_T, \mathcal{L}^1 (X_T)).
\end{aligned}
\right.
\label{eq:conditional_MKVFBSDE}
\end{align}
where $\overline{W} := (W, W^0)$ is a $(2\times d)$-dimensional Brownian Motion, $\Theta = (X, Y, Z)$, and $(\mathcal{L}^1(\Theta_t) )_{t\in[0,T]} $ is a $\mathbb{F}^0$-progressively measurable version.
Throughout this section, we only work with $(\Omega, \mathcal{F}, \mathbb{F},  \mathbb{P})$, and thus we use the shorthand notation $\mathbb{H}^2(\R^k) = \mathbb{H}^2(\mathbb{F}; \R^k)$.

Once we obtain the unique existence for Equation \eqref{eq:conditional_MKVFBSDE}, it is easy to prove the unique existence of the MFE as stated in Theorem \ref{thm:unique_existence_MFE}. Specifically, we only need to check the well-posedness of the system \eqref{eq:mean_field_MKVFBSDE}. The proof of the theorem is also provided at the end of this section.

\begin{assump}
\label{assump:well-posedness_conditional_MKVFBSDE}
Let 
\begin{align*}
&B, F : [0,T] \times \R^n \times \R^n \times \R^{2(n \times d)} \times \mathcal{P}_2(\R^n) \to \R^n, \\
&\Sigma : [0,T] \times \R^n \times \R^n \times \R^{2(n \times d)} \times \mathcal{P}_2(\R^n) \to \R^{n \times 2d}, \quad \text{and}\\
&G : \R^n \times \mathcal{P}_2(\R^n) \to \R^n
\end{align*}
be Borel measurable and satisfy:
\begin{enumerate}[label=(\roman*)]
  \item The functions $\varphi = B, \Sigma, F, G$ are Lipschitz continuous uniformly in t; that is, there exists a constant $L_f$ such that, for any $(x,y,z,\xi), (x',y',z',\xi') \in \R^n \times \R^n \times \R^{n \times 2d} \times \mathcal{P}_2(\R^{n\times n \times 2d})$, it holds
  \[
    |\varphi_t(x,y,z,\xi) - \varphi_t(x',y',z',\xi')| \leq L_f (|x - x'| + |y - y'| + |z - z'| + \mathcal{W}_2(\xi, \xi')).
  \]

  \item \label{assump:monotonicity_for_well_posedness}
    There exists $C_f> 0$  such that for any square integrable random variables $\Theta = (X,Y,Z), \Theta' = (X',Y',Z')$ of appropriate dimensions,
  \[
      \begin{aligned}
      \mathbb{E}\Big[ \, 
        -\ &\Delta X \cdot \left( F_t(\Theta, \mathcal{L}^1(\Theta)) - F_t(\Theta, \mathcal{L}^1(\Theta')) \right)\\
        +\  &\Delta Y \cdot \left( B_t(\Theta, \mathcal{L}^1(\Theta)) - B_t(\Theta', \mathcal{L}^1(\Theta')) \right) \\
        +\ &\Delta Z \cdot \left( \Sigma_t(\Theta, \mathcal{L}^1(\Theta)) - \Sigma_t(\Theta', \mathcal{L}^1(\Theta')) \right) 
      \, \Big] 
      \leq -C_f \mathbb{E}[|X - X'|^2],
      \end{aligned}
  \]
  and
  \[
      \mathbb{E} \left[ \Delta X \cdot \left( G(X, \mathcal{L}^1(X_T)) - G(X', \mathcal{L}^1(X_T')) \right) \right] 
      \geq C_f \mathbb{E}[|X - X'|^2] . 
  \]
\end{enumerate}
\end{assump}

Now we are ready to state our well-posedness result:
\begin{theorem}\label{thm:wellposedness_conditional_MKVFBSDE}
Let Assumption \ref{assump:well-posedness_conditional_MKVFBSDE} hold. Then, there exists a unique solution $(X,Y,Z) \in \mathbb{H}^2(\R^n) \times \mathbb{H}^2(\R^n) \times \mathbb{H}^2(\R^{2(n\times d)})$ to the equation \eqref{eq:conditional_MKVFBSDE}.
\end{theorem}

\begin{remark}
Although the above theorem only guarantees that $X$ (resp. $Y$) is in $\mathbb{H}^2(\R^n)$ and not necessarily in $\mathbb{S}^2(\R^n)$, if the coefficients $B,\Sigma$ (resp. $F$) additionally satisfy the linear growth condition uniformly in $t$, it is straightforward to check that $\E \sup_{t\in[0,T]} |X_t|^2 < \infty $ (resp. $\E \sup_{t\in[0,T]} |Y_t|^2 < \infty $), and thus $X \in \mathbb{S}^2(\R^n)$ (resp. $Y \in \mathbb{S}^2(\R^n)$), as in Lemma~\ref{lemma:SMP_MFG}. 
\end{remark}

\begin{proof}[Proof of Uniqueness of Theorem \ref{assump:well-posedness_conditional_MKVFBSDE}]
Let $\Theta$ and $\Theta'$ be two solutions, and 
denote for $\varphi = B, \Sigma, F$,
\begin{align*}
\Delta \Theta &:= \Theta - \Theta', \\
\Delta \varphi_t &:= \varphi_t(\Theta_t,\mathcal{L}^1(\Theta_t)) - \varphi_t(\Theta_t', \mathcal{L}^1(\Theta_t')), \\
\Delta G &:= G(X_T,\mathcal{L}^1(X_T)) - G(X_T', \mathcal{L}^1(X_T')).
\end{align*}
Then,
\[
\begin{cases}
\Delta X_t = \int_0^t \Delta B_s \, ds + \int_0^t \Delta \Sigma_s \, d\bar{W}_s, \\
\Delta Y_t = \Delta G + \int_t^T \Delta F_s \, ds - \int_t^T \Delta Z_s \, d\bar{W}_s.
\end{cases}
\]
Applying Itô's formula on $\Delta X_t \Delta Y_t$, we have
\begin{align*}
d(\Delta X_t \Delta Y_t) 
&= \Delta X_t \, d\Delta Y_t + \Delta Y_t \, d\Delta X_t + \Delta \Sigma_t \Delta Z_t \, dt \\
&= [-\Delta F_t \Delta X_t + \Delta B_t \Delta Y_t + \Delta \Sigma_t \Delta Z_t] dt 
+ [\Delta X_t \Delta Z_t + \Delta \Sigma_t \Delta Y_t] d\bar{W}_t.
\end{align*}
Note that $\Delta X_0 = 0$ and $\Delta Y_T = \Delta G$. 
Since we know from the standard estimates for SDEs and BSDEs that $\E \left[\sup_{t\in[0,T]} |\Delta X_t|^2 + \sup_{t\in[0,T]}|\Delta Y_t|^2  \right] < \infty$,
\begin{align*}
\mathbb{E}[\Delta G \Delta X_T] 
&= \mathbb{E}[\Delta Y_T \Delta X_T - \Delta Y_0 \Delta X_0] \\
&= \mathbb{E} \left[ \int_0^T [-\Delta F_t \Delta X_t + \Delta B_t \Delta Y_t + \Delta \Sigma_t \Delta Z_t] dt \right].
\end{align*}
By Assumption \ref{assump:well-posedness_conditional_MKVFBSDE}, we get
\[
0 \leq - C_f \mathbb{E} \left[ |\Delta X_T|^2 +  \int_0^T |\Delta X_t|^2  dt \right].
\]
Thus, from the standard estimates for BSDEs, it is straightforward to obtain $\E \left[ \int_0^T |\Delta Y_t|^2 + |\Delta Z_t|^2 \, dt \right] = 0$.
\end{proof}

Next, we prove the existence of a solution via the method of continuation. 
This method was developed for classical FBSDEs in \cite{hu1995solution,peng1999fully}, for MKV-FBSDEs in \cite{bensoussan2015well}, and for conditional MKV-FBSDEs in \cite{jackson2024quantitative} (where the flow of conditional probability is given by $(\mathcal{L}^1(X_t))_{t \in [0,T]}$, in contrast to $(\mathcal{L}^1(\Theta_t))_{t \in [0,T]}$ considered here).

First given some $\delta>0$,  consider the following system  for some $(b^0, \sigma^0, f^0, g^0)$:
\begin{equation}
\label{eq:FBSDE_delta}
\left\{
\begin{aligned}
&X_t = \xi + \int_0^t \left[ B_s^\delta(\Theta_s, \mathcal{L}^1(\Theta_s)) + b_s^0 \right] ds + \int_0^t \left[ \Sigma_s^\delta(\Theta_s, \mathcal{L}^1(\Theta_s)) + \sigma_s^0 \right] d\overline{W}_s, \\
&Y_t = G^\delta(X_T, \mathcal{L}^1(X_T)) + g^0 + \int_t^T \left[ F_s^\delta(\Theta_s, \mathcal{L}^1(\Theta_s)) + f_s^0 \right] ds - \int_t^T Z_s d\overline{W}_s
\end{aligned}
\right.
\end{equation}
where we defined the functions
\begin{align*}
B_t^\delta(\theta,\xi) &:= \delta B_t(\theta, \xi) - (1 - \delta)y, \\
\Sigma_t^\delta(\theta, \xi) &:= \delta \Sigma_t(\theta,\xi) - (1 - \delta)z, \\
F_t^\delta(\theta,\xi) &:= \delta F_t(\theta,\xi) + (1 - \delta)x, \\
G^\delta(x,\mu) &:= \delta G(x,\mu) + (1 - \delta)x,
\end{align*}
for 
$
\theta = (x, y, z)  \in \R^n \times \R^n \times \R^{2\times(n \times d)}
$
and 
$ \xi \in  \mathcal{P}_2(\R^n\times \R^n \times \R^{2\times(n\times d)})$.

We denote by FBSDE($\delta$) the class of FBSDEs defined above. We say that FBSDE($\delta$) is uniquely solvable if the equation \eqref{eq:FBSDE_delta} admits a unique solution for any $b^0, f^0 \in \mathbb{H}^2(\R^n), \sigma^0 \in \mathbb{H}^2(\R^{2(n \times d)})$, and $g^0 \in L^2(\Omega, \mathcal{F}_T, \mathbb{P};\R)$.
Note that FBSDE(1) corresponds exactly to \eqref{eq:conditional_MKVFBSDE}. The following lemma is essential to connect the solvability of FBSDE(0) and FBSDE(1).

\begin{lemma}\label{lemma:FBSDE(delta)_is_solvable}
Let Assumption \ref{assump:well-posedness_conditional_MKVFBSDE} hold. If FBSDE($\delta_0$) is uniquely solvable, then there exists a constant $\eta >0$, depending only on $L_f, C_f$ in Assumption~\ref{assump:well-posedness_conditional_MKVFBSDE}, such that FBSDE($\delta$) is uniquely solvable for any $\delta \in [\delta_0, \delta_0 + \eta]$.
\end{lemma}

With this lemma, we can now prove Theorem~\ref{thm:wellposedness_conditional_MKVFBSDE}.

\begin{proof}[Proof of Theorem~\ref{thm:wellposedness_conditional_MKVFBSDE}]
It is known that FBSDE(0) is solvable (see \cite[Lemma 8.4.3]{zhang2017backward}). Since the step size $\eta$ depends only on the structural constants and is independent of the initial point $\delta_0$, we can iteratively apply Lemma~\ref{lemma:FBSDE(delta)_is_solvable} starting from $\delta=0$ to cover the interval $[0,1]$ in finitely many steps. This implies that FBSDE(1) is solvable.
\end{proof}

Now, let us prove Lemma~\ref{lemma:FBSDE(delta)_is_solvable}. We adapt the argument of \cite[Appendix~A]{jackson2024quantitative} to our setting, where the coefficients depend on the law $\mathcal{L}^1(\Theta_s)$, not only its first marginal $\mathcal{L}^1(X_s)$.
\begin{proof}[Proof of Lemma~\ref{lemma:FBSDE(delta)_is_solvable}]
For any \( \delta \in [\delta_0, \delta_0 + \eta] \) where \( \eta > 0 \) is to be determined later, denote \( \varepsilon := \delta - \delta_0 \). For arbitrary $b^0, f^0 \in \mathbb{H}^2(\R^n), \sigma^0 \in \mathbb{H}^2(\R^{2(n \times d)})$, and $g^0 \in L^2(\Omega, \mathcal{F}_T, \mathbb{P};\R)$, consider the mapping 
$
\Psi : \mathbb{H}^2(\R^n) \times \mathbb{H}^2(\R^n) \times \mathbb{H}^2(\R^{2 \times (n \times d)}) 
\to \mathbb{H}^2(\R^n) \times \mathbb{H}^2(\R^n) \times \mathbb{H}^2(\R^{2 \times (n \times d)})
$
such that
$
\Psi(x, y, z) = (X, Y, Z)
$
with  $(X, Y, Z) $ satisfying the following FBSDE:
\begin{align*}
\left\{
    \begin{aligned}
        X_t &= X_0 + \int_0^t B_s^{\delta_0}(\Theta_s, \mathcal{L}^1(\Theta_s)) + \varepsilon \left[ y_s + B_s(\vartheta_s, \mathcal{L}^1(\vartheta_s)) \right] + b^0_s \,  ds \\
         &\quad + \int_0^t \Sigma_s^{\delta_0}(\Theta_s, \mathcal{L}^1(\Theta_s)) + \varepsilon \left[ z_s + \Sigma_s(\vartheta_s, \mathcal{L}^1(\vartheta_s)) \right] + \sigma^0_s \, d\overline{W}_s, \\
        Y_t &= G^{\delta_0}(X_T, \mathcal{L}^1(X_T)) + \varepsilon \left[ -x_T + G(x_T, \mathcal{L}^1(x_T)) \right] + g^0 \\
        &\quad + \int_t^T F_s^{\delta_0}(\Theta_s, \mathcal{L}^1(\Theta_s)) + \varepsilon \left[ -x_s + F_s(\vartheta_s, \mathcal{L}^1(\vartheta_s)) \right] ds + f^0_s \, ds \\
        &\quad - \int_t^T Z_s d\overline{W}_s.
    \end{aligned}
\right.
\end{align*}
By our assumption, FBDSE($\delta_0$) is uniquely solvable and then, $\Phi$ is well-defined.
It suffices to prove that $\Psi$ is a contraction mapping for a suitable choice of $\eta$. 
Let $(x, y, z), (x', y', z') \in \mathbb{H}^2(\R^n) \times \mathbb{H}^2(\R^n) \times \mathbb{H}^2(\R^{2 \times (n \times d)})$ 
and $(X, Y, Z) = \Psi(x, y, z)$ and  $(X', Y', Z') = \Psi(x', y', z')$. Denote $\Delta \vartheta := \vartheta - \vartheta', \Delta \Theta := \Theta - \Theta'$.
Then,
\begin{align*}
    d \Delta X_t \cdot \Delta Y_t 
    &= \Delta X_t \cdot \bigg[ 
        - \left( F_t^{\delta_0}(\Theta_t, \mathcal{L}^1(\Theta_t)) - F_t^{\delta_0}(\Theta_t', \mathcal{L}^1(\Theta_t')) \right) \\
    &\qquad\qquad\quad 
        + \varepsilon \left( \Delta x_t - F_t(\vartheta_t, \mathcal{L}^1(\vartheta_t)) + F_t(\vartheta_t', \mathcal{L}^1(\vartheta_t')) \right) 
    \bigg] dt \\
    &\quad + \Delta Y_t \cdot \bigg[ 
        B_t^{\delta_0}(\Theta_t,  \mathcal{L}^1(\Theta_t)) - B_t^{\delta_0}(\Theta_t',  \mathcal{L}^1(\Theta_t')) \\
    &\qquad\qquad\quad
        + \varepsilon \left( \Delta y_t + B_t(\vartheta_t, \mathcal{L}^1(\vartheta_t)) - B_t(\vartheta_t', \mathcal{L}^1(\vartheta_t')) \right) 
    \bigg] dt \\
    &\quad + \Delta Z_t \cdot \bigg[ 
        \Sigma_t^{\delta_0}(\Theta_t, \mathcal{L}^1(\Theta_t)) - \Sigma_t^{\delta_0}(\Theta_t', \mathcal{L}^1(\Theta_t')) \\
    &\qquad\qquad\quad
        + \varepsilon \left( \Delta z_t + \Sigma_t(\vartheta_t, \mathcal{L}^1(\vartheta_t)) - \Sigma_t(\vartheta_t', \mathcal{L}^1(\vartheta_t')) \right) 
    \bigg] dt + dM_t    
\end{align*}
for a martingale $M$.
Since the original coefficients $B, F, \Sigma$, and $G$ satisfy Assumption~\ref{assump:well-posedness_conditional_MKVFBSDE}, the interpolated coefficients $B^{\delta}, F^{\delta}, \Sigma^{\delta}$, and $G^{\delta}$ naturally inherit these properties. 
Consequently, they also fulfill Assumption \ref{assump:well-posedness_conditional_MKVFBSDE}, with the constant $C_f$ replaced by
\begin{align}
\label{eq:new_constant_for_displacement_monotonicity}
C_{\delta} := \delta C_f + 1 - \delta \geq \min(c_f, 1) =: \bar{C}_f.      
\end{align}

Since $\Delta X_0 = 0$, we have
\begin{equation*}
    \begin{split}
        \mathbb{E}[\Delta X_T \cdot \Delta Y_T] 
        + C_\delta \mathbb{E} \left[ \int_0^T |\Delta X_t|^2 \, dt \right] 
        &\leq \varepsilon \mathbb{E} \Bigg[ \int_0^T 
            - \Delta X_t \cdot \left( F_t(\vartheta_t, \mathcal{L}^1(\vartheta_t)) - F_t(\vartheta_t', \mathcal{L}^1(\vartheta_t')) \right) \\
        &\qquad \qquad  + \Delta Y_t \cdot \left( B_t(\vartheta_t, \mathcal{L}^1(\vartheta_t)) - B_t(\vartheta_t', \mathcal{L}^1(\vartheta_t')) \right) \\
        &\qquad \qquad  + \Delta Z_t \cdot \left( \Sigma_t(\vartheta_t, \mathcal{L}^1(\vartheta_t)) - \Sigma_t(\vartheta_t', \mathcal{L}^1(\vartheta_t')) \right) \\
        &\qquad \qquad  + \Delta X_t \cdot \Delta x_t + \Delta Y_t \cdot \Delta y_t + \Delta Z_t \cdot \Delta z_t \, dt \Bigg].
    \end{split}
\end{equation*}
Applying Young's inequality and the Lipschitz continuity of the coefficients $B, F, \Sigma$, and $G$, we obtain
\begin{align*}
\mathbb{E}[\Delta X_T \cdot \Delta Y_T] 
+ (C_\delta - \varepsilon) \mathbb{E} \left[ \int_0^T |\Delta X_s|^2 \, ds \right] 
\leq \varepsilon C \mathbb{E} \left[ \int_0^T \left( |\Delta \Theta_s|^2 + |\Delta \vartheta_s|^2 \right) ds \right].
\end{align*}
for a generic constant $C > 0$
\footnote{In the proof, $C$ may vary line by line, but it depends only on $C_f, L_f, T$ and not on $\delta_0$.}.
Here, we have also utilized the estimate:
\begin{align*}
\mathbb{E}[\mathcal{W}_2^2(\mathcal{L}^1(\Theta_t), \mathcal{L}^1(\Theta_t'))] 
&=
\mathbb{E}^0\left[\mathcal{W}_2^2(\mathcal{L}^1(\Theta_t), \mathcal{L}^1(\Theta_t'))\right]\\
&\leq
\mathbb{E}^0\left[\mathbb{E}^1 \left[|\Theta_t- \Theta_t'|^2\right] \right]\\
&\leq
\mathbb{E} \left[|\Theta_t- \Theta_t'|^2\right].
\end{align*}

Also, utilizing the fact that $G^{\delta_0}$ fulfills Assumption \ref{assump:well-posedness_conditional_MKVFBSDE} (with the constant $C_f$ substituted by $C_{\delta^0}$), we derive the lower bound:
\begin{equation*}
    \begin{split}
        \mathbb{E}[\Delta X_T \cdot \Delta Y_T] 
        &\geq C_\delta \mathbb{E}[|\Delta X_T|^2] \\
        &\quad + \varepsilon \mathbb{E} \left[ 
            - \Delta X_T \cdot \Delta x_T 
            + \Delta X_T \cdot \left( 
                G(x_T, \mathcal{L}^1(x_T)) 
                - G(x_T', \mathcal{L}^1(x_T')) 
            \right) 
        \right].
    \end{split}
\end{equation*}
Combining these two estimates thus yields
\begin{equation}
\label{eq:estimate_X_wellposedness_MKVFBSDE}
    \begin{split}
        &(C_\delta - \varepsilon)\, \mathbb{E}[|\Delta X_T|^2] 
        + (C_\delta - \varepsilon)\, \mathbb{E} \left[ \int_0^T |\Delta X_s|^2 \, ds \right] \\
        &\quad \leq \varepsilon C \mathbb{E} \left[ \int_0^T \left( |\Delta \Theta_s|^2 + |\Delta \vartheta_s|^2 \right) ds \right]
        + \varepsilon C \mathbb{E}[|\Delta x_T|^2]. 
    \end{split}
\end{equation}
Note that the $L_f$-Lipschitz property of $B, F, \Sigma$, and $G$, combined with $\delta_0 \leq 1$, implies that the interpolated functions $B^{\delta_0}, F^{\delta_0}, \Sigma^{\delta_0}$, and $G^{\delta_0}$ are $(L_f + 2)$-Lipschitz continuous. 
For notational simplicity, we shall continue to denote this generic Lipschitz constant by $L_f$. 
Now, applying It\^o's formula to the process $e^{\kappa t} |\Delta Y_t|^2$ with a parameter $\kappa > 0$ to be chosen later, we obtain
\begin{equation*}
    \begin{split}
        e^{\kappa t} \mathbb{E}[|\Delta Y_t|^2] 
        &+ \mathbb{E} \left[ \int_t^T e^{\kappa s} |\Delta Z_s|^2 ds \right] \\
        &\leq C e^{\kappa T} \mathbb{E}[|\Delta X_T|^2] + \varepsilon^2 C e^{\kappa T} \mathbb{E}[|\Delta x_T|^2] \\
        &\quad + \mathbb{E} \Bigg[ \int_t^T e^{\kappa s} \left( \Bigg( \frac{4L_f^2 + 1}{a} - \kappa \right) |\Delta Y_s|^2 
        + 2a \varepsilon e^{\kappa s} (|\Delta X_s|^2 + |\Delta Z_s|^2) + 3a \varepsilon e^{\kappa s} |\Delta \vartheta_s|^2 \Bigg) ds \Bigg],
    \end{split}
\end{equation*}
for any $a > 0$.
Taking $a$ sufficiently small and $\kappa$ large, we have
\begin{equation*}
    \begin{split}
        \mathbb{E} \left[ \int_0^T |\Delta Y_t|^2 \, dt \right] 
        + \mathbb{E} \left[ \int_0^T |\Delta Z_t|^2 \, dt \right] 
        &\leq C \mathbb{E}[|\Delta X_T|^2] 
        + \varepsilon^2 C \mathbb{E}[|\Delta x_T|^2] \\
        &\quad + a C \mathbb{E} \left[ \int_0^T |\Delta X_t|^2 \, dt \right] 
        + a C \mathbb{E} \left[ \int_0^T |\Delta \vartheta_t|^2 \, dt \right].
    \end{split}
\end{equation*}
Plug this into \eqref{eq:estimate_X_wellposedness_MKVFBSDE}, we get
\begin{equation*}
    \begin{split}
        (C_{\delta_0} - \varepsilon) \mathbb{E}[|\Delta X_T|^2] 
        &+ (C_{\delta_0} - \varepsilon) \mathbb{E} \left[ \int_0^T |\Delta X_s|^2 ds \right] 
        + \mathbb{E} \left[ \int_0^T \left( |\Delta Y_s|^2 + |\Delta Z_s|^2 \right) ds \right] \\
        &\leq (\varepsilon + a) C \mathbb{E} \left[ \int_0^T |\Delta \Theta_t|^2 dt \right] 
        + \varepsilon C \mathbb{E} \left[ \int_0^T |\Delta \vartheta_t|^2 dt \right] \\
        &\quad + C \mathbb{E}[|\Delta X_T|^2] 
        + \varepsilon C \mathbb{E}[|\Delta x_T|^2].
    \end{split}
\end{equation*}
Combining the estimate for $\mathbb{E}[|\Delta X_T|^2]$ provided by \eqref{eq:estimate_X_wellposedness_MKVFBSDE} with the property that $C_{\delta_0} \geq \bar{C}_f$ (see \eqref{eq:new_constant_for_displacement_monotonicity}), we deduce:
\begin{equation*}
    \begin{split}
        (\bar{C}_f - \varepsilon) \mathbb{E}[|\Delta X_T|^2] 
        &+ \bigg[ (\bar{C}_f - \varepsilon) - (2\varepsilon + a) C \bigg] 
        \mathbb{E} \left[ \int_0^T |\Delta X_t|^2 \, dt \right] \\
        &\quad + \bigg[ 1 - (2\varepsilon + a) C \bigg] 
        \mathbb{E} \left[ \int_0^T \left( |\Delta Y_t|^2 + |\Delta Z_t|^2 \right) \, dt \right] \\
        &\leq \varepsilon C \mathbb{E} \left[ \int_0^T |\Delta \vartheta_t|^2 \, dt \right] 
        + \varepsilon C \mathbb{E}[|\Delta x_T|^2].
    \end{split}
\end{equation*}
Given that $\bar{C}_f > 0$, we can choose parameters $a$ and $\varepsilon$ sufficiently small to ensure positivity of the coefficients, yielding the estimate:
\begin{align*}
\mathbb{E}[|\Delta X_T|^2] 
&+ \mathbb{E} \left[ \int_0^T |\Delta \Theta_s|^2 \, ds \right] 
\leq \varepsilon C \left\{ 
\mathbb{E} \left[ \int_0^T |\Delta \vartheta_t|^2 \, dt \right] 
+ \mathbb{E}[|\Delta x_T|^2] 
\right\}.
\end{align*}
This inequality demonstrates that $\Psi$ is a contraction mapping provided $\varepsilon$ is small enough. 
Consequently, we may define $\eta$ as the maximum possible value for $\varepsilon$. Importantly, since the constant $C$ is independent of $\delta_0$, the step size $\eta$ is also independent of $\delta_0$.
\end{proof}

Given this well-posedness result, we are ready to prove the unique existence of the MFE. 
\begin{proof}[Proof of Theorem \ref{thm:unique_existence_MFE}]
We use the same notation as \eqref{eq:from_b_to_B} and \eqref{eq:change_space_of_measure}.
To guarantee the unique existence of the solution to (\ref{eq:mean_field_MKVFBSDE}), we first need to check that the coefficients $B, \Sigma, \Sigma^0$, and $F$ are Lipschitz continuous uniformly in $t$. In fact, since we already know that $b,\sigma,\sigma^0, \partial_{x}H$ and $\partial_{x}g$ are Lipschitz by assumption, we only need to show that
\begin{equation}
\label{eq:wasserstein_estimate_for_wellposeness}
\mathcal{W}_2 (\varphi_t( \xi), \varphi_t( \xi'))
\leq 
C ( \mathcal{W}_2 (\xi, \xi') +   \mathcal{W}_2 (\mu, \mu')),
\end{equation}
for all $\xi, \xi' \in \mathcal{P}_2(\R^n \times \R^n \times \R^{n\times d} \times \R^{n\times d} )$.
To prove this, the argument in the proof of \cite[Theorem~1]{lauriere2022convergence} extends to our setting without difficulty.
We begin by applying the Kantorovich duality theorem, which yields the following representation:
\begin{align*}
&\mathcal{W}_2^2(\varphi_t( \xi), \varphi_t( \xi')) \\
&= \sup \left( \int_{\R^n \times \R^{\ell}} h_1(x, a) \varphi_t( \xi)(dx, da)
- \int_{\R^k \times \R^m} h_2(x', a') \varphi_t( \xi')(dx', da') \right) \\
&= \sup \Big( \int_{\R^n\times \R^n \times \R^{n\times d} \times \R^{n\times d}} 
        h_1(\operatorname{id}_x(\theta), \Lambda_t( \theta, \mu, 0)) \xi(d\theta)  \\
    &\qquad \qquad - \int_{\R^n\times \R^n \times \R^{n\times d} \times \R^{n\times d}} 
        h_2(\operatorname{id}_x(\theta'), \Lambda_t( \theta', \mu',0)) \xi'(d\theta') \Big), 
\end{align*}
Here, the supremum is taken over the class of bounded continuous functions $ h_1, h_2 : \R^n \times \R^{\ell} \to \R $ subject to the constraint $ h_1(x, a) - h_2(x', a') \leq |x - x'|^2 + |a - a'|^2 $ for all $ (x, a), (x', a') \in \R^n \times \R^{\ell} $.
Given that $ \Lambda $ is Lipschitz continuous, it follows that,
\begin{align*}
h_1(x, \Lambda_t(\theta, \mu,0)) - h_2(x', \Lambda_t(\theta', \mu',0)) 
&\leq |x - x'|^2 + |\Lambda_t(\theta, \mu)) - \Lambda_t( \theta', \mu')|^2 \\
&\leq C \left( |\theta - \theta'|^2  + \mathcal{W}_2^2(\mu, \mu') \right).
\end{align*}
This yields the estimate
\begin{align*}
\mathcal{W}_2^2(\varphi_t( \xi), \varphi_t( \xi')) 
&\leq \sup \Big( \int_{\R^n\times \R^n \times \R^{n\times d} \times \R^{n\times d}}
    \tilde{h}_1(\theta) \xi(d\theta) 
    - \int_{\R^n\times \R^n \times \R^{n\times d} \times \R^{n\times d}} 
    \tilde{h}_2(\theta')\xi'(d\theta')\Big),
\end{align*}
where the supremum is taken over all functions $ \tilde{h}_1, \tilde{h}_2 $ subject to 
\[ \tilde{h}_1(\theta) - \tilde{h}_2(\theta') \leq C \left( |\theta - \theta'|^2 + \mathcal{W}_2^2(\mu, \mu') \right).\]
Take an arbitrary transport plan $\Pi \in  \mathcal{P}_2((
\R^n\times \R^n\times \R^{n\times d}\times \R^{n\times d })^2)$ with its first marginal $\xi$ and the second one $\xi'$.
Again, using the duality representation of the Wasserstein distance,
\begin{align*}
\mathcal{W}_2^2(\varphi_t( \xi), \varphi_t( \xi')) 
&\leq \sup \Big( \int_{(\R^n\times \R^n \times \R^{n\times d} \times \R^{n\times d})^2}
\tilde{h}_1(\theta)  
-\tilde{h}_2(\theta') \,
d\Pi \Big)\\
&\leq C  \int_{(\R^n\times \R^n \times \R^{n\times d} \times \R^{n\times d})^2}
|\theta - \theta'|^2   \,
d\Pi  
+ C\mathcal{W}_2^2(\mu, \mu') .
\end{align*}
Since the choice of coupling $\Pi$ is arbitrary, minimizing the right-hand side over all admissible $\Pi$ gives
\begin{align*}
\mathcal{W}_2^2(\varphi_t( \xi), \varphi_t( \xi')) 
&\leq C \inf \int_{(\R^n\times \R^n \times \R^{n\times d} \times \R^{n\times d})^2}
|\theta - \theta'|^2   \,
d\Pi      + C\mathcal{W}_2^2(\mu, \mu') \\
&= C(\mathcal{W}_2(\xi, \xi') + \mathcal{W}_2^2(\mu, \mu') ) . 
\end{align*}
Consequently, we conclude that $ B,\Sigma, \Sigma^0, F $, and $ G $ are Lipschitz continuous.

Regarding the integrability condition of $(\mathcal{L}^1 (X_t, \alpha_t))_{t\in[0,T]}$, it is clear that the solution $\Theta$ of the conditional MKV-FBSDE \eqref{eq:mean_field_MKVFBSDE} satisfies
\begin{align*}
\E \left[
\int^T_0 M^2_2(\mathcal{L}^1 (X_t, \alpha_t)) \, dt 
\right] 
=
\E^0  \left[ \int^T_0 \E^1 \left[
 |X_t|^2+  |\alpha_t|^2
\right] 
 \, dt 
\right]
< \infty,  
\end{align*}
for $\alpha_t := \Lambda_t (\Theta_t, \mu, 0)$,
the proof is complete.
\end{proof}

\section{Construction of conditionally i.i.d.\ copies of solutions to the McKean--Vlasov FBSDE}
\label{sec:coupling_technique}
In this section, we construct a conditionally independent sequence of solutions \(\Theta^i = (X^i, Y^i, Z^i, Z^{0,i})\) to the following conditional MKV-FBSDEs on \((\Omega^N, \mathcal{F}^N, \mathbb{P}^N)\).
Each \(\Theta^i\) is identically distributed to the solution \(\Theta = (X, Y, Z, Z^0)\) defined on \((\Omega, \mathcal{F}, \mathbb{P})\).
\begin{equation}
\label{eq:conditional_MKVFBSDE_for_i}
\left\{
\begin{aligned}
&d X_t^i = B_t\left( \Theta^i_t, \mathcal{L}^1 (\Theta_t^i)\right)\, dt 
+ \Sigma_t \left( \Theta^i_t,  \mathcal{L}^1 (\Theta_t^i)\right) \, dW^i_t 
+ \Sigma^0_t \left( \Theta^i_t,  \mathcal{L}^1 (\Theta_t^i)\right) \, dW^0_t ,
\\
&dY_t^i =  -F_t\left( \Theta^i_t,  \mathcal{L}^1 (\Theta_t^i)\right) \, dt 
+ Z_t^i \, dW^i_t + Z^{0,i}_t \, dW^0_t ,
\\
& X_0^i  \sim \mu_0, \hspace{1cm} Y_T^i = G\left(X_T^i, \mathcal{L}^1 (X^i_T)\right),
\end{aligned}
\right.
\end{equation}
where we recall \eqref{eq:from_b_to_B}
\begin{equation}
\begin{aligned}
B_t(\theta,\xi)
 &= b_t\big(x, \Lambda_t(\theta, \mu, 0), \varphi_t(\xi)\big), \\
\Sigma_t(\theta,\xi)
 &= \sigma_t \big(x, \Lambda_t(\theta, \mu, 0), \varphi_t(\xi)\big), \\
\Sigma^0_t(\theta,\xi)
 &= \sigma^0_t\big( x, \Lambda_t(\theta, \mu, 0), \varphi_t(\xi)\big), \\
F_t(\theta,\xi)
 &= \partial_x H_t\big(x, \Lambda_t(\theta, \mu, 0), \varphi_t(\xi)\big), \\
G(x,\mu)
 &= \partial_x g(x,\mu).
\end{aligned}
\end{equation}
for a function $\varphi : [0,T] \times \mathcal{P}_2(
\R^n\times \R^n\times \R^{n\times d}\times \R^{n\times d } 
)
\to \mathcal{P}_2 (\R^n \times \R^{\ell} )$
\begin{equation}
\label{eq:change_space_measure}
\varphi_t(\xi) := \xi \circ \left( \mathrm{id}_{x}, \Lambda_t(\cdot, \cdot, \cdot, \cdot, \mu, 0) \right)^{-1},
\end{equation}
where $\mathrm{id}_{x}$ denotes the projection $(x,y,z,z^0)\mapsto x$. 
From \eqref{eq:wasserstein_estimate_for_wellposeness}, the map \(\xi \mapsto \varphi_t( \xi)\) is continuous and, consequently, Borel measurable.

Suppose that for \(i = 1\), the above FBSDE admits a solution \(\Theta\) which is \(\mathbb{F}\)-progressively measurable on \((\Omega, \mathcal{F}, \mathbb{P})\).  
Then, by \cite[Proposition~1.2.1~(i)]{zhang2017backward}, there exists a \(\big(\sigma(W^0_s, X_0, W_s \, ; \, 0\leq s \leq t) \big)_{t \in [0, T]}\)-progressively measurable version of \(\Theta\), which we again denote by \(\Theta\).  
From now on in this section, we will work with this version of \(\Theta\).
Thus, for each \(t\), by the measurability of \(\Theta_t\) with respect to \(\sigma(W^0_s, X_0, W_s \, ; \, 0\leq s \leq t)\), we have
\[
\Theta_t (\omega^0, x, \omega) = \Theta_t (W^0, X_0, W) (\omega^0, x, \omega) = \Theta_t (W_{\cdot \wedge t}^0, X_0, W_{\cdot \wedge t}) (\omega^0, x, \omega) = \Theta_t(\omega_{\cdot \wedge t}^0, x, \omega_{\cdot \wedge t}),
\]
for any \((\omega^0, x, \omega) \in \mathcal{C} \times \R^n \times \mathcal{C}\).

Furthermore, since \(\Theta\) is \(\big(\sigma(W^0_s, X_0, W_s \, ; \, 0\leq s \leq t) \big)_{t \in [0, T]}\)-progressively measurable, the function
\[
(t, \omega^0, x, \omega) \mapsto \Theta_t(\omega^0, x, \omega)
\]
is \(\mathcal{B}([0,T]) \otimes \sigma(W^0, X_0, W)\)-measurable, and hence Borel measurable on \([0,T] \times \mathcal{C} \times \R^n \times \mathcal{C}\).

Now, we define a process \(\Theta^i\) on \((\Omega^N, \mathcal{F}^N, \mathbb{P}^N)\) by
\[
\Theta_t^i(\omega^0, \mathbf{x}, \boldsymbol{\omega}) := \Theta_t (W^0, X_0^i, W^i)(\omega^0, \mathbf{x}, \boldsymbol{\omega}) = \Theta_t (W_{\cdot \wedge t}^0, X_0^i, W_{\cdot \wedge t}^i)(\omega^0, \mathbf{x}, \boldsymbol{\omega}),
\]
where we regard \(\Theta\) as a function from \(\mathcal{C} \times \R^n \times \mathcal{C}\) to \(\R^n \times \R^n \times \R^{n\times d} \times \R^{n\times d}\).
It is clear that \(\Theta^i\) is \(\mathbb{F}^N\)- progressively measurable. 


Furthermore, since \(\Theta \in \mathbb{H}^2(\mathbb{F};\R^n) \times \mathbb{H}^2(\mathbb{F};\R^n) \times \mathbb{H}^2(\mathbb{F};\R^{n \times d}) \times \mathbb{H}^2(\mathbb{F};\R^{n \times d})\), it follows that \(\Theta^i \in \mathbb{H}^2(\mathbb{F}^N;\R^n) \times \mathbb{H}^2(\mathbb{F}^N;\R^n) \times \mathbb{H}^2(\mathbb{F}^N;\R^{n \times d}) \times \mathbb{H}^2(\mathbb{F}^N;\R^{n \times d})\).

By this construction and Lemma~\ref{lemma:process_conditional_dist}, the sequence of processes \((\Theta^i)_{i=1,\ldots,N}\) is conditionally i.i.d.\ in the sense defined in Section~\ref{sec:main_theorem}, with the common conditional distribution given by \(\mathcal{L}^1(\Theta_t)\). In addition, the distribution of $\Theta^i : \Omega^N \to L^2([0,T])$ is identical for all $i=1,\ldots,N$.

In the subsequent lemmas, we show that each \(\Theta^i\) is a solution to equation~\eqref{eq:conditional_MKVFBSDE_for_i} for \(i = 1, \ldots, N\).  
We first state the following lemma, which will be used in the sequel. The proof is inspired by \cite[Lemma~10.1, p.~125]{rogers1987processes}.
\begin{lemma}
\label{lemma:rep_of_stoch_int}
Let $\varphi \in \mathbb{H}^2 (\mathbb{F}; \R^k)$.
Then, there exists a Borel measurable function $\Phi: \mathcal{C} \times \R^n \times  \mathcal{C} \to  C([0,T];\R^k)$, such that
\begin{equation}
\label{eq:existence_Phi}
\int^\cdot_0 \varphi_t \, dW_t = \Phi(W^0, X^0, W), \,\, \text{$\mathbb{P}$-a.s.},
\end{equation}
or more precisely, it holds that
\[
\int^t_0 \varphi_s(W^0, X_0, W)(w^0, x, w)\, dW_s = \Phi_t(W^0, X_0, W)(w^0, x, w), \,\, \text{for any $t\in[0,T]$}, 
\]
for $\mathbb{P}$-almost all $(\omega^0,x, \omega)\in \Omega$.

Moreover, on $(\Omega^N, \mathcal{F}^N, \mathbb{P}^N)$, it holds that 
\[
\int^\cdot_0 \varphi_t(W^0, X_0^i, W^i) \, dW^i_t = \Phi(W^0, X_0^i, W^i), \,\, \text{$\mathbb{P}^N$-a.s.},
\]
\end{lemma}
\begin{proof}
First, suppose that \(\varphi\) is a simple predictable process; that is, there exists a finite partition \(0 = t_0 < t_1 < \dots < t_M = T\), and \(\mathcal{F}_{t_m}\)-measurable, bounded, \(\R^{k \times d}\)-valued random variables \(\varphi_m\) for \(m = 0, \ldots, M-1\), such that \(\varphi\) is of the form
\[
\varphi_t = \varphi_0 \1_{\{0\}}(t) + \sum_{m=1}^{M-1} \varphi_m \1_{(t_m, t_{m+1}]}(t).
\]
Then, we choose a version of each \(\varphi_m\) that is measurable with respect to the the raw filtration \(\sigma(W^0_s, X_0, W_s \, ; \, s \leq t_m)\).  Note that for this version of \(\varphi_m\), the map \((\omega^0, x, \omega) \mapsto \varphi_m(\omega^0, x, \omega)\) is Borel measurable on \(\mathcal{C} \times \R^n \times \mathcal{C}\).

It then follows from the definition of the stochastic integral that
\[
\int_0^\cdot \varphi_t \, dW_t 
= \sum_{m=0}^{M-1} \varphi_m(W^0_{\cdot \wedge t_m}, X_0, W_{\cdot \wedge t_m}) \big( W_{\cdot \wedge t_{m+1}} - W_{\cdot \wedge t_m} \big), \quad \mathbb{P}\text{-a.s.}
\]

Since for any \(t \in [0,T]\), the map \(\mathcal{C} \ni \omega \mapsto \omega_{\cdot \wedge t} \in \mathcal{C}\) is continuous with respect to the supremum norm on \(\mathcal{C}\), it is measurable with respect to the Borel \(\sigma\)-algebra on \(\mathcal{C}\).  
Therefore, together with the measurability of \(\varphi_m\), we deduce that there exists a measurable map \(\Phi\) satisfying \eqref{eq:existence_Phi}.

For a general process \(\varphi \in \mathbb{H}^2(\mathbb{F}; \R^k)\), there exists a sequence of simple predictable processes \((\varphi^m)_{m \in \mathbb{N}}\) such that
\[
\lim_{m \rightarrow \infty} \mathbb{E} \left[ \int_0^T \left| \varphi_t - \varphi_t^m \right|^2 dt \right] = 0.
\]
Then, for the sequence of functions \((\Phi^m)_{m \in \mathbb{N}}\), satisfying  
\[
\Phi_\cdot^m(W^0, X_0, W) = \int_0^\cdot \varphi_t^m \, dW_t, \quad \text{$\mathbb{P}$-a.s.} \quad \text{for all } m \in \mathbb{N},
\]
we define
\[
\Phi_t(W^0, X_0, W) :=     
\begin{cases}
    \lim_{m \rightarrow \infty} \Phi_t^m(W^0, X_0, W), & \text{if the limit exists},\\
    0, & \text{otherwise}.
\end{cases}
\]
Since the sequence \((\Phi_T^m(W^0, X_0, W))_{m \in \mathbb{N}}\) converges in \(L^2(\Omega, \mathcal{F}, \mathbb{P}; \R^k)\), the processes \(\Phi_t^m(W^0, X_0, W)\) converge to \(\Phi_t(W^0, X_0, W)\) uniformly in \(t \in [0,T]\), \(\mathbb{P}\)-a.s.  
Therefore, \(\Phi_{\cdot}(W^0, X_0, W) \in \mathcal{C}\), \(\mathbb{P}\)-a.s.  
It then follows from the definition of the stochastic integral that \eqref{eq:existence_Phi} holds.

The final statement also holds if we take the same approximating sequence.
\end{proof}

\begin{remark}
\label{rem:rep_also_holds_for_W0}
From the above proof, a similar result holds if we replace \(W\) in \eqref{eq:existence_Phi} by \(W^0\).  
This fact will also be used in the next lemma.
\end{remark}

We are finally ready to state the following coupling result. The proof is similar to \cite[Theorem 10.4, p.126]{rogers1987processes}.
\begin{lemma} 
\label{lemma:iid_copies_solves_MKVFBSDE}
Assume that the conditional MKV-FBSDE \eqref{eq:conditional_MKVFBSDE_for_i} admits a solution $\Theta  \in \mathbb{H}^2(\mathbb{F};\R^n) \times \mathbb{H}^2(\mathbb{F};\R^n) \times \mathbb{H}^2(\mathbb{F};\R^{n\times d}) \times \mathbb{H}^2(\mathbb{F};\R^{n\times d})$ for $i=1$ on the probability space $(\Omega, \mathcal{F}, \mathbb{P})$, and that Assumption \ref{assump:standard_assump} \ref{assump:differentiability},\ref{assump:Lipschitz} hold. 
Let $(\Theta^i)_{i=1,\ldots N}$ be the conditionally i.i.d.\ sequence of $\mathbb{F}^N$-progressively measurable version of the stochastic processes as constructed above. 

Then, for each \(i\), the process \(\Theta^i \in \mathbb{H}^2(\mathbb{F}^N; \R^n) \times \mathbb{H}^2(\mathbb{F}^N; \R^n) \times \mathbb{H}^2(\mathbb{F}^N; \R^{n \times d}) \times \mathbb{H}^2(\mathbb{F}^N; \R^{n \times d})\) is a solution to equation~\eqref{eq:conditional_MKVFBSDE_for_i} on the probability space \((\Omega^N, \mathcal{F}^N, \mathbb{P}^N)\).
\end{lemma}
\begin{proof}
Define
\[
\varphi^B: [0,T] \times \Omega \ni (t, \omega^0, x, \omega) \mapsto B_t\left(\Theta_t, \mathcal{L}^1(\Theta_t)\right)(\omega^0, x, \omega) \in \R^d .
\]
Then, by Lemma~\ref{lemma:process_conditional_dist}, the process \(\varphi^B\) is \(\mathbb{F}^N\)-progressively measurable.  
Hence, we can select a \(\big(\sigma(W^0_s, X_0, W_s \, ; \, s \leq t)\big)_{t \in [0,T]}\)-progressively measurable version of \(\varphi^B\).
Thus, if we define
\begin{align*}
\Phi^B_t(\omega^0, x, \omega) = \Phi^B_t (W^0, X_0, W)(\omega^0, x, \omega) &:=  \int^t_0 \varphi^B_s (\omega^0, x, \omega) \,ds\\
&= \int^t_0 \varphi^B_s (W^0, X_0, W)(\omega^0, x, \omega) \,ds,
\end{align*}
then, by Fubini's theorem, the map \((\omega^0, x, \omega) \mapsto \Phi^B_t(\omega^0, x, \omega)\) is Borel measurable on \(\mathcal{C} \times \R^n \times \mathcal{C}\).
Moreover, from previous Lemma \ref{lemma:rep_of_stoch_int} and the subsequent Remark \ref{rem:rep_also_holds_for_W0}, there exist Borel measurable functions $\Phi^{\Sigma}, \Phi^{\Sigma^0}, \Phi^{Z},$ and $\Phi^{Z^0}$ from $\mathcal{C} \times \R^n \times  \mathcal{C}$ to  $C([0,T];\R^{n\times d})$, such that the following equalities hold \(\mathbb{P}\)-a.s.:
\begin{align*}
\Phi_{\cdot}^{\Sigma}(W^0, X_0, W) &=  \int^\cdot_0 \Sigma(t, \Theta_t, \mathcal{L}^1(\Theta_t))\, dW_t ,\\
\Phi_{\cdot}^{\Sigma^0}(W^0, X_0, W)  &=  \int^\cdot_0 \Sigma^0(t, \Theta_t, \mathcal{L}^1(\Theta_t))\, dW^0_t ,\\
\Phi_{\cdot}^{Z}(W^0, X_0, W) &= \int^\cdot_0 Z_t\, dW_t ,\\
\Phi_{\cdot}^{Z^0}(W^0, X_0, W) &= \int^\cdot_0 Z^0_t \, dW^0_t.
\end{align*}
In addition, we can find a Borel measurable map $\Phi^G$ such that 
$\Phi^{G}(W^0, X_0, W) = G(X_T, \mathcal{L}(X_T))$, $\mathbb{P}$-a.s.

Next, define
\begin{align*}
&\Phi_t^X(\omega^0, x, \omega) := \textup{id}_x \circ \Theta_t(\omega^0, x, \omega), \quad 
\Phi_t^Y(\omega^0, x, \omega) := \textup{id}_y \circ \Theta_t(\omega^0, x, \omega), \\
&\Phi_t^Z(\omega^0, x, \omega) := \textup{id}_z \circ \Theta_t(\omega^0, x, \omega), \quad
\Phi_t^{Z^0}(\omega^0, x, \omega) := \textup{id}_{z^0} \circ \Theta_t(\omega^0, x, \omega), 
\end{align*}
where for any $(x,y,z,z^0)\in \R^n \times \R^n \times \R^{n\times d} \times \R^{n\times d}$,
\[
\textup{id}_x(x,y,z,z^0)   := x, \quad 
\textup{id}_y(x,y,z,z^0)   := y, \quad
\textup{id}_z(x,y,z,z^0)     := z, \,\, \text{and} \,\,\,\,
\textup{id}_{z^0}(x,y,z,z^0) := z^0,
\]
are projections. Note also that for each $t \in [0,T]$, \(\Phi_t^X\), \(\Phi_t^Y\), \(\Phi_t^Z\), and \(\Phi_t^{Z^0}\) are Borel measurable maps from \(\mathcal{C} \times \R^n \times \mathcal{C}\) into the corresponding Euclidean spaces, each equipped with its Borel \(\sigma\)-algebra.

Therefore, from equation~\eqref{eq:conditional_MKVFBSDE_for_i} with \(i = 1\), it follows that for any \(t \in [0,T]\), the following equalities hold \(\mathbb{P}\)-a.s.:
\begin{align*}
\left\{
\begin{aligned}
\Phi^X_t(W^0, X_0, W) &= X_0 + \Phi^B_t(W^0, X_0, W) + \Phi^{\Sigma}_t(W^0, X_0, W) + \Phi^{\Sigma^0}_t(W^0, X_0, W), \\
\Phi^Y_t(W^0, X_0, W) &= \Phi^G(W^0, X_0, W) - \Phi^Z_t(W^0, X_0, W) - \Phi^{Z^0}_t(W^0, X_0, W).
\end{aligned}
\right.
\end{align*}
Since the law of \((W^0, X_0, W)\) and \((W^0, X_0^i, W^i)\) is the same, it follows that for any \(t \in [0, T]\), the following equalities hold \(\mathbb{P}^N\)-a.s.:
\begin{align*}
\left\{
\begin{aligned}
\Phi^X_t(W^0, X_0^i, W^i) &= X_0^i + \Phi^B_t(W^0, X_0^i, W^i) + \Phi^{\Sigma}_t(W^0, X_0^i, W^i) + \Phi^{\Sigma^0}_t(W^0, X_0^i, W^i), \\
\Phi^Y_t(W^0, X_0^i, W^i) &= \Phi^G_t(W^0, X_0^i, W^i) - \Phi^Z_t(W^0, X_0^i, W^i) - \Phi^{Z^0}_t(W^0, X_0^i, W^i).
\end{aligned}
\right.
\end{align*}
Moreover, using the estimate~\eqref{eq:wasserstein_estimate_for_wellposeness}, the Lipschitz continuity of \(b\) and \(\Lambda\), and the fact that for any \(t \in [0, T]\), \(\mathcal{L}^1(\Theta_t) = \mathcal{L}^1(\Theta^i_t)\), \(\mathbb{P}^0\)-a.s. for all $t\in[0,T]$, we obtain
\begin{align*}
&\mathbb{E} \left[
\int_0^T \left| B_t\left( \Theta_t^i, \mathcal{L}^1(\Theta_t) \right) - B_t\left( \Theta_t^i, \mathcal{L}^1(\Theta^i_t) \right) \right|^2 dt 
\right] \\
&\quad \leq C \int_0^T \mathbb{E}^0 \left[ 
\mathcal{W}_2^2\left( \mathcal{L}^1(X_t), \mathcal{L}^1(X_t^i) \right) 
+ \mathcal{W}_2^2\left( \mathcal{L}^1(\Theta_t), \mathcal{L}^1(\Theta^i_t) \right)
\right] dt \\
&\quad = 0.
\end{align*}

A similar \(dt \otimes \mathbb{P}^N\)-equivalence also holds for \(\Sigma\), \(\Sigma^0\), \(F\), and \(G\), using the Lipschitz continuity of \(\sigma\), \(\sigma^0\), \(\partial_x H\), \(\partial_x g\), and \(\Lambda\).  
Therefore, we conclude that \(\Theta^i\) satisfies equation~\eqref{eq:conditional_MKVFBSDE_for_i} for all \(i = 1, \ldots, N\).
\end{proof}

\begin{remark}
\label{rem:non_canonical}
It might be possible to work within a general (non-canonical) probability setup by employing a type of Yamada–Watanabe theorem for FBSDEs (see \cite[Theorem 1.33]{carmona2018probabilistic_II}). Simply put, if the strong uniqueness holds for \eqref{eq:conditional_MKVFBSDE_for_i}, with $i=1$, then $(X,Y,\int (Z_s, Z_s^0)\,ds)$ can be represented via a function $\Phi$ depending on the $X, W$, and $W^0$. Accordingly, by defining $(X^i,Y^i,\int (Z^i_s, Z_s^{0,i})\,ds):= \Phi(W^0, X_0^i, W^i)$,  $(X^i,Y^i,\int (Z^i_s, Z_s^{0,i})\,ds)_{i\in \{1,\ldots,N\}}$ is conditionally i.i.d.\ and solves \eqref{eq:conditional_MKVFBSDE_for_i}. However, the result is stated only for $\int^\cdot_0 (Z_s, Z_s^0) ds$, not for $Z$ and $Z^0$ themselves. In the proof of Theorem \ref{thm:main}, we specifically need $(Z^i, Z^{0,i})_{i\in\{1,\ldots,N\}}$ to be conditionally i.i.d.\ and since the space accommodating these processes are less regular than that of  $\int^\cdot_0 (Z_s, Z_s^0) ds$, the argument becomes more delicate. This is why we have chosen to work in the canonical setup, where establishing conditional i.i.d.-ness is more straightforward.
\end{remark}

\end{appendix}

\bibliographystyle{plain}
\bibliography{mybibfile}

\end{document}